\documentclass[a4paper,10pt]{article}
\usepackage{graphicx}
\usepackage{rotating}

\title{Multiscale Finite Element approach for "weakly" random problems
  and related issues} 
\author{Claude Le Bris, Fr\'ed\'eric Legoll, Florian Thomines
\\ \\
{\footnotesize 
{\'E}cole Nationale des Ponts et Chauss{\'e}es, 6 et 8
    avenue Blaise Pascal, 77455 Marne-La-Vall{\'e}e Cedex 2}
\\ 
{\footnotesize 
and INRIA Rocquencourt, MICMAC team-project, 
Domaine de Voluceau, B.P. 105,}
\\ 
{\footnotesize
78153 Le Chesnay Cedex, France}
\\
{\footnotesize 
lebris@cermics.enpc.fr, \{legoll,thominef\}@lami.enpc.fr}
}
\date{\today}

\usepackage{amsfonts}
\usepackage{amsmath,amssymb,amsthm}
\usepackage[english]{babel}

\newcommand{\dps}{\displaystyle}
\newcommand{\esp}{\mathbb{E}}
\newcommand{\var}{\mathbb{V}\textrm{ar}}
\newcommand{\cell}{\mathbf{K}}
\newcommand{\RR}{\mathbb R}
\newcommand{\ZZ}{\mathbb Z}
\newcommand{\PP}{\mathbb P}
\newcommand{\NN}{\mathbb N}
\newcommand{\eps}{\varepsilon}
\newcommand{\wper}{w^0}

\newcommand{\h}{{H^1_h}}
\newcommand{\m}{\overline m}
\newcommand{\Rop}{{\cal R}^\eps}

\newtheorem{theorem}{Theorem}

\newtheorem{lemme}[theorem]{Lemma}
\newtheorem{remark}[theorem]{Remark}

\usepackage{setspace}

\begin{document}
\selectlanguage{english}
\maketitle

\begin{abstract}
We address multiscale elliptic problems with random coefficients that
are a perturbation of multiscale deterministic problems.
Our approach consists in taking benefit of the perturbative context to
suitably modify the classical Finite Element basis into a deterministic
multiscale Finite Element basis. The latter essentially shares the same
approximation properties as a multiscale Finite Element basis directly
generated on the random problem. The specific reference method that we
use is the Multiscale Finite Element Method. Using numerical
experiments, we demonstrate the efficiency of our approach and the
computational speed-up with respect to a more standard approach. We
provide a complete analysis of the approach, extending that available
for the deterministic setting. 
\end{abstract}

\section{Overview of our approach and results}

The Multiscale Finite Element Method (henceforth abbreviated as MsFEM)
is a popular 
numerical approach for multiscale problems
(see~\cite{hou1997,hou1999,Efendiev2009,Efendiev2000,Efendiev2004,Allaire2005,hou2004,Chen2002,Chen2006,bal2010}).
It consists in a Galerkin approximation of the original problem over
a finite dimensional space generated by basis functions that are
specifically {\em adapted} to the problem under consideration.

This approach is popular for a twofold reason. First, its use is not
restricted to multiscale problems that converge to a
homogenized problem in the limit of vanishing ratio between the small
scale and the macroscopic scale. It may be applied to much more general
situations. Second, when the problem does converge to a
homogenization problem, the MsFEM approach is meant to approximate the
solution of the problem with the small scale $\varepsilon$ at its actual
small value and not "only" in the asymptotic regime $\varepsilon \to 0$,
which is the regime addressed by homogenization theory. 

To fix the ideas, consider the problem of finding $u^\varepsilon$ solving
\begin{equation}
\label{0}
 -\mbox{div}\left[A^\varepsilon \nabla
u^{\varepsilon} \right] = f \ \ \mbox{in $\mathcal{D}$}, 
\quad
u^\varepsilon = 0 \ \ \mbox{on $\partial \mathcal{D}$},
\end{equation}
on a bounded domain $\mathcal{D} \subset \RR^d$, with $f \in
L^2(\mathcal{D})$, and where $A^\varepsilon$ is a uniformly bounded,
coercive matrix that varies at scale
$\varepsilon$. A standard Finite Element Method (FEM) would require a
space discretization of the domain at the scale $\varepsilon$ in order
to capture the oscillations of $u^\varepsilon$ at scale $\eps$. This is
prohibitively expensive. The MsFEM
aims at accurately approximating $u^\varepsilon$ using a limited number
of degrees of freedom. It does not require the matrix $A^\varepsilon$ to
be periodic (namely $A^\varepsilon(x) = A_{per}(x/\varepsilon)$ for
a fixed periodic matrix $A_{per}$) or
stationary. 

\medskip

We now briefly describe the approach and present
the aim of this article. Starting from a coarse mesh ${\cal T}_h$ with a
standard (say $\mathbb{P}_1$) Finite Element basis set of functions 
$\left\{ \phi_i^0 \right\}_{i=1}^L$, generating the associated space
$$
{\cal V}_h:=\text{span}(\phi_i^0,i=1,\cdots,L),
$$ 
we first numerically build the MsFEM basis functions
$\phi_i^\varepsilon$. Several definitions of these basis functions have
been proposed in the literature (yielding different numerical
methods), and we detail this in the sequel (see e.g.~\eqref{PB:MsFEM-Over-nb}-\eqref{eq:MsFEM-Over-nb-1}-\eqref{eq:MsFEM-Over-nb-2}). For the moment, it is
sufficient to know that, to each $\phi_i^0$, which varies at the macroscopic
scale, is associated a function $\phi_i^\varepsilon$, with variations at
the scale~$\varepsilon$. In practice, $\phi_i^\varepsilon$ is
{\em numerically} computed (in fact, {\em pre-computed}), using the
specificities of the problem addressed.
These highly oscillatory functions $\phi_i^\varepsilon$ generate the finite
dimensional space  
$$
{\cal W}_h:=\text{span}(\phi_i^\varepsilon,i=1,\cdots,L).
$$
Note that ${\cal W}_h$ and ${\cal V}_h$ share the same dimension. 

We next define the MsFEM solution $u_M$ using a Galerkin approximation
of~\eqref{0} on ${\cal W}_h$, instead of ${\cal V}_h$. Again, details
will be given below. The MsFEM solution $u_M$ provided by the approach reads
$$
u_M(x) = \sum\limits_{i=1}^L (U_M)_i \ \phi_i^\varepsilon(x),
$$
for some coefficients $\left\{ (U_M)_i \right\}_{i=1}^L$. Of course, 
these coefficients depend on $\varepsilon$, but this dependency is kept
implicit in the sequel.

\medskip

We now turn our attention to the stochastic problem
\begin{equation}
\label{0-sto}
 -\mbox{div}\left[A^\varepsilon(\cdot,\omega) \nabla
u^{\varepsilon}(\cdot,\omega) \right] = f \ \ \mbox{in $\mathcal{D}$}, 
\quad
u^\varepsilon(\cdot,\omega) = 0 \ \ \mbox{on $\partial \mathcal{D}$},
\end{equation}
and a typical quantity of interest $\esp \left[ u^{\varepsilon}(x,\cdot)
\right]$, which is traditionally approximated using
a Monte Carlo method. Introducing a set of ${\cal M}$ realizations of the
stochastic matrix $\left\{A^{\varepsilon,m}\right\}_{1\leq
  m\leq {\cal M}}$, a direct, na\"ive application of the MsFEM paradigm would
consist in first computing for each realization $m$ the stochastic MsFEM
basis functions $\phi_i^{\varepsilon,m}(x,\omega)$, next performing 
a Galerkin approximation of~\eqref{0-sto} using this MsFEM basis
set to compute $\left\{u_M^m(x,\omega)\right\}_{1\leq m\leq {\cal M}}$,
and eventually approximating $\esp \left[ u^{\varepsilon}(x,\cdot) \right]$ by
$$
\esp \left[ u^{\varepsilon}(x,\cdot) \right] \approx 
\frac{1}{\cal M} \sum_{m=1}^{\cal M} u_M^m(x,\omega).
$$
Such an approach is unpractical because of the prohibitively expensive computational load. 

\medskip

To reduce the computational cost and make the MsFEM approach practical
in such a stochastic context, a natural idea we investigate in this article is
to consider a less generic setting, for which a dedicated, more
computationally affordable approach, can be designed. 
One possibility is to consider matrices 
$A^\varepsilon(x,\omega) \equiv A^\varepsilon(x) + B(x,\omega)$
in~\eqref{0-sto} that are not highly
oscillatory in their stochastic part. 
In such cases, dedicated approaches have been proposed, we refer
to~\cite{Ginting2010}
for more details.
Another approach is to reduce the number of Monte-Carlo simulations used
for the computation of the multiscale basis
functions. In~\cite{Aarnes2009,Dostert2008}, the authors assume that their
coefficient can be written as a Karhunen-Lo\`eve type expansion,
and apply a collocation method to \emph{a priori} choose some sparse
realizations for which they compute the multiscale basis functions. 

In this article, we consider one of the many alternate variants of
problem~\eqref{0-sto}. We suppose that 
$A^\varepsilon(x,\omega)$ is highly oscillatory in
both its deterministic and stochastic components, but that it is a {\em
  perturbation} of a deterministic matrix. More precisely, we 
assume that
\begin{equation}
\label{eq:decompo}
A^\varepsilon(x,\omega) \equiv A_\eta^\varepsilon(x,\omega) 
=
A_0^\varepsilon(x) + \eta A_1^\varepsilon(x,\omega),
\end{equation}
where $A_0^\varepsilon$ is a deterministic matrix and $\eta$ is a small
deterministic parameter. This model may be well suited for heterogeneous
materials (or, more generally, media) that, although not periodic, are {\em not
fully} stochastic, in the sense that they may be considered as a
{\em perturbation} of a deterministic material. We call this setting the
{\em weakly stochastic setting}. Note that many practical situations,
involving actual materials or media, can be considered, at a good level
of approximation, as perturbations of a deterministic (often periodic)
setting (see e.g.~\cite{enumath}).

In a series of recent works (see~\cite{cras-stoc,jmpa,cras-ronan}
and~\cite{cras-arnaud,arnaud1,arnaud2}; see also~\cite{Bris-sg} for a
unified presentation), we have considered such a setting, in the
context of homogenization theory (the matrix $A_\eta^\varepsilon(x,\omega)$ 
in~\eqref{0-sto}-\eqref{eq:decompo} reads $A_\eta^\varepsilon(x,\omega) =
A_\eta(x/\varepsilon,\omega)$ for a stationary matrix $A_\eta(x,\omega)$, which
is, in a sense to be made precise, a perturbation of a periodic
matrix). We have 
shown there that, in such a case, the workload for computing
the homogenized solution is significantly lighter
than for {\em generic} stochastic homogenization, and actually
comparable to the workload for periodic homogenization. 
We will show in the sequel that the MsFEM can be adapted to
this weakly stochastic setting, providing an approximation of the solution
$u^\eps_\eta$ to~\eqref{0-sto}-\eqref{eq:decompo}, for fixed $\eps$, at a much
smaller computational cost than the direct approach.

The main idea of our proposed approach is to compute a set
of {\em deterministic} MsFEM basis functions using $A_0^\varepsilon$, the
deterministic part of $A_\eta^\varepsilon$ in the
expansion~\eqref{eq:decompo}, and then to perform Monte 
Carlo realizations at the macroscale level using a set of ${\cal M}$
realizations of the random matrix
$\left\{A^{\varepsilon,m}_\eta(x,\omega)\right\}_{1\leq m \leq {\cal M}}$ (see
Section~\ref{sec:msfem} for a detailed presentation). Note
that, for each of these realizations, we solve the {\em original} problem, with
the {\em complete} matrix $A^\varepsilon_\eta$, and not only its
deterministic part. Only the basis set is taken
deterministic. By construction, the approach provides an
approximation 
$$
u_S(x,\omega) = \sum\limits_{i=1}^L (U_S(\omega))_i \ \phi_i^\varepsilon(x)
$$
of $u^\varepsilon_\eta(x,\omega)$, where the basis functions $\phi_i^\varepsilon$ are {\em deterministic}. 
These basis functions are computed only once, hence the cost to
compute $\left\{u_S^m(x,\omega)\right\}_{1\leq m\leq {\cal M}}$ is much smaller
than the cost to compute $\left\{u_M^m(x,\omega)\right\}_{1\leq
  m\leq {\cal M}}$. This is especially true if~\eqref{0-sto} has to be solved
for many right-hand sides $f$. We
expect that this approximation $u_S$ is as accurate as 
$u_M$ for small $\eta$. We show below that this is indeed the case, when
$A^\varepsilon_\eta$ is a perturbation of $A^\varepsilon_0$ (see
Section~\ref{sec:num} for numerical tests).

\medskip

We would like to note that the MsFEM is not the only multiscale
technique based on finite elements. The bottom line of our approach,
consisting of generating suitable multiscale functions for the
discretization of a weakly stochastic problem, using for this purpose
the deterministic reference problem, can in principle be applied to
other multiscale techniques. Another popular technique is the HMM
method~\cite{hmm_init,hmm,hmm_review}, 
for which our approach could in principle be easily adapted. 

\medskip

In the numerical tests reported on in Section~\ref{sec:num}, we compare, in the
$H^1$ norm,
$u^\varepsilon_\eta$ (the exact solution to~\eqref{0-sto} with the
matrix $A^\varepsilon \equiv A^\varepsilon_\eta$ given by~\eqref{eq:decompo}) 
with $u_S$ (the solution provided by our approach) and $u_M$ (the
solution provided by the ideal, expensive approach). The quantity 
$\|u^\varepsilon_\eta-u_M\|_{H^1({\cal D})}$ somewhat represents the best possible
accuracy that we can achieve, in the sense that our approach inherits the
limitations of the MsFEM approach. We thus cannot expect our
approximation $u_S$ to be more accurate than $u_M$. We can only hope to compute an approximation of comparable quality with a much reduced workload. The numerical
results we obtain confirm that, for small $\eta$ in~\eqref{eq:decompo},
the quantity $\|u_S-u_\eta^\varepsilon\|_{H^1({\cal D})}$ is of the same
order of magnitude as $\|u_M - u_\eta^\varepsilon\|_{H^1({\cal D})}$,
although, we repeat it, the computational cost to compute $u_S$ is much
smaller than that to compute $u_M$.  

\medskip

We next derive error bounds for our approach in
Section~\ref{sec:analysis}. We recall that, in the
deterministic setting, a classical context for proving convergence of
the MsFEM approach is the case when, in the reference problem~\eqref{0},
the matrix reads 
$\dps A^\varepsilon(x) = A_{per}\left(\frac{x}{\eps}\right)$ for a fixed
periodic matrix $A_{per}$. Likewise, to be able to perform our
theoretical analysis in the stochastic setting, we assume in Section~\ref{sec:analysis} that
$\dps A_\eta^\varepsilon(x,\omega) = 
A_\eta \left(\frac{x}{\eps},\omega\right)$ for a fixed stationary
random matrix $A_\eta$. The problem~\eqref{0-sto}-\eqref{eq:decompo}
then admits a homogenized limit when $\eps$ vanishes. 

Our proof follows the same lines as that in the deterministic
setting, which we now briefly review (see the introduction of
Section~\ref{sec:analysis} for more details on the structure of the proof). 
The MsFEM is a Galerkin approximation, that we assume momentarily, for the sake
of clarity, to be a {\em conforming} approximation (this is indeed the
case when, for defining the highly oscillatory basis functions 
$\phi_i^\varepsilon$, oversampling is {\em not} used). The error
is then estimated using the C\'ea lemma:
$$
\|u^\eps - u_M \|_{H^1({\cal D})}
\leq C \inf\limits_{v_h \in \mathcal{W}_h} 
\|u^\eps - v_h\|_{H^1({\cal D})}, 
$$
where $u^\eps$ is the solution to the reference deterministic highly oscillatory
problem~\eqref{0}, $u_M$ is the MsFEM solution and $C$ is a constant
independent of $\eps$ and $h$. Taking advantage of the
homogenization setting, we introduce the two-scale expansion 
$$
v^\eps = u^\star + \varepsilon \sum_{i=1}^d 
\wper_{e_i} \left( \frac{\cdot}{\varepsilon} \right) \partial_i u^\star
$$
of $u^\eps$, where $u^\star$ is the homogenized solution,
$\wper_{e_i}$ is the periodic corrector associated to $e_i \in
\RR^d$, and $\partial_i u^\star$ denotes the partial derivative $\dps
\frac{\partial u^\star}{\partial x_i}$. We next write  
$$
\|u^\eps - u_M \|_{H^1({\cal D})}
\leq C \left( \|u^\eps - v^\eps \|_{H^1({\cal D})}
+
\inf\limits_{v_h \in \mathcal{W}_h} \|v^\eps - v_h\|_{H^1({\cal D})} \right).
$$
The first term in the
above right-hand side is estimated using standard homogenization
results on the {\em rate} of convergence of $v^\eps - u^\eps$. To estimate the second term, one considers a
suitably chosen element $v_h \in \mathcal{W}_h$, for which 
$\dps \|v^\eps - v_h\|_{H^1({\cal D})}$ can be directly bounded.

Following the same strategy in our stochastic setting, we estimate the distance between the solution $u_\eta^\varepsilon$ to the reference stochastic
problem~\eqref{0-sto}-\eqref{eq:decompo}
and the weakly stochastic MsFEM solution $u_S$ as
$$
\|u^\eps_\eta(\cdot,\omega) - u_S(\cdot,\omega) \|_{H^1({\cal D})}
\leq C \left( 
\|u^\eps_\eta(\cdot,\omega) - v^\eps_\eta(\cdot,\omega) \|_{H^1({\cal D})}
+
\inf\limits_{v_h \in \mathcal{W}_h} 
\|v^\eps_\eta(\cdot,\omega) - v_h\|_{H^1({\cal D})} \right).
$$
We observe that a key ingredient for
the proof is the rate of convergence of the difference between the reference
solution $u^\eps_\eta$ and its two-scale expansion $v^\eps_\eta$. Such a
result is classical in periodic homogenization, but, to the best of our
knowledge, open in
the {\em general} stationary case (in dimensions higher than
one). One only knows that $u^\eps_\eta - v^\eps_\eta$ vanishes (in some
appropriate norm) when $\eps \to 0$.  
However, in the particular case when $A_\eta^\varepsilon$ is only
{\em weakly} stochastic, it is possible to obtain such a result, as we have
shown in~\cite{rate}. Hence, exploiting the specificity of our
weakly stochastic setting, we are able to obtain (see our main result,
Theorem~\ref{theo:H1} and estimate~\eqref{est-H1}):
$$
\sqrt{ \esp \left[ \|u^\varepsilon_\eta - u_S\|^2_\h \right] } \leq C 
\left( \sqrt{\eps} + h + \frac{\varepsilon}{h} + \eta \left(
    \frac{\varepsilon}{h} \right)^{d/2} \ln(N(h)) 
+ \eta + \eta^2 \mathcal{C}(\eta) \right),
$$
where $\| \cdot \|_\h$ is a broken $H^1$ norm, $C$ is a constant
independent of $\varepsilon$, $h$ and $\eta$, 
$\mathcal{C}$ is a bounded function as $\eta$ goes to $0$, and $N(h)$ is
the number of elements in the mesh (roughly of order $h^{-d}$ in
dimension $d$). 
As is often the case in the deterministic setting, we
use here (both for our numerical tests and in the analysis) the
oversampling technique. Consequently, the basis functions 
$\phi_i^\eps$ do not belong to $H^1_0({\cal D})$, hence the use of a
broken $H^1$ norm in the above estimate. 
As we point out below, when $\eta=0$
in~\eqref{eq:decompo}, our approach reduces to the standard
deterministic MsFEM (with oversampling), and the above
estimates then agree with those proved in~\cite{Efendiev2000}. 

\bigskip

This article is organized as follows. First, in
Section~\ref{sec:msfem}, we describe the MsFEM approach. For consistency, we begin by
the deterministic setting in Section~\ref{sec:msfem_det}, and point out
there that the direct adaptation to the general 
stochastic setting yields a prohibitively expensive approach. The
adaptation of the approach to the {\em 
  weakly stochastic setting} is described in
Section~\ref{sec:weak-stoch}. We next turn to numerical simulations, in
Section~\ref{sec:num}. Some procedures to efficiently implement the
approach are first described in Section~\ref{sec:implem}. We next
consider a one-dimensional test (see Section~\ref{sec:test-1d}), which
is useful for several reasons. First, it allows to calibrate some
numerical parameters, such as the number $M$ of independent realizations
when estimating the exact expectation by an empirical mean. Second, we
assess the accuracy of our approach with respect to the magnitude of $\eta$. We
demonstrate there that $\eta$ does not have to be extremely small for
our method to be very efficient. For instance, on the test case
considered in Section~\ref{sec:test-1d}, we show that our approach is as
accurate as the expensive, direct approach as soon as $\eta$ is such that
$$
\left\|
\frac{\eta a^\varepsilon_1}{a_0^\varepsilon}\right\|_{L^\infty(\RR\times
\Omega)} 
\text{is equal to or smaller than 0.1},
$$
where $a_0^\varepsilon$ is the deterministic component of the diffusion
coefficient $a^\eps_\eta$ and $\eta a_1^\varepsilon$ is the stochastic
component (see expansion~\eqref{eq:decompo}). Lastly, we also assess 
the accuracy of our approach with respect to 
the presence of frequencies in the random coefficient $a^\eps_\eta$ that
are not taken into account in the MsFEM basis set. 
We next turn to two test cases in dimension two, where we observe that
our approach behaves as well as in the one-dimensional case (see
Section~\ref{sec:2d-test}). In particular, in
Section~\ref{sec:test-2d-2}, we successfully address a classical
test-case of the literature. 

Section~\ref{sec:analysis} is devoted to the analysis of the approach,
in the homogenization setting. Our main result, Theorem~\ref{theo:H1}, is
presented in Section~\ref{sec:estim}, and proved in
Section~\ref{sec:proof_main}. The proofs of some technical results are
collected in Appendix~\ref{sec:tech-proof}. 
In addition, we specifically consider the one dimensional case in
Section~\ref{sec:1d}.

\section{MsFEM-type approaches}
\label{sec:msfem}

For consistency and the convenience of the reader, 
we present in this section the MsFEM approach to solve the original elliptic 
problem~\eqref{0}. For
clarity, we begin by presenting the approach in a deterministic
setting. The reader familiar with the MsFEM may easily skip this section
and directly proceed to Section~\ref{sec:weak-stoch}, where we present
our approach in a weakly stochastic setting.

\subsection{Description in a classical deterministic setting}
\label{sec:msfem_det}

Let $u^\varepsilon \in H^1(\mathcal{D})$ be the
solution to~\eqref{0}, where the matrix 
$A^{\varepsilon} \in (L^\infty(\mathcal{D}))^{d \times d}$ satisfies the
standard coercivity condition: there exists two constants 
$a_+ \geq a_- > 0$ such that
$$
\forall \eps, \quad
\forall x \in {\cal D}, \quad 
\forall \xi \in \RR^d, \quad
a_- |\xi|^2 \leq \xi^T A^\varepsilon(x) \xi 
\quad \text{and} \quad
\| A^{\varepsilon} \|_{L^\infty(\mathcal{D})} 
\leq a_+.
$$
Note that the MsFEM approach is not restricted to the
periodic setting. We therefore do not assume that
$A^\varepsilon(x)=A_{per}(x/\varepsilon)$ for a fixed periodic matrix
$A_{per}$. 

The MsFEM approach consists in performing a variational approximation of
(\ref{0}) where the basis functions are {\em precomputed}
and \emph{encode the fast oscillations} present in (\ref{0}). In
the sequel we argue on the following formulation, equivalent
to~(\ref{0}):  
\begin{equation}
 \mbox{Find }u^\varepsilon \in H^1_0(\mathcal{D}) \mbox{ such that, for
   any $v \in H^1_0(\mathcal{D})$}, \quad
 {\cal A}_\varepsilon(u^\varepsilon,v)=b(v),
\label{eq:det-var}
\end{equation}
where
$$
{\cal A}_\varepsilon(u,v)= \int_\mathcal{D}
(\nabla v(x))^T A^\varepsilon(x) \nabla u(x) \, dx 
\ \mbox{ and } \ 
b(v) = \int_{\mathcal{D}}f(x) \: v(x) \, dx.
$$
The MsFEM is a three-step approach:
\begin{enumerate}
\item introduce a standard discretization of the domain $\mathcal{D}$
  using a coarse mesh as compared to the small scale oscillations of
  $A^\varepsilon$.  
\item for each element $\cell$ of the coarse mesh, compute the basis
  function $\phi^{\varepsilon,\cell}_i$ as the solution of an elliptic
  equation posed in $\cell$ (see
  e.g.~\eqref{PB:MsFEM-Over-nb}-\eqref{eq:MsFEM-Over-nb-1}-\eqref{eq:MsFEM-Over-nb-2} below). 
\item solve the Galerkin approximation of (\ref{eq:det-var}), for the
  set of basis functions defined at Step $2$. 
\end{enumerate}
The advantage of the approach is that, for the same accuracy of the
approximation as that provided by a standard FEM, 
the macroscale mesh can be chosen sufficiently coarse so that the
resulting discretized problem has a limited number of degrees of
freedom, and may thus be computationally solved inexpensively. This is
observed in practice~\cite{hou1997}, and proven by a theoretical
analysis (see~\cite{hou1999,Efendiev2000}) when the problem
(\ref{eq:det-var}) admits a homogenized limit. See
also~\cite{Efendiev2009} and references therein.

\medskip

To further illustrate this fact, we reproduce here a simple
one-dimensional analysis we borrow from
A. Lozinsky (see~\cite[Chap. 6]{lozinski_hdr} and~\cite{carballal}). 
This analysis explains remarkably well the interest of the
approach, and, in contrast to~\cite{hou1999,Efendiev2000}, is not
restricted to a homogenization setting. Consider the one-dimensional
domain ${\cal D} = (0,1)$ and the reference problem
$$
{\cal L} u = f, \quad u(0) = u(1) = 0,
$$
for the operator ${\cal L}u := -(\nu u')'$, where $f \in L^2(0,1)$ and $\nu \in L^\infty(0,1)$ with $\nu(x) \geq
\nu_{min} > 0$ almost everywhere on $(0,1)$. The function $\nu$ may have
oscillations at a small scale. The associated weak formulation reads
\begin{equation}
\label{eq:fv_1d}
\text{Find $u \in H^1_0(0,1)$ such that, for any $v \in H^1_0(0,1)$,
\quad $a(u,v) = b(v)$},
\end{equation}
with
$$
a(u,v) = \int_0^1 \nu(x) u'(x) v'(x) \, dx 
\quad \text{and} \quad
b(v) = \int_0^1 f(x) v(x) \, dx.
$$ 
We now introduce the nodes $0 = x_0 < x_1 < \dots < x_L = 1$ that define
the elements $K_i = [x_{i-1},x_i]$. Let $h = \max |x_i - x_{i-1} |$ be
the mesh size. The multiscale finite element space 
\begin{equation}\label{def:whlozin}
{\cal W}_h = \left\{v_h \in C^0(0,1) \text{ such that ${\cal L} v_h = 0$ on
  each $K_i$} \right\},
\end{equation}
defined using the operator ${\cal L}$, is
{\em adapted} to the problem under study. We next proceed with a Galerkin
approximation of~\eqref{eq:fv_1d} using the space ${\cal W}_h$:
$$
\text{\em Find $u_h \in {\cal W}_h$ such that, for any $v_h \in {\cal W}_h$,
\quad $a(u_h,v_h) = b(v_h)$}.
$$
The solution $u_h$ then satisfies
\begin{equation}
\label{eq:estim_1d}
\| u - u_h \|_E \leq \frac{h}{\pi \sqrt{\nu_{min}}} \|f \|_{L^2(0,1)}
\end{equation}
where $\| \cdot \|_E = \sqrt{a(\cdot,\cdot)}$ is the energy norm. The
proof of this estimate goes as follows. By definition of $u$ and $u_h$,
we have $a(u-u_h,v_h) = 0$ for any $v_h \in {\cal W}_h$. Hence, $u_h$ is
the orthogonal projection of $u$ on ${\cal W}_h$ according to the scalar
product $a(\cdot,\cdot)$. 
Since $\| \cdot \|_E$ is the norm associated to that scalar product, we
have 
\begin{equation}
\label{eq:proj_1d}
\| u - u_h \|_E = \inf_{v_h \in {\cal W}_h} \| u - v_h \|_E.
\end{equation}
Choose $v_h$ to be the finite element interpolant of $u$, which is
defined by $v_h(x_i) = u(x_i)$ for any $i = 0, 1, \dots,L$, 
and consider the interpolation error $e = u - v_h$. On each element
$K_i$, we have, precisely because the space $\mathcal{W}_h$ is defined as~\eqref{def:whlozin}, 
$$
{\cal L} e = -(\nu e')' = f 
\quad \text{with} \quad e(x_{i-1}) = e(x_i) = 0.
$$
We multiply by $e$, integrate by part and obtain 
\begin{equation}
\label{eq:1d_interm}
\int^{x_i}_{x_{i-1}}
\nu(x) |e'(x)|^2 \, dx 
=
\int^{x_i}_{x_{i-1}} f(x) e(x) \, dx
\leq 
\|f\|_{L^2(K_i)} \|e\|_{L^2(K_i)}.
\end{equation}
Since $e$ vanishes on the boundary of $K_i$, the
Poincar\'e inequality with the best constant $(x_i - x_{i-1})/\pi$ yields
$$
\| e \|_{L^2(K_i)} 
\leq 
\frac{x_i - x_{i-1}}{\pi} \| e' \|_{L^2(K_i)}
\leq 
\frac{h}{\pi \sqrt{\nu_{min}}} 
\left( \int^{x_i}_{x_{i-1}} \nu(x) |e'(x)|^2 \, dx \right)^{1/2}.
$$
By substitution in~\eqref{eq:1d_interm}, we obtain
$$
\int^{x_i}_{x_{i-1}}
\nu(x) |e'(x)|^2 \, dx
\leq 
\frac{h^2}{\pi^2 \nu_{min}}
\|f\|^2_{L^2(K_i)}.
$$
Summing over the elements and using~\eqref{eq:proj_1d}
yields~\eqref{eq:estim_1d}. Using again that $\nu$ is bounded from
below, we deduce from~\eqref{eq:estim_1d} that
$$
\| u - u_h \|_{H^1(0,1)} \leq \frac{h}{C_{\cal D} \, \pi \, \nu_{min}} 
\|f \|_{L^2(0,1)},
$$
where $C_{\cal D}$ is the Poincar\'e constant of the domain ${\cal D} = (0,1)$.
As pointed out
in~\cite[Chap. 6]{lozinski_hdr}, the interest of the above estimate (or of
estimate~\eqref{eq:estim_1d}) lies in the fact that the constant in the
right-hand side only depends on $\nu$ through $\nu_{min}$, and remains
the same even if $\nu$ oscillates at a small scale. In contrast, for a
standard finite element method, the error is also proportional to $h$,
but with a constant that depends on the $H^2$ norm of the exact solution
$u$. With a standard finite element space $\mathcal{W}_h$, we indeed
classically deduce by C\'ea's lemma that
$$
\|u-u_h\|_{H^1(0,1)} 
\leq 
\frac{\| \nu \|_{L^\infty(0,1)}}{C_{\cal D} \, \nu_{min}} 
\inf\limits_{v_h \in \mathcal{W}_h} \|u-v_h\|_{H^1(0,1)} 
=
\frac{\| \nu \|_{L^\infty(0,1)}}{C_{\cal D} \, \nu_{min}} \
\|u - R_h u\|_{H^1(0,1)},
$$
where $C_{\cal D}$ is the Poincar\'e constant of the domain ${\cal D} =
(0,1)$, and $R_h u$ is the projection of $u$ on $\mathcal{W}_h$
according to the $H^1$ scalar product. 
We thus obtain that
$$
\|u-u_h\|_{H^1(0,1)} \leq C h \frac{\| \nu
  \|_{L^\infty(0,1)}}{\nu_{min}} \, \|D^2u\|_{L^2(0,1)},
$$
where $C$ is independent from the functions $\nu$ and $u$. If $\nu$
oscillates at a 
small scale (e.g. $\nu(x) = \bar
\nu(x/\eps)$ for a fixed function $\bar \nu$), the $H^2$ norm of
$u$ may be large (of the order of $\eps^{-1}$). A FEM approach then
requires $h$ to be smaller than $\eps$ to reach a good
accuracy.

We conclude this illustration by noting that such a general analysis of
the MsFEM approach is
not available in dimension $d \geq 2$. The analysis presented
in~\cite{hou1999,Efendiev2000}, which is performed without any
restriction on the dimension, additionally assumes that the matrix
$A^\varepsilon$ in (\ref{eq:det-var}) reads
$A^\varepsilon(x) = A_{per}(x/\varepsilon)$ for a fixed {\em periodic}
matrix $A_{per}$. 

\medskip

We now describe the MsFEM in a multidimensional setting.

\paragraph{Definition of the coarse mesh}

For simplicity (see Remark~\ref{rem:P2} below), we consider a classical
$\mathbb{P}_1$ discretization of the domain $\mathcal{D}$. We denote by
$\mathcal{T}_h$ the corresponding mesh, with $L$ nodes. Let
$\phi^{0}_i$, $i=1,\cdots,L$, be the basis functions. We introduce the
finite element space
$$
\mathcal{V}_h := \text{span}(\phi_i^0,i=1,\cdots,L),
$$
and define the restriction 
$$
 \phi^{0,\cell}_i:=\left.\phi^{0}_i\right|_\cell
$$
of these functions in each element $\cell$.
\begin{remark}
\label{rem:P2}
We refer to \cite{Allaire2005} for a presentation of a
MsFEM method that uses $\mathbb{P}_2$ macroscale basis functions.
\end{remark}

\paragraph{Definition of the MsFEM  basis}

Several definitions of the MsFEM basis functions have been proposed in
the literature (see e.g.
\cite{hou1997,hou1999,Efendiev2000,Allaire2005,Efendiev2009,hou2004}).
They all follow the same pattern but they give rise to various
methods. We present in the following the particular method that we have
implemented. It makes use of the oversampling technique introduced
in~\cite{hou1997} and developed in~\cite{Gloria2008}. 

For any element $\cell$, we consider a domain $\mathbf{S} \supset \cell$
(see Figure~\ref{Fig:OverS}), obtained from $\cell$ by an homothetic
transformation of center the centroid of $\cell$, and of 
ratio larger than $1$.
  
\begin{figure}[htbp]
 \centering
\includegraphics[scale=0.3]{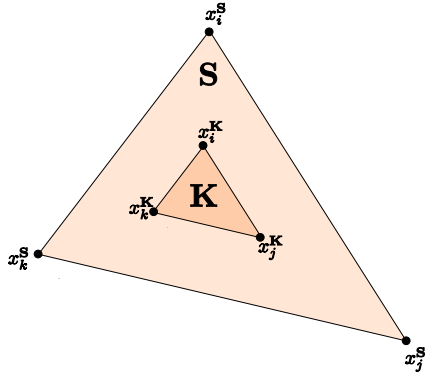}
 \caption{Definition of $\mathbf{S}$ (in 2D for clarity)}
 \label{Fig:OverS}
\end{figure}

Let $x^\mathbf{S}_j$ denote the coordinate of the vertex $j$ of the
domain $\mathbf{S}$. For any vertex $i$ of $\mathbf{S}$, we introduce
the affine function $\chi^{0,\mathbf{S}}_i$ (defined on $\mathbf{S}$)
that satisfies the condition
$\chi^{0,\mathbf{S}}_i(x^\mathbf{S}_j)=\delta_{ij}$ for all $j$. Let 
$\chi^{\varepsilon,\mathbf{S}}_i \in H^1(\mathbf{S})$ be the unique
solution to the problem
\begin{equation}
-\mbox{div}\left[A^{\varepsilon}(x)\nabla
   \chi^{\varepsilon,\mathbf{S}}_i(x) \right] = 0 
\ \ \mbox{in $\mathbf{S}$}, 
\quad
\chi^{\varepsilon,\mathbf{S}}_i = \chi^{0,\mathbf{S}}_i
\ \ \mbox{on $\partial \mathbf{S}$},
\label{PB:MsFEM-Over-nb}
\end{equation}
which, in practice, is numerically solved e.g. using a finite
element method with a mesh size adapted to the small scale~$\varepsilon$. 
We then define the local basis functions
\begin{equation}
\phi_i^{\varepsilon,\cell}=\sum\limits_{j=1}^{d+1} \alpha_{ij}
 \left. \chi^{\varepsilon,\mathbf{S}}_j \right|_{\cell}
\label{eq:MsFEM-Over-nb-1}
\end{equation}
as linear combinations of the restrictions of
$\chi^{\varepsilon,\mathbf{S}}_i$ on~$\cell$, with $\alpha_{ij}$
chosen such that 
\begin{equation}
\forall 1 \leq i,j \leq d+1, \quad 
\phi_i^{0,\cell}(x^\cell_j)=\sum\limits_{j=1}^{d+1} \alpha_{ij}
 \chi^{0,\mathbf{S}}_j(x^\cell_j)
=\delta_{ij},
\label{eq:MsFEM-Over-nb-2}
\end{equation}
where $x^\cell_j$ denotes the coordinate of the $j$th vertex of the
element $\cell$. Note that the condition~\eqref{eq:MsFEM-Over-nb-2} is
enforced on the function $\phi_i^{0,\cell}$, and not on
$\phi_i^{\varepsilon,\cell}$. The coefficients $\alpha_{ij}$ are
consequently independent from $\varepsilon$. As $\phi_i^{0,\cell}$ and
$\left. \chi^{0,\mathbf{S}}_j \right|_{\cell}$ are both affine on
$\cell$, condition~\eqref{eq:MsFEM-Over-nb-2} implies that
\begin{equation}
\forall 1 \leq i \leq d+1, \quad \forall x \in \cell, \quad 
\phi_i^{0,\cell}(x)=\sum\limits_{j=1}^{d+1} \alpha_{ij}
\chi^{0,\mathbf{S}}_j(x).
\label{eq:MsFEM-Over-nb-3}
\end{equation}
We next introduce the functions $\phi_i^\varepsilon$ defined on
${\mathcal D}$ by 
$\left.\phi_i^\varepsilon\right|_\cell = \phi_i^{\varepsilon,\cell}$ for
all elements $\cell$.

Note that the problems (\ref{PB:MsFEM-Over-nb}), indexed by
$\mathbf{S}$,  are all independent from one another. They may be
solved in parallel.  

\paragraph{Macroscale problem}

We now introduce the finite dimensional space
$$
\mathcal{W}_h := \mbox{span}(\phi_i^{\varepsilon},\ i=1,\cdots,L),
$$
and proceed with the approximation
\begin{equation}
\mbox{\em Find $u_M \in \mathcal{W}_h$ such that, for any
  $v \in \mathcal{W}_h$,
\quad ${\cal A}^h_\varepsilon(u_M,v)=b(v)$},
\label{eq:var-mac-det}
\end{equation}
of (\ref{eq:det-var}), where
$$
{\cal A}^h_\varepsilon(u,v)= \sum_\cell \int_\cell
(\nabla v(x))^T 
A^\varepsilon(x) \nabla u(x) \, dx 
\ \mbox{ and } \ 
b(v) = \int_{\mathcal{D}} f(x) \: v(x) \, dx.
$$
Observe that $\phi_i^\varepsilon$ has jumps across the edges of the
triangulation (due to the use of the oversampling technique), hence
$\mathcal{W}_h \not\subset H^1(\mathcal{D})$, thus 
the broken integral used to define ${\cal A}^h_\varepsilon(u,v)$. On the
other hand, since $\mathcal{W}_h \subset L^2(\mathcal{D})$, the linear
form $b$ is well defined for $v \in \mathcal{W}_h$. The
formulation (\ref{eq:var-mac-det}) is a {\em non-conforming} Galerkin
approximation of (\ref{eq:det-var}). This brings additional
error terms in the error estimation (see Lemma~\ref{lem:Strang2} in
Section~\ref{sec:analysis}). On another note, remark that the dimension of
$\mathcal{W}_h$ is equal to $L$. The formulation (\ref{eq:var-mac-det})
hence requires solving a linear system with only a limited number of
degrees of freedom. 

\bigskip

We are now in position to substantiate our claim in the introduction,
where we briefly mentioned that, in the stochastic setting, a direct
application of the MsFEM to approximate the solution to~\eqref{0-sto} is
{\em unpractical}. It would indeed 
lead to compute, for {\em each realization} of $A^\eps(x,\omega)$, first
a basis set and second a macroscale solution. This approach has been
briefly examined theoretically in~\cite{Chen2008a}. It is prohibitively
expensive. We therefore turn to an alternate approach.

\subsection{A weakly stochastic setting}
\label{sec:weak-stoch}

We now restrict the general setting and propose a dedicated, practical MsFEM
type approach. Following up on previous works
(see~\cite{Bris-sg,Blanc2010,Ronan2010,enumath}) and as announced
in~\eqref{eq:decompo}, we assume here 
that the random matrix $A^\varepsilon(x,\omega)$ in (\ref{0-sto}) is
a {\em perturbation of a deterministic matrix}, in the sense that 
\begin{equation}
\label{def:A-eta} 
A^\varepsilon(x,\omega) \equiv A^\varepsilon_\eta(x,\omega)
= A_0^\varepsilon(x) +\eta \: A^\varepsilon_1(x,\omega),
\end{equation}
where $\eta \in \RR$ is a
small deterministic parameter, $A^\varepsilon_0$ and $A^\varepsilon_1$
are bounded matrices, and $A^\varepsilon_0$ is
coercive, uniformly in $\varepsilon$. We also assume that the matrix
$A_\eta^\varepsilon$ itself 
satisfies the coercivity and boundedness assumptions, uniformly in
$\eta$ and $\varepsilon$ (we refer to~\cite{cras-arnaud,arnaud1,arnaud2}
and~\cite{jmpa,cras-ronan} for other perturbative settings). 

The principle of the proposed approach is to compute the MsFEM basis set
of functions with the {\em deterministic part} $A^\varepsilon_0$ of the
matrix $A^\varepsilon_\eta$, and next to
perform Monte-Carlo realizations for the macroscale
problem~\eqref{0-sto}-\eqref{def:A-eta}, where we keep the exact matrix
$A^\varepsilon_\eta$ (and not only its deterministic part).
Following the approach sketched in 
Section~\ref{sec:msfem_det}, we first solve (\ref{PB:MsFEM-Over-nb}),
with $A^\varepsilon(x) \equiv A_0^\varepsilon(x)$, and build the {\em
  deterministic} finite dimensional space
$$
\mathcal{W}_h := \mbox{span}(\phi_i^{\varepsilon}, \ i=1,\cdots,L)
$$
following~\eqref{eq:MsFEM-Over-nb-1}-\eqref{eq:MsFEM-Over-nb-2}.
We next proceed with a standard Galerkin approximation of
(\ref{0-sto})-(\ref{def:A-eta}) using $\mathcal{W}_h$. For each
$m\in\{1,\cdots ,M\}$, 
we consider a realization $A^{\varepsilon,m}_\eta(\cdot,\omega)$ and
compute $u^m_S(\cdot,\omega) \in \mathcal{W}_h$ such that 
\begin{equation}
\forall v \in \mathcal{W}_h, \quad  
\sum\limits_\cell \int_\cell (\nabla v(x))^T 
A^{\varepsilon,m}_\eta(x,\omega) \nabla u^m_S(x,\omega) \, dx
= \int_{\mathcal{D}} f(x) \: v(x) \,  dx.
\label{eq:var-mac-stoch}
\end{equation}
Since the MsFEM basis functions are only computed once (rather than for
each realization of $A^\varepsilon_\eta(x,\omega)$), a large
computational gain is expected, and obtained, in comparison to the
direct approach described above.

\section{Numerical simulations}
\label{sec:num}

This section is devoted to the many numerical simulations we have
performed. We first discuss some implementation details. Next, we
numerically estimate the performance of our
approach on various test cases, and assess its sensitivity with respect to the 
magnitude of $\eta$. We consider in Section~\ref{sec:test-1d} a test
case in dimension one. In Section~\ref{sec:2d-test}, we next study two
test cases in dimension two. We also study how the presence in
$A_1^\eps$ (the random component of the matrix $A^\eps_\eta$) of high
frequencies that are not present in the deterministic
component $A_0^\eps$, and that are thus not encoded in the highly
oscillatory basis functions, affects the accuracy of our approach. 

\medskip

Let $u^\varepsilon_\eta$ be the reference solution
to~\eqref{0-sto}-\eqref{eq:decompo} obtained using a finite element method
with a mesh size adapted to the small scale $\varepsilon$,
$u_S$ be the approximation given by our approach (described in Section
\ref{sec:weak-stoch}) and $u_M$ be the approximation given by the direct
approach (in
which the MsFEM basis set is recomputed for each realization
$A^{\varepsilon,m}_\eta(x,\omega)$, as explained at the end of
Section~\ref{sec:msfem_det}). 
Our goal is to compare the error $u_S-u_\eta^\varepsilon$ of our numerical
approximation with the error $u_M-u_\eta^\varepsilon$ of the direct and
expensive approach. When $\eta$ is small,
we expect the approximation $u_S$ to be essentially as accurate
as the approximation $u_M$, and we show below that this is indeed the case.

In the sequel, we assess the accuracy using the estimators
\begin{equation}
e_{L^2}(u_1,u_2)=\esp\left(\frac{\|u_1 - u_2 \|_{L^2({\cal
        D})}}{\|u_2\|_{L^2({\cal D})}}\right)
\quad \text{and} \quad
e_{H^1}(u_1,u_2)=\esp\left(\frac{\|u_1 - u_2 \|_{\h}}{\|u_2\|_{\h}}\right),
\label{eq:err-estim}
\end{equation}
where $u_1$ and $u_2$ are the solutions
obtained with any two different methods, and 
\begin{equation}
\label{eq:broken_H1}
\| u \|_\h := \left(\sum\limits_{\cell \in \mathcal{T}_h} 
\| u \|^2_{H^1(\cell)}\right)^{1/2}
\end{equation}
is the
broken $H^1$ norm. The expectation is in turn
computed using a Monte-Carlo method. Considering $M$ realizations
$\left\{ X_m(\omega)\right\}_{1\leq m\leq M}$ of a random variable, e.g.
$\dps X(\omega)=\frac{\|u_1(\cdot,\omega) - u_2(\cdot,\omega)
  \|_{\h}}{\|u_2(\cdot,\omega)\|_{\h}}$, we
compute the empirical mean 
$\mu_M$ and the empirical standard deviation $\sigma_M$ as 
\begin{equation}
\label{def:muM}
\mu_M(X)      = \frac{1}{M}\sum\limits_{m=1}^M X_m(\omega), 
\quad \quad
\sigma_M^2(X) = \frac{1}{M-1}\sum\limits_{m=1}^M
\left(X_m(\omega)-\mu_M(X)\right)^2.
\end{equation}
As a classical consequence of the Central Limit Theorem, the following
estimate is commonly employed:
$$
\left|\esp(X)-\mu_M(X)\right|\leq 1.96 \frac{\sigma_M}{\sqrt{M}}.
$$
It provides a practical evaluation of $\esp(X)$ from the knowledge of
$\mu_M(X)$ and $\sigma_M(X)$. The numerical parameters have been
determined by an empirical study of 
convergence. For instance, for the reference solution, we choose the
mesh size $h$ such that the quantity 
$\dps 
\frac{\|u_\eta^{\varepsilon,h}-u_\eta^{\varepsilon,h/2}\|_{H^1({\cal
      D})}}{\|u_\eta^{\varepsilon,h/2}\|_{H^1({\cal D})}}$ 
is smaller than 0.03 \%,
thereby formally admitting that the 
approximation has converged in $h$. The MsFEM parameters are determined
likewise. 

\medskip

All the computations have been performed using FreeFem++~\cite{hecht},
with the MPI tools. 

\subsection{Implementation details}
\label{sec:implem}

In the deterministic version of the MsFEM, the same matrix
$A^\varepsilon$ appears in the definition (\ref{PB:MsFEM-Over-nb}) of
the basis functions and in the macroscale variational formulation
(\ref{eq:var-mac-det}). This can be used to expedite the computation of
the stiffness matrix associated with (\ref{eq:var-mac-det}). In our
approach, described in Section \ref{sec:weak-stoch}, the
matrix that appears in the definition of the basis functions is
$A_0^\varepsilon$, whereas the macroscale variational problem involves
$A^\varepsilon_\eta \equiv A_0^\varepsilon+\eta A^\varepsilon_1$. An
additional numerical computation is thus needed.

To solve (\ref{eq:var-mac-stoch}), we need to compute, for each element
$\cell$ and each realization $A^{\varepsilon,\m}_\eta(x,\omega)$, the integrals 
\begin{equation}
\mathcal{K}^{\eta,\m}_{ij}(\omega)
=
\int_{\cell} \left(\nabla \phi^{\varepsilon,\cell}_i(x)\right)^T
A^{\varepsilon,\m}_\eta(x,\omega) \nabla \phi^{\varepsilon,\cell}_j(x) \,
dx,
\label{eq:Ass-int}
\end{equation}
where $\phi_i^{\varepsilon,\cell}$ are deterministic functions. 
We recall that 
$A^\varepsilon_\eta(x,\omega) = 
A_0^\varepsilon(x) + \eta A^\varepsilon_1(x,\omega)$
(see~\eqref{def:A-eta}). To allow for an efficient evaluation
of~\eqref{eq:Ass-int}, we assume henceforth that $A^\varepsilon_1$
is of the form 
\begin{equation}
\label{structA1}
A^\varepsilon_1(x,\omega)= \sum\limits_{k \in \ZZ^d} 
\mathbf{1}_{Q+k}\left(\frac{x}{\varepsilon}\right) \ X_k(\omega) \
B^k_\varepsilon(x),
\end{equation}
where $Q=(0,1)^d$, where 
$\left(X_k\right)_{k\in \ZZ^d}$ are scalar random variables, and
for any $k\in \ZZ^d$, $x \mapsto B^k_\varepsilon(x) \in \RR^{d
  \times d}$ are some deterministic functions. We comment on this
assumption in Remark~\ref{rem:assum_A1} below. The important consequence
of~\eqref{structA1} is that we can write the integral~\eqref{eq:Ass-int}
as a linear combination of {\em  deterministic integrals} over cells of
size $\varepsilon$, with {\em random} coefficients. To simplify the
notation, we assume that the spatial dimension is $d=2$. We define
$$
p=\left\lfloor \min \left(\frac{y_i}{\varepsilon},\frac{y_j}{\varepsilon},\frac{y_k}{\varepsilon}\right)\right\rfloor, 
\quad
q=\left\lfloor \max \left(\frac{y_i}{\varepsilon},\frac{y_j}{\varepsilon},\frac{y_k}{\varepsilon}\right)\right\rfloor + 1,
$$
and likewise, we define the integers $l$ and $m$ (see Fig.~\ref{Fig:Ass}). We can then write (\ref{eq:Ass-int}) as
\begin{equation}
\mathcal{K}^{\eta,\m}_{ij}(\omega)=\int_{\cell} 
\left(\nabla \phi^{\varepsilon,\cell}_i(x)\right)^T A^{\varepsilon,\m}_\eta(x,\omega) \nabla
\phi^{\varepsilon,\cell}_j(x) \, dx = \mathcal{K}^{0,\cell}_{ij} + \eta
\sum\limits_{\alpha=p}^{q-1} \sum\limits_{\beta=l}^{m-1}
X^{\m}_{\alpha,\beta}(\omega) \mathcal{K}^{1,\cell}_{\alpha\beta i j}, 
\label{eq:Ass}
\end{equation}
where
\begin{eqnarray}
\mathcal{K}^{0,\cell}_{ij} &=& \int_{\cell}
\left(\nabla \phi^{\varepsilon,\cell}_i(x)\right)^T A_0^\varepsilon(x)
\nabla \phi^{\varepsilon,\cell}_j(x) \, dx, 
\label{def:int-Aper} 
\\
\mathcal{K}^{1,\cell}_{\alpha\beta i j} &=&
 \int\limits_{\alpha\varepsilon}^{(\alpha+1)\varepsilon}
 \int\limits_{\beta\varepsilon}^{(\beta+1)\varepsilon}\!\!\!
 \mathbf{1}_\cell(x) \left(\nabla \phi^{\varepsilon,\cell}_i(x)\right)^T
 B_\varepsilon^k(x) \nabla \phi^{\varepsilon,\cell}_j(x) \, dx.
\label{def:int-A1}
\end{eqnarray}

\begin{figure}[htbp]
\begin{center}
\includegraphics[scale=0.3]{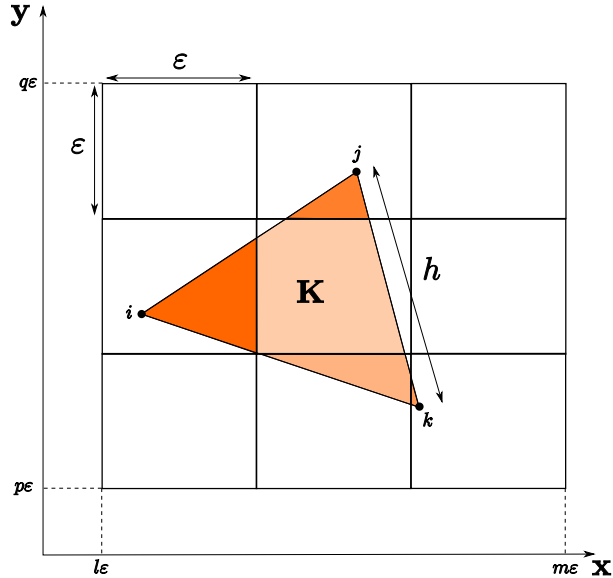}
\end{center}
\caption{To practically compute the integral (\ref{eq:Ass-int}), we
  write that each element $\cell$ (here in dimension $d=2$) is a subset
  of a quadrangle (here $[l \eps, m \eps] \times [p \eps, q \eps]$)
composed of cells of size $\varepsilon^d$.} 
\label{Fig:Ass}
\end{figure}

We thus compute once the deterministic integrals~(\ref{def:int-Aper})
and~(\ref{def:int-A1}). Next, for each realization of $A^\varepsilon_\eta$, we
evaluate the stiffness matrix elements
$\mathcal{K}^{\eta,\m}_{ij}(\omega)$ using the right hand side of
(\ref{eq:Ass}). No numerical quadrature is needed. As a consequence
of~\eqref{structA1}, 
most of the work for assembling the stiffness matrix is only performed once,
independently of the number of Monte Carlo realizations. This
significantly contributes to the gain in term of computational cost. 

\begin{remark}
\label{rem:assum_A1}
Assumption~\eqref{structA1} is quite general, and already covers
many interesting cases in practice. 
As explained above, the point in~\eqref{structA1} is that
$A_1^\eps$ is a {\em direct product} (or here, a sum of direct products) of a
function depending on $x$ with a random variable that only depends on
$\omega$. Otherwise stated, $A_1^\eps(x,\omega)$ depends {\em linearly},
in an explicit way, of $\omega$. A similar assumption is made when
applying reduced basis methods~\cite{RB_general} to a problem of the type
\begin{equation}
\label{eq:rb}
\text{Find $u_\lambda$ such that, for any $v$, $a(u_\lambda,v;\lambda) =
  b(v)$,}
\end{equation}
where $a(\cdot,\cdot;\lambda)$ is a bilinear form parameterized by
$\lambda$. Assume this problem has been solved for some values
$\left\{ \lambda_i \right\}_{i=1}^I$ of the parameter, yielding the
functions $\left\{ u_{\lambda_i} \right\}_{i=1}^I$. Under the assumption that 
$a(\cdot,\cdot;\lambda) = a_0(\cdot,\cdot) + \lambda a_1(\cdot,\cdot)$
(namely, $a(\cdot,\cdot;\lambda)$ depends {\em linearly} on $\lambda$),
one can precompute the stiffness matrix elements
$a_0(u_{\lambda_i},u_{\lambda_j})$ and
$a_1(u_{\lambda_i},u_{\lambda_j})$ for any $1 \leq i,j \leq I$. This
allows to next perform a very efficient Galerkin approximation of the
problem~\eqref{eq:rb} (for any $\lambda$) on the space 
$\text{Span}(u_{\lambda_i},i=1,\cdots,I)$.
\end{remark}

\subsection{One-dimensional test-case}
\label{sec:test-1d}

The purpose of this section is threefold. We first calibrate the number
$M$ of realizations considered for the Monte-Carlo method for the
two-dimensional numerical experiments that we consider in the sequel. We
next investigate how the accuracy of our approach depends on $\eta$ and on
the presence of frequencies in the random coefficient $a^\eps_\eta$ that
are not taken into account in the MsFEM basis set functions. The low
computational costs that we face in this one-dimensional situation
allow us to test our approach more comprehensively than in the
two-dimensional test-cases described below.

Let $\left( X_{k}\right)_{k \in \ZZ}$ denote a
sequence of independent, identically distributed scalar random variables
uniformly distributed in $[0,1]$. We consider the random coefficient
$$
a^\varepsilon_\eta(x,\omega)=\sum\limits_{k \in \ZZ}\mathbf{1}_{(k,k+1]}\left(\frac{x}{\varepsilon}\right)
\, \left(5 + 50 \sin^2\left(\frac{\pi x}{\varepsilon}\right)+\eta X_{k}(\omega) \: \kappa \:
  \sin^2\left(\frac{\zeta\pi x}{\varepsilon}\right)\right), 
$$
which is a particular example of the expansion (\ref{def:A-eta}) with
$$
a_0^\varepsilon(x)= 5 + 50 \sin^2\left(\frac{\pi x}{\varepsilon}\right)
\quad \text{and} \quad
a^\varepsilon_1(x,\omega)= \sum\limits_{k \in \ZZ}
\mathbf{1}_{(k,k+1]}\left(\frac{x}{\varepsilon}\right)\ X_{k}(\omega) \
\kappa \: \sin^2\left(\frac{\zeta \pi x}{\varepsilon}\right),
$$
and that satisfies the structural assumption~\eqref{structA1}.
We set $\varepsilon=0.025$ and choose $\kappa$ such that the quantity
\begin{equation}
\mathcal{K}(\kappa,\zeta)=\left\|
\frac{a^\varepsilon_1}{a_0^\varepsilon}\right\|_{L^\infty({\cal D} \times \Omega)} = \text{SupEss}_{\omega \in \Omega} \left\|
\frac{a^\varepsilon_1(\cdot,\omega)}{a_0^\varepsilon(\cdot,\omega)}\right\|_{L^\infty({\cal D})}
\label{def:control-1d}
\end{equation}
has the same value $\mathcal{K}=1$ for the three different values of
$\zeta = \{1,3,7\}$ we consider below. We analytically compute the
reference function $u_\eta^\varepsilon$, solution to
$$
-\frac{d}{dx}\left(a^\varepsilon_\eta\left(x,\omega\right)\frac{d
    u_\eta^\varepsilon}{dx}(x,\omega)\right) = 1 \ \ \mbox{in $(0,1)$},
\quad
u_\eta^\varepsilon(0,\omega) = u_\eta^\varepsilon(1,\omega) = 0,
$$
as well as the MsFEM basis functions for both approaches. 
Let $u_M$ and $u_S$ be the approximation of
$u_\eta^\varepsilon$ by the two MsFEM approaches described above, where
the coarse mesh size is $h=1/30$. 

\medskip

We first calibrate the number of independent realizations to accurately
approximate the exact expectation in~\eqref{eq:err-estim} by the
empirical mean~\eqref{def:muM}. To this aim, we 
present on Fig.~\ref{Fig:1Dconv} the mean and the confidence interval
computed using (\ref{def:muM}) 
for an increasing number $M$ of realizations (we compute up to 1000
independent realizations). We check that this
indicator reaches a plateau for $M \geq 30$, and thus converges
fast. On this example, considering 30 realizations is hence sufficient to
accurately compute the error~\eqref{eq:err-estim}. Based on this
observation, we will only consider $M=30$ realizations in
the two dimensional examples of Section~\ref{sec:2d-test}. 

\begin{figure}[htbp]
\centering
\includegraphics[angle=-90,scale=0.8]{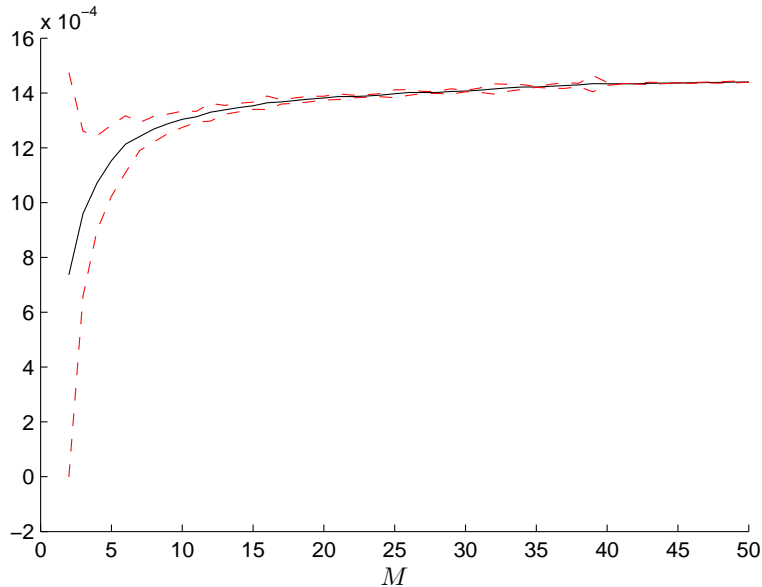}
\caption{Convergence of the indicator $e_{H^1}(u_M,u_\eta^\varepsilon)$
  (see~\eqref{eq:err-estim}), for $\eta=1$, $\zeta=1$ and
  $\kappa=55$. For each value of $M$, we plot the empirical mean along
  with its confidence interval, computed from the first $M$
  realizations. We only plot the results for the first $50$ realizations.
} 
\label{Fig:1Dconv}
\end{figure}

\begin{remark}
There is no reason to think that the calibration of our parameters that
we perform in the one-dimensional situation provides an adequate
adaptation of these parameters for the higher dimensional setting. We
however see no other manner to proceed and the approach has indeed
provided us with good results.  

Note also that the MsFEM approach is much more accurate in the
one-dimensional setting than in the two-dimensional setting (compare
Tables~\ref{Tab:1Dh1a},~\ref{Tab:1Dh1b} and~\ref{Tab:1Dh1c} with
Tables~\ref{Tab:h1a} and~\ref{Tab:h1b} below). This is due to the
specificity of the one dimensional setting. 
However, one-dimensional examples remain relevant for e.g.
assessing how the MsFEM accuracy depends on $\eta$. 
\end{remark}

\medskip

We now check how the accuracy of our
approach depends on $\eta$. 
In Tables~\ref{Tab:1Dh1a}, \ref{Tab:1Dh1b}, \ref{Tab:1Dh1c},
\ref{Tab:1DL2a}, \ref{Tab:1DL2b} and \ref{Tab:1DL2c}, we report the estimators
(\ref{eq:err-estim}), along with their confidence intervals, for various
choices of $(\kappa,\zeta)$ that all correspond to
$\mathcal{K}(\kappa,\zeta)=1$.  
For $\eta \leq 0.1$, we observe that
$\|u_S-u_\eta^\varepsilon\|$ and $\|u_M-u_\eta^\varepsilon\|$ are of the
same order of magnitude, and are both larger than $\|u_M-u_S\|$ (both in
$L^2$ and broken $H^1$ norms). We thus obtain the same 
accuracy with the direct and the weak stochastic MsFEM approaches,
whereas the weak stochastic MsFEM 
is computationally (much) less expensive. For $\eta=1$, as expected, the
accuracy of the approximation $u_S$
deteriorates. The accuracy of $u_M$ is independent of $\eta$. 

\begin{remark}
In Section~\ref{sec:analysis}, we estimate in the $H^1$ (broken) norm
the error between the reference solution $u_\eta^\eps$ and the weak
MsFEM solution $u_S$. For information, we also include in
Tables~\ref{Tab:1Dh1a}--\ref{Tab:1DL2c} the numerical comparison in the
$L^2$ norm. 
\end{remark}

\begin{table}[htbp]
\center
\caption{$H^1(0,1)$ error~\eqref{eq:err-estim} (in \%) for $\kappa=55$ and
  $\zeta=1$} 
{\begin{tabular}{c | c|  c|  c}
\hline
$\eta$ & $e_{H^1}(u_M,u_\eta^\varepsilon)$ & $e_{H^1}(u_S,u_\eta^\varepsilon)$ & $e_{H^1}(u_S,u_M)$ \\
\hline
$1$    &  $0.14644 \pm 0.00036$ & $2.62550 \pm 0.02696$ & $2.44359 \pm 
0.02696$ \\
$0.1$  &  $0.16001 \pm 0.00006$ & $0.15021 \pm 0.00051$ & $0.07036 \pm 
0.00044$ \\
$0.01$ &  $0.16258 \pm 0.00000$ & $0.10837 \pm 0.00002$ & $0.04825 \pm 
0.00025$ \\
\hline
\end{tabular} 
}\label{Tab:1Dh1a}
\center
\caption{$H^1(0,1)$ error~\eqref{eq:err-estim} (in \%) for $\kappa=14.38$ and $\zeta=3$}
{\begin{tabular}{c | c|  c|  c}
\hline
$\eta$ & $e_{H^1}(u_M,u_\eta^\varepsilon)$ & $e_{H^1}(u_S,u_\eta^\varepsilon)$ & $e_{H^1}(u_S,u_M)$ \\
\hline
$1$ &    $0.18269 \pm 0.00030$ & $2.38950 \pm 0.02277$ & $2.23869 \pm 0.02230$ 
\\
$0.1$ &  $0.16529 \pm 0.00003$ & $0.14959 \pm 0.00055$ & $0.08082 \pm 0.00041$ 
\\
$0.01$ & $0.16314 \pm 0.00000$ & $0.10840 \pm 0.00000$ & $0.04954 \pm 0.00001$ 
\\
\hline
\end{tabular} 
}\label{Tab:1Dh1b}
\center
\caption{$H^1(0,1)$ error~\eqref{eq:err-estim} (in \%) for $\kappa=8.39$ and $\zeta=7$}
{\begin{tabular}{c | c|  c|  c}
\hline
$\eta$ & $e_{H^1}(u_M,u_\eta^\varepsilon)$ & $e_{H^1}(u_S,u_\eta^\varepsilon)$ & $e_{H^1}(u_S,u_M)$ \\
\hline
$1$ &    $0.17436 \pm 0.00026$ & $2,34495 \pm 0.02105$ & $2,27358 \pm 0.02089$ 
\\
$0.1$ &  $0.16465 \pm 0.00004$ & $0.15748 \pm 0.00067$ & $0.09803 \pm 0.00053$ 
\\
$0.01$ & $0.16308 \pm 0.00000$ & $0.10846 \pm 0.00000$ & $0.05054 \pm 0.00001$ 
\\
\hline
\end{tabular} 
}\label{Tab:1Dh1c}
\end{table}

\begin{table}[htbp]
\center
\caption{$L^2(0,1)$ error~\eqref{eq:err-estim} (in \%) for $\kappa=55$ and $\zeta=1$}
{\begin{tabular}{c|  c|  c|  c}
\hline
$\eta$ & $e_{L^2}(u_M,u_\eta^\varepsilon)$ & $e_{L^2}(u_S,u_\eta^\varepsilon)$ & $e_{L^2}(u_S,u_M)$ \\
\hline
$1$ &     $0.00018 \pm 0.00000$ & $0.07286 \pm 0.00317$ & $0.06861 \pm 
0.00306$ \\
$0.1$ &   $0.00018 \pm 0.00000$ & $0.00045 \pm 0.00002$ & $0.00024 \pm 
0.00001$ \\
$0.01$ &  $0.00018 \pm 0.00000$ & $0.00015 \pm 0.00000$ & $0.00002 \pm 
0.00000$ \\
\hline
\end{tabular} 
}\label{Tab:1DL2a}
\center
\caption{$L^2(0,1)$ error~\eqref{eq:err-estim} (in \%) for $\kappa=14.38$ and $\zeta=3$}
{\begin{tabular}{c|  c|  c|  c}
\hline
$\eta$ & $e_{L^2}(u_M,u_\eta^\varepsilon)$ & $e_{L^2}(u_S,u_\eta^\varepsilon)$ & $e_{L^2}(u_S,u_M)$ \\
\hline
$1$ &    $0.00019 \pm 0.00000$ & $0.06658 \pm 0.00270$ & $0.06238 \pm 0.00261$ 
\\
$0.1$ &  $0.00018 \pm 0.00000$ & $0.00036 \pm 0.00001$ & $0.00019 \pm 0.00001$ 
\\
$0.01$ & $0.00018 \pm 0.00000$ & $0.00015 \pm 0.00000$ & $0.00002 \pm 0.00000$ 
\\
\hline
\end{tabular} 
}\label{Tab:1DL2b}
\center
\caption{$L^2(0,1)$ error~\eqref{eq:err-estim} (in \%) for $\kappa=8.39$ and $\zeta=7$}
{\begin{tabular}{c|  c|  c|  c}
\hline
$\eta$ & $e_{L^2}(u_M,u_\eta^\varepsilon)$ & $e_{L^2}(u_S,u_\eta^\varepsilon)$ & $e_{L^2}(u_S,u_M)$ \\
\hline
$1$ &    $0.00018 \pm 0.00000$ & $0.08903 \pm 0.00310$ & $0.08410 \pm 0.00261$ 
\\
$0.1$ &  $0.00018 \pm 0.00000$ & $0.00037 \pm 0.00002$ & $0.00016 \pm 0.00000$ 
\\
$0.01$ & $0.00018 \pm 0.00000$ & $0.00015 \pm 0.00000$ & $0.00003 \pm 0.00000$ 
\\
\hline
\end{tabular} 
}\label{Tab:1DL2c}
\end{table}

\medskip

We now turn to a different question. 
In the example considered here, some frequencies present in
$a^\varepsilon_1$ do not appear in $a_0^\varepsilon$, and are thus not
captured in the highly oscillatory basis functions
$\phi_i^\varepsilon$. We now show that our approach can still handle this
case, provided the amplitude of these modes remains small. 

We first consider the case when
the amplitude $\kappa$ associated to the frequency $\zeta$ is kept 
constant, and compare 
the performance of our approach in the case $\zeta = 1$ and $\zeta = 3$.
In the latter case, a relevant high frequency is not taken into
account in the basis set functions. Comparing Tables~\ref{Tab:1Dh1a}
and~\ref{Tab:1DL2a} (corresponding to $\zeta = 1$) with
Tables~\ref{Tab:1Dh1ab} and~\ref{Tab:1DL2ab} (corresponding to $\zeta =
3$) for a given value of $\eta$, we see that the accuracy of our approach
deteriorates. This is not unexpected, of course. Otherwise stated, to achieve a given accuracy (say an
error of 15 \% in the broken $H^1$ norm), we need to take smaller values of
$\eta$ (namely $\eta \leq 0.01$) when $\zeta=3$ than when $\zeta=1$ (in
which case $\eta = 0.1$ is already a sufficiently small value). 

We now run the comparison differently. As we increase the gap
between the frequency present in $a^\varepsilon_1$ and that present
in $a^\varepsilon_0$ (i.e., as we increase $\zeta$), we simultaneously
decrease the amplitude $\kappa$ of that mode. In practice, we
do this by keeping constant the parameter $\mathcal{K}(\kappa,\zeta)$
defined by~(\ref{def:control-1d}). Then the accuracy of our approach
remains constant, and is independent of $\zeta$. See indeed the
numerical results of Tables~\ref{Tab:1Dh1a}-\ref{Tab:1DL2c}, that all
correspond to the choice $\mathcal{K}(\kappa,\zeta)=1$, for three
different values of $\zeta$. We observe that, at fixed $\eta$,
errors are comparable, and independent of the value of $(\kappa,\zeta)$. 

In conclusion, the accuracy of our approach depends {\em both} on the
amplitude $\kappa$ and the value $\zeta$ of the high frequency not taken
into account in the MsFEM basis set functions. If $\zeta$ and $\kappa$
are scaled 
so that $\mathcal{K}(\kappa,\zeta)$ remains constant (which implies that 
$\kappa$ decreases if $\zeta$ increases), then the accuracy
of our approach remains constant.

\begin{table}[htbp]
\center
\caption{$H^1(0,1)$ error~\eqref{eq:err-estim} (in \%) for $\kappa=55$ and $\zeta=3$}
{\begin{tabular}{c | c|  c|  c}
\hline
$\eta$ & $e_{H^1}(u_M,u_\eta^\varepsilon)$ & $e_{H^1}(u_S,u_\eta^\varepsilon)$ & $e_{H^1}(u_S,u_M)$ \\
\hline
$1$    &  $0.21826 \pm 0.00073$ & $12.30047 \pm 0.10647$ & $12.01694 \pm 
0.10617$ \\
$0.1$  &  $0.17142 \pm 0.00013$ & $0.59293 \pm 0.00519$ & $0.49523 \pm 
0.00489$ \\
$0.01$ &  $0.16383 \pm 0.00001$ & $0.11448 \pm 0.00014$ & $0.05247 \pm 
0.00007$ \\
\hline
\end{tabular} 
}\label{Tab:1Dh1ab}
 \center
\caption{$L^2(0,1)$ error~\eqref{eq:err-estim} (in \%) for $\kappa=55$ and $\zeta=3$}
{\begin{tabular}{c|  c|  c|  c}
\hline
$\eta$ & $e_{L^2}(u_M,u_\eta^\varepsilon)$ & $e_{L^2}(u_S,u_\eta^\varepsilon)$ & $e_{L^2}(u_S,u_M)$ \\
\hline
$1$ &     $0.00022 \pm 0.00000$ & $1.53780 \pm 0.03878$ & $1.51837 \pm 
0.00385$ \\
$0.1$ &   $0.00019 \pm 0.00000$ & $0.00503 \pm 0.00027$ & $0.00406 \pm 
0.00024$ \\
$0.01$ &  $0.00018 \pm 0.00000$ & $0.00018 \pm 0.00000$ & $0.00005 \pm 
0.00000$ \\
\hline
\end{tabular} 
}\label{Tab:1DL2ab}
\end{table}

\subsection{Two-dimensional test-cases}
\label{sec:2d-test}

We now test our approach on two-dimensional test cases. Using the first
test case, we show, similarly to the one-dimensional situation, that the weak
stochastic MsFEM yields accurate results, provided the parameter
$\eta$ is sufficiently small, and provided that the amplitude associated to
frequencies present in $A^\varepsilon_\eta$ but not encoded in the
deterministic basis functions is small (see
Section~\ref{sec:test-2d-1}). Next, in Section~\ref{sec:test-2d-2}, we
consider a test case similar to a classical benchmark test case of the literature. We again observe that our
approach is efficient.
For both cases, we show that the parameter $\eta$ does not need to be
extremely small for our approach to be highly competitive. 

\subsubsection{A multi-frequency case} 
\label{sec:test-2d-1}

In line with what we observed in the one-dimensional case, we show here that
the weak stochastic MsFEM provides interesting results even in the case
when not all the frequencies present in $A^\varepsilon_\eta$ are captured in the deterministic
basis functions, provided their amplitude is not too large. To this aim,
we consider the following numerical 
example. 

Let $\left( X_{k,l}\right)_{(k,l) \in \ZZ^2}$ denote a
sequence of independent, identically distributed scalar random variables
uniformly distributed in the interval $[0,1]$. We consider the 
random matrix 
$$
A_\eta^\varepsilon(x,y,\omega)
=
a_0^\varepsilon(x,y) \ \text{Id}_2 + 
\eta a_1^\varepsilon(x,y,\omega) \ \text{Id}_2,
$$
with
\begin{eqnarray*}
a_0^\varepsilon(x,y) &=& 5 + 50 \sin^2 \left(\frac{\pi x}{\eps}\right)
\sin^2 \left(\frac{\pi y}{\eps} \right),
\\
a_1^\varepsilon(x,y,\omega) &=& \sum\limits_{(k,l) \in \ZZ^2} 
\mathbf{1}_{(k,k+1]} \left(\frac{x}{\varepsilon}\right) 
\mathbf{1}_{(l,l+1]} \left(\frac{y}{\varepsilon}\right) 
\left( X_{k,l}(\omega) \: \kappa \: 
\sin^2 \left( \frac{\zeta \pi x}{\eps} \right) 
\sin^2 \left( \frac{\zeta \pi y}{\eps} \right) \right).
\end{eqnarray*}
Again, this choice is a particular example of the expansion
(\ref{def:A-eta}) satisfying the structural
assumption~\eqref{structA1}. We consider two different values of
$\zeta$, namely $\zeta=1$ and $\zeta=3$. As in the previous test case,
the frequency $\zeta$ is 
not present in the deterministic part of $A^\eps_\eta$, and thus not
encoded in the basis functions. 
In line with what we observed in Section~\ref{sec:test-1d}, we choose
the amplitude $\kappa$ associated to that frequency such that the
quantity~\eqref{def:control-1d} 
has the same value $\mathcal{K}=1$ for both values of $\zeta$. We
compute $u_\eta^\varepsilon$ solution to  
$$
-\mbox{div}\left[A_\eta^\varepsilon(\cdot,\omega) \, \nabla 
u_\eta^\varepsilon(\cdot,\omega)\right] = 1 \ \ \mbox{in $\mathcal{D}$}, 
\quad 
u_\eta^\varepsilon(\cdot,\omega) = 0 \ \ \mbox{ on $\partial \mathcal{D}$}, 
$$
on the domain $\mathcal{D}= (0,1)^2$ with $\varepsilon=0.025$. Let $u_M$
and $u_S$ be its approximation by the two MsFEM approaches described
above. The numerical parameters for the computation are again determined
using an empirical study of convergence. We use for the reference
solution $u^\varepsilon_\eta$ a fine mesh of size $h_f=\varepsilon/40$. The
MsFEM basis functions are computed in each element $\cell$ using a mesh
of size $h_M=\varepsilon/80$. The oversampling parameter (i.e. the scale
ratio of the homothetic transformation between $\cell$ and $\mathbf{S}$,
see Fig.~\ref{Fig:OverS}) is equal to $3$. The coarse mesh size is
$h=1/30$. In view of the results of Section~\ref{sec:test-1d}, we
consider $M=30$ independent realizations, which will prove to again be
sufficient to obtain accurate results. 

In Tables~\ref{Tab:h1a} and \ref{Tab:h1b} (Tables~\ref{Tab:L2a}
and~\ref{Tab:L2b} respectively), we report the estimator
(\ref{eq:err-estim}), along with its 
confidence interval, for the broken $H^1(\mathcal{D})$ norm and for the
$L^2(\mathcal{D})$ norm, respectively. The results obtained here confirm our
observations in the one-dimensional setting (Section
\ref{sec:test-1d}):
\begin{itemize}
\item for given $\zeta$ and $\kappa$, we observe that, when $\eta$ is
  sufficiently small (here, $\eta \leq 0.1$), the alternative approach
  provides a solution $u_S$ that is an approximation of
  $u_\eta^\varepsilon$ as accurate as $u_M$, for a much smaller
  computational cost (as the MsFEM basis set has only been computed once
  rather than for each independent realization of $A_\eta^\varepsilon$).
\item our approach yields accurate results even if the frequency
$\zeta$ is not encoded in the basis functions $\phi_i^\varepsilon$,
provided the associated amplitude $\kappa$ is scaled accordingly.
Figures in Table~\ref{Tab:h1a} 
(respectively Table~\ref{Tab:L2a}) are very close to those of
Table~\ref{Tab:h1b} (respectively Table~\ref{Tab:L2b}). This confirms
that the error made by the weak stochastic MsFEM seems to be independent
of $\kappa$ and $\zeta$, provided these two parameters are scaled so
that $\mathcal{K}(\kappa,\zeta)$ remains constant. If $\zeta$ becomes different than 1, the frequency present in
$a_0^\varepsilon$, then the amplitude $\kappa$ associated to the
frequency $\zeta$ has to decrease to keep $\mathcal{K}(\kappa,\zeta)$
(and thus the accuracy of $u_S$) constant. 
\end{itemize}
These observations again demonstrate the efficiency of the
approach.

\begin{table}[htbp]
\center
\caption{$H^1(\mathcal{D})$ error~\eqref{eq:err-estim} (in \%) for $\kappa=73.61$ and $\zeta=1$}
{\begin{tabular}{c | c|  c|  c}
\hline
$\eta$ & $e_{H^1}(u_M,u_\eta^\varepsilon)$ & $e_{H^1}(u_S,u_\eta^\varepsilon)$ & $e_{H^1}(u_S,u_M)$ \\
\hline
$1$    &  $7.8437 \pm 0.1350$ & $19.8818 \pm 0.4123$ & $18.8662 \pm 0.4216$ \\
$0.1$  &  $6.8053 \pm 0.0165$ & $7.3868  \pm 0.0276$ & $3.1528  \pm 0.0517$ \\
$0.01$ &  $6.7338 \pm 0.0017$ & $6.9795  \pm 0.0016$ & $1.8763  \pm 0.0013$ \\
\hline
\end{tabular} 
}\label{Tab:h1a}
\center
\caption{$H^1(\mathcal{D})$ error~\eqref{eq:err-estim} (in \%) for $\kappa=10$ and $\zeta=3$}
{\begin{tabular}{c | c|  c|  c}
\hline
$\eta$ & $e_{H^1}(u_M,u_\eta^\varepsilon)$ & $e_{H^1}(u_S,u_\eta^\varepsilon)$ & $e_{H^1}(u_S,u_M)$ \\
\hline
$1$ &    $6.7224 \pm 0.0368$ & $12.7292 \pm 0.2172$ & $10.8128 \pm 0.2442$ \\
$0.1$ &  $6.7154 \pm 0.0044$ & $ 7.1069 \pm 0.0128$ & $ 2.2925 \pm 0.0206$ \\
$0.01$ & $6.1725 \pm 0.0004$ & $ 6.9770 \pm 0.0010$ & $ 1.8504 \pm 0.0003$ \\
\hline
\end{tabular} 
}\label{Tab:h1b}
\center
\caption{$L^2(\mathcal{D})$ error~\eqref{eq:err-estim} (in \%) for $\kappa=73.61$ and $\zeta=1$}
{\begin{tabular}{c|  c|  c|  c}
\hline
$\eta$ & $e_{L^2}(u_M,u_\eta^\varepsilon)$ & $e_{L^2}(u_S,u_\eta^\varepsilon)$ & $e_{L^2}(u_S,u_M)$ \\
\hline
$1$ &     $1.4355 \pm 0.0795$ & $4.1649 \pm 0.1652$ & $2.8468 \pm 0.1694$ \\
$0.1$ &   $1.0630 \pm 0.0108$ & $1.1369 \pm 0.0075$ & $0.1441 \pm 0.0354$ \\
$0.01$ &  $1.0211 \pm 0.0011$ & $1.1512 \pm 0.0007$ & $0.1351 \pm 0.0014$ \\
\hline
\end{tabular} 
}\label{Tab:L2a}
\center
\caption{$L^2(\mathcal{D})$ error~\eqref{eq:err-estim} (in \%) for $\kappa=10$ and $\zeta=3$}
{\begin{tabular}{c|  c|  c|  c}
\hline
$\eta$ & $e_{L^2}(u_M,u_\eta^\varepsilon)$ & $e_{L^2}(u_S,u_\eta^\varepsilon)$ & $e_{L^2}(u_S,u_M)$ \\
\hline
$1$ &    $1.0744 \pm 0.0127$ & $1.8433 \pm 0.0582$ & $0.8426 \pm 0.0832$ \\
$0.1$ &  $1.0226 \pm 0.0015$ & $1.1249 \pm 0.0038$ & $0.1147 \pm 0.0073$ \\
$0.01$ & $1.0170 \pm 0.0001$ & $1.1551 \pm 0.0004$ & $0.1427 \pm 0.0003$ \\
\hline
\end{tabular} 
}\label{Tab:L2b}
\end{table}

\begin{remark}
In Tables~\ref{Tab:h1a}-\ref{Tab:L2b}, we observe that the size of the
confidence interval is much smaller than the distance between two
different errors. This a posteriori validates the choice of the number
$M$ of Monte Carlo realizations according to the calibration we
performed in the one-dimensional setting. In the two-dimensional setting
studied here, we observe that considering
$M=30$ realizations is again sufficient. The same conclusion holds for
results presented in Tables~\ref{Tab:h1}-\ref{Tab:L2eta} below. 
\end{remark}

\subsubsection{A classical test case}
\label{sec:test-2d-2}

We consider in this section a test case similar to a classical test case
of the literature (see e.g.~\cite{hou1997,hou2004,Chen2008,Efendiev2000}). Let 
$\left( X_{k,l}\right)_{(k,l) \in \ZZ^2}$ denote a sequence of
independent, identically distributed scalar random variables uniformly
distributed in the interval $[0,1]$. We consider the random matrix 
$$
A_\eta^\varepsilon(x,y,\omega)=\sum\limits_{(k,l) \in
  \ZZ^2} \mathbf{1}_{(k,k+1]} \left(\frac{x}{\varepsilon}\right)
\mathbf{1}_{(l,l+1]} \left(\frac{y}{\varepsilon}\right) \left(
  \frac{2+P \sin(2 \pi x / \varepsilon)}{2+P\sin(2 \pi y / \varepsilon)}
  + \frac{2+\sin(2 \pi y / \varepsilon)}{2+P\sin(2 \pi x /
    \varepsilon)}\right) \left(1+\eta X_{k,l}(\omega)\right) \ \text{Id}_2, 
$$
with $P=1.8$ and $\varepsilon=0.025$. We compute the reference
solution $u_\eta^\varepsilon$ and its two approximations $u_M$ and $u_S$
with the same numerical parameters as in Section~\ref{sec:test-2d-1}.  

In Tables~\ref{Tab:h1} and~\ref{Tab:L2}, we report the estimator
(\ref{eq:err-estim}), along with its confidence interval, for the broken
$H^1(\mathcal{D})$ norm and for the $L^2(\mathcal{D})$ norm,
respectively. We again see 
that, when $\eta$ is sufficiently small, $u_S$ is an
approximation of the reference solution $u^\eps_\eta$ as accurate as
$u_M$. In Tables~\ref{Tab:h1eta} and~\ref{Tab:L2eta}, we report on the
accuracy of $u_S$, for more values of $\eta$. Assuming that the accuracy
of $u_M$ does not depend on $\eta$ (which is consistent with the results
reported in Tables~\ref{Tab:h1} and~\ref{Tab:L2}), we see that our
approach is as accurate as the direct, expensive MsFEM approach, as soon
as $\eta \leq 0.1$ (if we use the broken $H^1$ norm to assess accuracy) and
$\eta \leq 0.25$ (if we rather use the $L^2$ norm). The parameter $\eta$
hence does not need to be extremely small for our approach to be
highly competitive. 

\begin{table}[htbp]
\center
\caption{$H^1(\mathcal{D})$ error~\eqref{eq:err-estim} (in \%)}
{\begin{tabular}{c | c|  c|  c}
\hline
$\eta$ & $e_{H^1}(u_M,u_\eta^\varepsilon)$ & $e_{H^1}(u_S,u_\eta^\varepsilon)$ & $e_{H^1}(u_S,u_M)$ \\
\hline
$1$ & $8.1154 \pm 0.1913$ & $17.3678 \pm 0.7784$ & $15.5113 \pm 0.8689$ \\
$0.1$ & $7.1664 \pm 0.0199$ & $7.0524 \pm 0.0705$ & $2.5638 \pm 0.1006$ \\
$0.01$ & $7.1453 \pm 0.0020$ & $7.2837 \pm 0.0067$ & $1.3882 \pm 0.0020$ \\
\hline
\end{tabular} 
}\label{Tab:h1}
\center
\caption{$L^2(\mathcal{D})$ error~\eqref{eq:err-estim} (in $\%$)}
{\begin{tabular}{c|  c|  c|  c}
\hline
$\eta$ & $e_{L^2}(u_M,u_\eta^\varepsilon)$ & $e_{L^2}(u_S,u_\eta^\varepsilon)$ & $e_{L^2}(u_S,u_M)$ \\
\hline
$1$ & $0.5620    \pm 0.0803$ & $1.6855 \pm 0.4860$ & $1.4739 \pm 0.5048$ \\
$0.1$ & $0.5354  \pm 0.0160$ & $0.5688 \pm 0.0630$ & $0.1984 \pm 0.0712$ \\
$0.01$ & $0.5347 \pm 0.0012$ & $0.6192 \pm 0.0054$ & $0.1072 \pm 0.0032$ \\
\hline
\end{tabular} 
}\label{Tab:L2}
\end{table}

\begin{table}
\centering
\begin{minipage}{0.4\linewidth}
\center
\caption{$H^1(\mathcal{D})$ error~\eqref{eq:err-estim} (in $\%$)}
\renewcommand{\arraystretch}{1.5}
{\begin{tabular}{c | c }
\hline
$\eta$  & $e_{H^1}(u_S,u_\eta^\varepsilon)$  \\
\hline
$1$     & $17.3678 \pm 0.7784$    \\
$0.5$   & $15.9578 \pm 0.3461$     \\
$0.25$  & $10.6130 \pm 0.1591$     \\
$0.1$   & $7.0524 \pm 0.0705$     \\
$0.01$  & $7.2837 \pm 0.0067$     \\
\hline
\end{tabular} 
}\label{Tab:h1eta}
\end{minipage}
\begin{minipage}{0.4\linewidth}
\center
\caption{$L^2(\mathcal{D})$ error~\eqref{eq:err-estim} (in $\%$)}
\renewcommand{\arraystretch}{1.5}
{\begin{tabular}{c|  c }
\hline
$\eta$ & $e_{L^2}(u_S,u_\eta^\varepsilon)$  \\
\hline
$1$    & $1.6855 \pm 0.4860$     \\
$0.5$  & $1.0246 \pm 0.4414$     \\
$0.25$ & $0.5291 \pm 0.2285$     \\
$0.1$  & $0.5688 \pm 0.0630$     \\
$0.01$ & $0.6192 \pm 0.0054$     \\
\hline
\end{tabular} 
}\label{Tab:L2eta}  
\end{minipage}
\end{table}

\section{Analysis}
\label{sec:analysis}

This section is devoted to the analysis of the approach introduced in
Section~\ref{sec:weak-stoch}, and to the derivation of error bounds. As
is often the case for the 
MsFEM (see e.g.~\cite{Efendiev2000}), we perform the analysis in a
setting where the problem~\eqref{0-sto}-\eqref{eq:decompo} that we
consider admits a homogenized limit as $\eps$ vanishes (although, we
repeat it, the approach is used in practice for more general cases). 
The structure of
our proof is similar to that for the deterministic setting, which we
now overview (we refer to~\cite{Efendiev2000} for all the details). 

\medskip

In the case when the oversampling technique is not used, the MsFEM is
a conforming Galerkin approximation, and the error is estimated using
the C\'ea lemma: 
$$
\|u^\eps - u_M \|_{H^1}
\leq C \inf\limits_{v_h \in \mathcal{W}_h} 
\|u^\eps - v_h\|_{H^1}, 
$$
where $u^\eps$ is the solution to the reference deterministic highly oscillatory
problem~\eqref{0}, $u_M$ is the MsFEM solution, and the constant $C$
is independent from $\eps$ and $h$. In the case when the oversampling
technique is used, the MsFEM is a non-conforming Galerkin method. The
error is then bounded from above by the sum of the best approximation
error (the right-hand side of the above estimate) and the non-conforming
error (that we do not detail here): 
$$
\|u^\eps - u_M \|_{H^1}
\leq C \left[ \inf\limits_{v_h \in \mathcal{W}_h} 
\|u^\eps - v_h\|_{H^1} 
+
\text{non-conforming error} \right].
$$
Note that, in the non-conforming case, the MsFEM solution $u_M$ does not
belong to $H^1$, and one should write the above estimate with a broken
$H^1$ norm rather than the $H^1$ norm. For the sake of clarity, we
ignore this distinction in this preliminary discussion. 

Taking advantage of the
homogenization setting, we introduce the two-scale expansion
$$ 
v^\eps = u^\star + \varepsilon \sum_{i=1}^d 
\wper_{e_i} \left( \frac{\cdot}{\varepsilon} \right) \partial_i u^\star
$$
of
$u^\eps$, where $u^\star$ is the homogenized solution, $\wper_{e_i}$ is the
periodic corrector associated to $e_i \in \RR^d$, and $\partial_i
u^\star$ denotes the partial derivative $\dps \frac{\partial
  u^\star}{\partial x_i}$. We next write
$$
\|u^\eps - u_M \|_{H^1}
\leq C \left[ \|u^\eps - v^\eps \|_{H^1}
+
\inf\limits_{v_h \in \mathcal{W}_h} \|v^\eps - v_h\|_{H^1} 
+
\text{non-conforming error} 
\right].
$$
The first term in the right-hand side is estimated using standard
homogenization results. To estimate the second term, one considers a
suitably chosen element $v_h \in \mathcal{W}_h$, for which 
$\dps \|v^\eps - v_h\|_{H^1}$ can be estimated directly. 
The main idea is that the 
highly oscillating part of $v^\eps$ can be well approached by an element
in $\mathcal{W}_h$, since, by construction, the highly oscillatory basis
functions are defined by a problem similar to the corrector problem, and
thus encode the same highly oscillatory behavior as that present in the
correctors $\wper_{e_i}$. We are thus left with approximating the slowly
varying components of $v^\eps$, for which standard FEM estimates are
used. Lastly, we again use the fact that our problem admits a
homogenized limit to estimate the third term, i.e. the non-conforming error.

\medskip

In the sequel, we follow the same strategy in our stochastic setting. We
hence first write (see~\eqref{Strang2} below) that
\begin{equation}
\label{eq:lundi}
\|u_\eta^\eps(\cdot,\omega) - u_S(\cdot,\omega) \|_{H^1}
\leq C \left[
\inf\limits_{v_h \in \mathcal{W}_h} \|u_\eta^\eps(\cdot,\omega) -
v_h(\cdot,\omega) \|_{H^1}
+
\text{non-conforming error} 
\right],
\end{equation}
where $u_\eta^\eps$ is the solution to the reference stochastic
problem~\eqref{0-sto}-\eqref{eq:decompo} and $C$ is a deterministic
constant independent from $\eps$, $h$ and $\eta$ 
(note that, in~\eqref{Strang2}, we use a {\em broken} $H^1$ norm rather
than the $H^1$ norm; as pointed out above, this is due to the
fact that our approach is a {\em non-conforming} Galerkin
approximation; we ignore this distinction in the current discussion). 
To estimate the best approximation error (the first term
in the right-hand side of~\eqref{eq:lundi} above), we use the triangle
inequality, and write (see~\eqref{eq:triang2_pre} below) that
\begin{equation}
\label{tutu}
\inf_{v_h \in \mathcal{W}_h} 
\|u_\eta^\eps(\cdot,\omega) - v_h(\cdot,\omega) \|_{H^1}
\leq 
\|u_\eta^\eps(\cdot,\omega) - v_\eta^\eps(\cdot,\omega) \|_{H^1}
+
\inf\limits_{v_h \in \mathcal{W}_h} \|v_\eta^\eps(\cdot,\omega) -
v_h(\cdot,\omega) \|_{H^1},
\end{equation}
where $v_\eta^\eps$ is the two-scale expansion of the solution
$u_\eta^\eps$ truncated at order $\varepsilon^2$. A first
difficulty owes to the fact that, in the general stochastic setting, no
estimate is known on $\|u_\eta^\eps(\cdot,\omega) -
v_\eta^\eps(\cdot,\omega) \|_{H^1}$. One 
only knows that its expectation vanishes when $\eps \to 0$. However, in
the present article, we consider a {\em weakly} stochastic 
case. In that setting, we have derived such a convergence rate type
result in~\cite{rate}, and 
we can thus bound the first term of~\eqref{tutu} 
(see Section~\ref{sec:2scale} below for more details). The second term, 
$\dps \inf\limits_{v_h \in \mathcal{W}_h} 
\|v_\eta^\eps(\cdot,\omega) - v_h\|_{H^1}$, of~\eqref{tutu}, is
estimated using an explicit construction of a suitable $v_h$
(see~\eqref{eq:triang2_pre2}), similarly to the deterministic
setting. We again use there our specific weakly stochastic setting. 
Lastly, the non-conforming error (the second term
in the right-hand side of~\eqref{eq:lundi} above) is estimated following
arguments similar to those of the deterministic case, using that
our problem admits a homogenized limit and is weakly stochastic.

\medskip

This section is organized as follows. The error estimation is presented
in Section~\ref{sec:estim}. 
We first recall in Section~\ref{sec:homo-eq} the formulation of the
homogenized problem, and some results specific to the weakly stochastic
case. Next, in Section~\ref{sec:2scale}, we establish an 
error bound between the reference solution $u_\eta^\varepsilon$
and its two-scale expansion $v^\eps_\eta$ (see
Theorem~\ref{exp2scaleueta}), which allows to bound the first term in
the right-hand side of~\eqref{tutu}. Our main result,
Theorem~\ref{theo:H1}, is given in
Section~\ref{sec:weak_stochastic-MsFEM}, and proved in
Section~\ref{sec:proof_main}. The proof essentially consists in
explicitly 
building a function $v_h \in {\cal W}_h$ such that the second term
of~\eqref{tutu} can be directly estimated. It also makes use of several
technical results (Lemmas~\ref{Lem1},~\ref{Lem2}
and~\ref{coro:lambda} below) to bound the non-conforming error, i.e. the
second term in the right hand side of~\eqref{eq:lundi}. The proof of
these technical results is postponed until
Appendix~\ref{sec:tech-proof}. Last, in Section~\ref{sec:1d}, we
specifically consider the one dimensional case.

\bigskip

Before proceeding further, we recall the setting of stochastic
homogenization we work with. The reader familiar with this theory
may directly proceed to Section~\ref{sec:estim}.
Let $(\Omega, {\mathcal F}, \PP)$ be a probability
space. For a random variable $X\in L^1(\Omega, d\PP)$, we denote by
$\esp(X) = \int_\Omega X(\omega) d\PP(\omega)$ its expectation value. We
assume that the group $(\ZZ^d, +)$ acts on $\Omega$. We denote by
$(\tau_k)_{k\in \ZZ^d}$ this action, and assume that it preserves the
measure $\PP$, i.e. 
$$
  \forall k\in \ZZ^d, \quad \forall A \in {\cal F}, \quad \PP(\tau_k A)
  = \PP(A).
$$
We assume that $\tau$ is {\em ergodic}, that is,
$$
\forall A \in {\mathcal F}, \quad \left(\forall k \in \ZZ^d, \quad \tau_k A = A
    \right) \Rightarrow (\PP(A) = 0 \quad\mbox{or}\quad 1).
$$
We define the following notion of stationarity: any
$F\in L^1_{\rm loc}\left(\RR^d, L^1(\Omega)\right)$ is said to be
{\em stationary} if
\begin{equation}
  \label{eq:stationnarite-disc}
  \forall k\in \ZZ^d, \quad F(x+k, \omega) = F(x,\tau_k\omega)
  \mbox{ almost everywhere, almost surely}.
\end{equation}
Note that we have chosen to present the theory in 
a \emph{discrete} stationary setting, which is more
appropriate for our specific purpose, which is to consider a setting
close to {\em periodic} homogenization.
Random homogenization is more often presented in the \emph{continuous}
stationary setting. This is only a matter of small modifications. We
refer to the bibliography for the latter.  

\medskip

For the sake of analysis, we assume in this section that the matrix
$A^\varepsilon_\eta(x,\omega)$ in~\eqref{0-sto}-\eqref{eq:decompo}
reads $\dps A^\varepsilon_\eta(x,\omega) = A_\eta
\left(\frac{x}{\varepsilon},\omega\right)$, where the random matrix
$A_\eta$ is stationary in the sense
of~\eqref{eq:stationnarite-disc}. 
The problem~\eqref{0-sto} now reads
\begin{equation}
-\mbox{div} \left[A_\eta\left(\frac{\cdot}{\varepsilon},\omega\right)
  \nabla u^\varepsilon_\eta(\cdot,\omega) \right] = f \ \ \mbox{in
  $\mathcal{D}$},  
\quad
u^\varepsilon_\eta(\cdot,\omega) = 0 \ \ \mbox{ on $\partial \mathcal{D}$},
\label{PB:stoch-a}
\end{equation}
where $A_\eta(\cdot,\omega) \in (L^\infty(\RR^d))^{d \times
  d}$ satisfies the standard coercivity and boundedness
conditions: there exists two constants $a_+ \geq a_- > 0$ such that 
\begin{equation}
\label{eq:unif_eta}
\forall \eta, \ \forall \xi \in \RR^d, \quad 
a_- |\xi|^2 \leq A_\eta(x,\omega) \xi \cdot \xi 
\ \ \text{a.e. on $\RR^d$, a.s.}
\quad \text{and} \quad
\| A_\eta(\cdot ,\omega) \|_{L^\infty(\RR^d)} \leq a_+ \ \ \text{a.s.}
\end{equation}
Due to the stationarity assumption on $A_\eta$, the
problem~\eqref{PB:stoch-a} admits a homogenized limit when $\eps \to
0$. Note that, to the best of our knowledge, all analyses of the MsFEM
approach in the deterministic setting that have been proposed in the
literature are performed under a similar assumption (the matrix $A^\eps$
in~\eqref{0} is assumed to read $\dps A^\eps(x) =
A_{per}\left(\frac{x}{\varepsilon}\right)$ for a fixed periodic matrix
$A_{per}$, see e.g.~\cite{hou1999,Efendiev2000}).

In addition, in line with~\eqref{eq:decompo} and~\eqref{def:A-eta}, we
assume that $A_\eta$ is of the form
\begin{equation}
A_\eta(x,\omega)=A_{per}(x)+ \eta \: A_1(x,\omega),
\label{def:A-eta-1}
\end{equation}
where $\eta \in \RR$ is small parameter (we henceforth assume that
$|\eta| \leq 1$),
$A_{per}$ is a symmetric bounded $Q$-periodic matrix 
($Q=[0,1]^d$) 
satisfying the ellipticity condition almost everywhere on $\RR^d$, and
$A_1$ is a symmetric bounded stationary matrix: $|A_1(x,\omega)|\leq C$
almost everywhere in $\RR^d$, almost surely. Since $\eta$ is small, our
problem is {\em weakly} stochastic. 

In line with~\eqref{structA1}, we furthermore assume that $A_1$ is of
the form 
\begin{equation}
\label{struc:A1}
A_1(x,\omega) = \sum\limits_{k \in \ZZ^d} \mathbf{1}_{Q+k}(x)
X_k(\omega) \, B_{per}(x),
\end{equation}
where $\left(X_k(\omega)\right)_{k\in \ZZ^d}$ is a sequence of
i.i.d. scalar random variables such that 
$$
\exists C, \, \forall k \in \ZZ^d, \quad 
|X_k(\omega)| \leq C \quad \text{ almost surely,}
$$
and $B_{per} \in \left( L^\infty(\RR^d) \right)^{d \times d}$ is a
$Q$-periodic matrix. Besides being used in Theorem~\ref{exp2scaleueta}
below, this assumption is also used in the proof of
Lemma~\ref{coro:lambda}, to recognize that some quantity
(namely,~\eqref{eq:lundi2} below) is a normalized
sum of {\em i.i.d. variables}, on which we can use Central Limit Theorem
arguments. 
As mentioned in Section~\ref{sec:implem} above, the
form~\eqref{struc:A1} is not essential. The point
in~\eqref{struc:A1} is that  
$A_1$ is a sum of {\em direct products} of a function depending on $x$
with a random variable only depending on $\omega$.
Assumptions alternative to~\eqref{struc:A1} could be made, that still
satisfy this framework. 

Finally, we assume that
\begin{eqnarray}
\label{hyp:a_holder}
\text{$A_{per}$ is H\"older continuous},
\\
\label{hyp:b_holder}
\text{$B_{per}$ is H\"older continuous}.
\end{eqnarray}
We use these assumptions to obtain a rate of convergence of the
two-scale expansion of $u^\eps_\eta$ (see~\cite{rate} and
Theorem~\ref{exp2scaleueta} below), and hence
control the first term in the right-hand side of~\eqref{tutu}. 
Such assumptions are standard when proving convergence rates of
two-scale expansions (see e.g.~\cite[p.~28]{Jikov1994}). 
In turn, to control the second term 
in~\eqref{tutu} and the non-conforming error (the second term
in~\eqref{eq:lundi}), we do not need $B_{per}$ to be H\"older
continuous, and only use the fact that $A_{per}$ is H\"older
continuous (to obtain e.g. Lemmas~\ref{lem:chi},~\ref{Lem1},~\ref{Lem2}
and~\ref{LemChen}).  
The numerical examples that we have considered in Section~\ref{sec:num}
satisfy assumptions~\eqref{hyp:a_holder}-\eqref{hyp:b_holder} (remark that
assumption~\eqref{hyp:a_holder} is also satisfied in the numerical
examples considered in e.g.~\cite{Efendiev2009}). 

Note that we have assumed $A_{per}$ and $B_{per}$ to be symmetric only
for the sake of simplicity. The arguments used below carry over to the
non-symmetric case up to slight modifications. 

\subsection{Error estimation}
\label{sec:estim}

To bound the error between the reference solution $u^\eps_\eta$ and the
MsFEM solution $u_S$, we use in many instances that we work in a {\em
  weakly} stochastic homogenization setting.
We first recall in Section~\ref{sec:homo-eq} some results specific to
weakly stochastic homogenization. This setting also allows to state
rates of convergence for the two-scale expansion of $u^\eps_\eta$, as we
explain in Section~\ref{sec:2scale}. Our main result,
Theorem~\ref{theo:H1}, is given in
Section~\ref{sec:weak_stochastic-MsFEM}. 

\subsubsection{The homogenized equation}
\label{sec:homo-eq}

Under the conditions recalled above, it is known (see
e.g.~\cite{blp,Jikov1994}) that the solution 
$u^\varepsilon_\eta(\cdot,\omega)$
to~\eqref{PB:stoch-a} a.s. converges weakly in $H^1_0(\mathcal{D})$ as
$\varepsilon \rightarrow 0$ to the deterministic solution $u^\star_\eta$ of
the homogenized equation
\begin{equation}
-\mbox{div}\left[A_\eta^\star\nabla u^\star_\eta \right] =
f \ \ \mbox{in $\mathcal{D}$}, 
\quad
u^\star_\eta = 0 \ \ \mbox{ on $\partial \mathcal{D}$}.
\label{PB:homo}
\end{equation}
The homogenized matrix is given by
\begin{equation}
\left(A_\eta^\star\right)_{ij}=\esp\left(\int_{Q}  (e_i + \nabla
  w^\eta_{e_i}(y,\cdot))^T A_\eta(y,\cdot)(e_j + \nabla w^\eta_{e_j}(y,\cdot))
  \,dy\right), 
\label{Astareta}
\end{equation}
where, for any $p \in \RR^d$, $w^\eta_p$ is the unique (up to the
addition of a random constant) solution to the corrector problem 
\begin{equation}
\left\{
\begin{array}{ll}
\dps-\mbox{div} \left[A_\eta\left(\cdot,\omega\right)(p + \nabla
  w^\eta_p(\cdot,\omega))
\right] = 0 & \mbox{ in } \RR^d, 
\\ \noalign{\vskip 3pt}
\nabla w^\eta_p 
\mbox{ is stationary in the sense of~\eqref{eq:stationnarite-disc},} 
\\ \noalign{\vskip 3pt}
\dps \esp \left( \int_{Q}\nabla w^\eta_p(y,\cdot) \, dy \right)= 0.
\end{array}
\right.
\label{PB:correcteur}
\end{equation}
The variational problem associated with (\ref{PB:homo}) writes: find
$u_\eta^\star \in H^1_0(\mathcal{D})$ such that  
$$
\forall v \in H^1_0(\mathcal{D}), \quad 
{\cal A}_\eta^\star(u_\eta^\star,v)=b(v),
$$
where
\begin{equation}
{\cal A}^\star_\eta(u,v) = \int_\mathcal{D} \left(\nabla
  v(x)\right)^T A_\eta^\star \nabla u(x) \, dx 
\quad \mbox{and} \quad
b(v)=\int_\mathcal{D} f(x)v(x) \, dx 
\label{def:homo-bili}.
\end{equation}

As shown in~\cite{Blanc2010,Ronan2010}, in the weakly stochastic
setting, the homogenized matrix $A_\eta^\star$ can be expanded in terms
of a series in powers of $\eta$:
\begin{equation}
A^\star_\eta = A_{per}^\star + \eta A_1^\star + \eta^2 A_2^\star(\eta), 
\label{exp:Astareta}
\end{equation}
where $A_2^\star(\eta)$ is a deterministic matrix, that depends on
$\eta$ and is bounded as
$\eta\rightarrow 0$, and where, for any $1 \leq i,j \leq d$,
\begin{eqnarray}
(A_{per}^\star)_{ij} 
&=& 
\int_Q (e_i + \nabla \wper_{e_i})^T A_{per} (e_j + \nabla
\wper_{e_j}), 
\label{def:A0star}
\\
(A_1^\star)_{ij}     
&=& 
\int_Q (e_i + \nabla \wper_{e_i})^T \esp(A_1) (e_j + \nabla
\wper_{e_j}),
\label{def:A1star}
\end{eqnarray}
where, for any $p \in \RR^d$, $\wper_p$ is the unique (up to the
addition of a constant) solution 
to the deterministic corrector problem associated to the periodic matrix
$A_{per}$: 
\begin{equation}
\left\{
\begin{array}{ll}
\dps-\mbox{div} \left[A_{per}(p+\nabla \wper_p)\right] = 0, 
\\
w_p^0 \mbox{ is $Q$-periodic}.
\end{array}
\right.
\label{PB:w0}
\end{equation}
Under the assumption~\eqref{struc:A1}, we have $A_1^\star =
\esp(X_0) \, \overline{B}$, with
\begin{equation}
\label{def:overlineB}
\forall 1 \leq i,j \leq d, \quad
\overline{B}_{ij} = 
\int_Q (e_i + \nabla \wper_{e_i})^T B_{per} (e_j + \nabla \wper_{e_j}).
\end{equation}

\begin{remark}
In general, when $A_{per}$ is not symmetric, the expression of
$A^\star_1$ includes additional terms. Indeed, writing $\nabla w^\eta_p =
\nabla \wper_p + \eta \nabla w^1_p + O(\eta^2)$, we in general need
$\esp(\nabla w^1_p)$ to compute $A^\star_1$ (see
e.g.~\cite{Ronan2010,Blanc2010}). In the symmetric case, these
additional terms vanish, see e.g.~\cite[Remark 4.2
p.~117]{these-arnaud}. In the non-symmetric case, the
expression~\eqref{def:A1star} of 
$A^\star_1$ needs to be slightly modified, but the
expansion~\eqref{exp:Astareta} remains true. Our arguments hence carry
over to the non-symmetric case.
\end{remark}

Using the expansion~\eqref{exp:Astareta} of $A_\eta^\star$ with respect
to $\eta$, it is easy to see that the solution $u_\eta^\star$
to~\eqref{PB:homo} can also be expanded in a series in powers of $\eta$.
We have 
\begin{equation}
\label{exp:ustareta}
u_\eta^\star = u_0^\star + \eta \esp(X_0) \overline{u}_1^\star + \eta^2
r_\eta
\quad \text{with} \quad
\|r_\eta\|_{H^1(\mathcal{D})} \leq C,
\end{equation}
where $C$ is a constant independent of $\eta$, and where 
$u_0^\star$ solves
\begin{equation}
\label{eq:homog-0}
-\hbox{div}\left[A^\star_{per} \nabla u_0^\star \right] = f
\ \ \text{in $\mathcal{D}$}, 
\quad 
u_0^\star = 0 \ \ \text{on $\partial \mathcal{D}$},
\end{equation}
and $\overline{u}_1^\star$ solves
\begin{equation}
\label{PB:u1barstar}
-\hbox{div}\left[A^\star_{per} \nabla \overline{u}_1^\star \right]
= 
\hbox{div}\left[\overline{B} \nabla u_0^\star \right] 
\ \ \text{in $\mathcal{D}$}, 
\quad 
\overline{u}_1^\star=0 \ \ \text{ on $\partial\mathcal{D}$}.
\end{equation}
The expansion~\eqref{exp:ustareta} will be
useful in the sequel. We will also need a bound on $u_\eta^\star$ and
$r_\eta$ in the $H^2$ norm. Recall that $u_\eta^\star$ is the solution
to~\eqref{PB:homo}, whereas $r_\eta$ is solution to
\begin{equation}
\label{PB:r_eta}
-\hbox{div}\left[A^\star_\eta \nabla r_\eta \right]
= 
\hbox{div}\left[ A^\star_2(\eta) \left( \nabla u_0^\star 
+ \eta \esp(X_0) \nabla \overline{u}_1^\star 
\right) +  
\esp(X_0) A^\star_1 \nabla \overline{u}_1^\star 
\right] 
\ \ \text{ in $\mathcal{D}$}, 
\quad
r_\eta =0 \ \ \text{ on $\partial\mathcal{D}$}.
\end{equation}
In view of~\eqref{eq:unif_eta}, we have, almost surely and almost
everywhere, 
$a_- \ \text{Id} \leq A_\eta \leq a_+ \ \text{Id}$ in the sense of symmetric
matrices. Recalling that homogenization preserves the order of symmetric
matrices (see e.g.~\cite[page~12]{composites}), we deduce that
$$
\forall \eta, \ \forall \xi \in \RR^d, \quad 
a_- |\xi|^2 \leq A^\star_\eta \xi \cdot \xi 
\leq a_+ |\xi|^2.
$$
In addition, the right-hand sides of~\eqref{PB:homo}
and~\eqref{PB:r_eta} are bounded 
uniformly in $\eta$ in the $L^2$ norm. 
Using~\cite[Theorems~9.15 and~9.14]{gilbarg-trudinger},
we obtain that there exists $C$ such that
\begin{equation}
\label{eq:bound_reta_H2}
\forall \eta, \quad 
\| u^\star_\eta\|_{H^2(\mathcal{D})} \leq C
\quad \text{and} \quad
\| r_\eta\|_{H^2(\mathcal{D})} \leq C.
\end{equation}
 
\subsubsection{Two scale expansion of the reference solution $u_\eta^\varepsilon$}
\label{sec:2scale}

As recalled above, the standard error analysis for the MsFEM in the
deterministic setting is performed in the case when the matrix $A^\eps$
in~\eqref{0} reads $A^\eps(x) \equiv A_{per}(x/\eps)$ for a fixed
periodic matrix $A_{per}$. The problem~\eqref{0} then admits a
homogenized limit. To obtain bounds on the MsFEM error, one step of the
proof is to approximate the oscillatory solution $u^\varepsilon$ by its
two-scale expansion  $\dps u^\star + \varepsilon \sum_{i=1}^d 
\wper_{e_i} \left( \frac{\cdot}{\varepsilon} \right) \partial_i u^\star$,
where $u^\star$ is the homogenized solution, $\wper_p$ is the
periodic corrector associated to $p \in \RR^d$, and $\dps \partial_i
u^\star = \frac{\partial u^\star}{\partial x_i}$. In the deterministic
case, it is known (see e.g.~\cite{blp,cd,Jikov1994}) that, under some
regularity assumptions on $A_{per}$ and $u^\star$,  
\begin{equation}
\label{eq:rate}
\left\| u^\varepsilon - \left[ 
u^\star + \varepsilon \sum_ {i=1}^d 
\wper_{e_i} \left( \frac{\cdot}{\varepsilon} \right) \partial_i u^\star 
\right] \right\|_{H^1(\mathcal{D})} \leq C \sqrt{\varepsilon}
\end{equation}
for a constant $C$ independent of $\varepsilon$. 

In the stochastic case, it is known
that $\dps \esp \left[ \left\| u^\varepsilon - \left[ 
u^\star + \varepsilon \sum_{i=1}^d 
w_{e_i} \left( \frac{\cdot}{\varepsilon},\omega \right) \partial_i u^\star 
\right] \right\|^2_{H^1(\mathcal{D})} \right]$ converges to 0 as $\varepsilon \to
0$ (see~\cite[Theorem~3]{Papanicolaou-Varadhan}), but no rate of
convergence is known (except in some one-dimensional situations, see
e.g.~\cite{bal_residu,bourgeat_residu,residu-bll}).  
However, in the present article, and as announced above, we consider a
{\em weakly} stochastic 
case. In this setting, we have derived in~\cite{rate} a result similar
to~\eqref{eq:rate}. We now state this result, which will
be useful for our analysis.

\begin{theorem}[from~\cite{rate}, Theorem 2]
\label{exp2scaleueta}
Assume $d>1$.
Let $u_\eta^\varepsilon$ be the solution to~\eqref{PB:stoch-a}, and
assume that $A_\eta$
satisfies~\eqref{def:A-eta-1}-\eqref{struc:A1}-\eqref{hyp:a_holder}-\eqref{hyp:b_holder}.  
Let $A^\star_{per}$, $\wper_p$, $u_0^\star$ and $\overline{u}_1^\star$
be defined by~\eqref{def:A0star}, \eqref{PB:w0},~\eqref{eq:homog-0}
and~\eqref{PB:u1barstar}.
The two-scale expansion $v^\varepsilon_\eta$ of $u_\eta^\varepsilon$
reads
\begin{multline}
\label{expueta}
v^\varepsilon_\eta(\cdot,\omega) = u_0^\star + \eta \esp(X_0)
\overline{u}_1^\star + \varepsilon \sum\limits_{i=1}^d
\left[ \wper_{e_i}\left(\frac{\cdot}{\varepsilon}\right) (\partial_i u_0^\star
  + \eta \esp(X_0) \partial_i \overline{u}_1^\star) \right. 
\\
\left.+ \eta \esp(X_0) \psi_{e_i} \left(\frac{\cdot}{\varepsilon}\right)
  \partial_i u_0^\star + \eta \sum\limits_{k \in I_\varepsilon}
  (X_k(\omega)-\esp(X_0))
\ \chi_{e_i}\left(\frac{\cdot}{\varepsilon}-k\right)\partial_i u_0^\star
\right],
\end{multline}
where
$$
I_\eps = \left\{ k \in \ZZ^d \text{ such that } \varepsilon(Q+k)
  \cap \mathcal{D} \ne \emptyset \right\}, 
$$
and where, for any $p \in \RR^d$, $\psi_p$ is the solution (unique up to the
addition of a constant) to
\begin{equation}\label{def:psip}
\left\{ 
\begin{array}{l}
-\hbox{div}\left[A_{per} \nabla \psi_p \right]
= 
\hbox{div}\left[B_{per} \left(p + \nabla \wper_p\right) \right],
\\
\psi_p \ \text{is $Q$-periodic},
\end{array}
\right.
\end{equation}
and $\chi_p$ is the unique solution to
\begin{equation}
\label{eq:def_chi}
 \left\{ 
\begin{array}{l l}
-\hbox{div}\left[A_{per} \nabla \chi_{p} \right]=
\hbox{div}\left[\mathbf{1}_{Q} B_{per}(p + \nabla \wper_p)\right] & \text{
  in $\RR^d$,} 
\\
\chi_{p} \in L^2_{loc}(\RR^d), \quad 
\nabla \chi_{p} \in \left(L^2(\RR^d)\right)^d,
\\
\dps \lim_{|x| \to \infty} \chi_p(x) = 0.
\end{array}
\right.
\end{equation}
We assume that $u_0^\star \in
W^{2,\infty}(\mathcal{D})$ and
$\overline{u}_1^\star \in W^{2,\infty}(\mathcal{D})$.
Then
\begin{equation}
\sqrt{ \esp \left[ 
\|u_\eta^\varepsilon-v^\varepsilon_\eta\|^2_{H^1(\mathcal{D})} 
\right] }
\leq C \left(\sqrt{\varepsilon} + \eta \sqrt{\varepsilon
    \ln(1/\varepsilon)} + \eta^2 \right), 
\label{restueta}
\end{equation}
where $C$ is a constant independent of $\varepsilon$ and $\eta$.
\end{theorem}
As pointed out above, and in~\cite{rate}, the
assumptions~\eqref{hyp:a_holder}-\eqref{hyp:b_holder} are standard
assumptions when proving convergence rates of two-scale expansions 
(see e.g.~\cite[p.~28]{Jikov1994}). Likewise, the assumption $u_0^\star
\in W^{2,\infty}(\mathcal{D})$ (and subsequently
$\overline{u}_1^\star \in W^{2,\infty}(\mathcal{D})$) is a
standard assumption (see e.g.~\cite[Theorem
2.1]{allaire-amar} and~\cite[p.~28]{Jikov1994}).
In view of~\eqref{eq:homog-0}, this assumption implies that
the right hand side $f$ in~\eqref{PB:stoch-a} belongs to
$L^\infty(\mathcal{D})$. 

In dimension $d=1$, the boundary conditions of~\eqref{eq:def_chi} need
to be modified for this problem to have a solution. We have derived
in~\cite{rate} the following result, which is the one-dimensional
version of Theorem~\ref{exp2scaleueta} (note that we need below weaker
assumptions than in Theorem~\ref{exp2scaleueta}, as pointed out
in~\cite{rate}: we do not need to
assume~\eqref{hyp:a_holder}-\eqref{hyp:b_holder}, and the assumption $f
\in L^2({\cal D})$ is enough): 

\begin{theorem}[from~\cite{rate}, Theorem 3]
\label{exp2scaleueta-1D}
Assume that the dimension $d$ is equal to one.
Let $u_\eta^\varepsilon$ be the solution to~\eqref{PB:stoch-a} in the
domain ${\cal D}$ with $f \in L^2({\cal D})$, and assume that $A_\eta$
satisfies~\eqref{def:A-eta-1}-\eqref{struc:A1}.
Let $v^\varepsilon_\eta$ be defined by~\eqref{expueta}, where the
definition~\eqref{eq:def_chi} is replaced by
\begin{equation}
\label{eq:def_chi-1d}
\left\{ 
\begin{array}{l l}
-\left[ A_{per} \chi' \right]' =
\left[\mathbf{1}_{(0,1)} B_{per}(1 + (\wper)') \right]' & \text{
  in $\RR$,} 
\\
\chi \in L^2_{loc}(\RR), \quad 
\chi' \in L^2(\RR),
\end{array}
\right.
\end{equation}
where $\wper$ solves~\eqref{PB:w0}.
Then
\begin{equation}
\label{restueta-1d-H1-Linfty}
\sqrt{ \esp \left[ 
\|u^\varepsilon_\eta-v^\varepsilon_\eta\|^2_{L^\infty(\mathcal{D})} 
\right] }
+
\sqrt{ \esp \left[ 
\|u^\varepsilon_\eta-v^\varepsilon_\eta\|^2_{H^1(\mathcal{D})} 
\right] }
\leq 
C \left( \varepsilon + \eta \sqrt{\varepsilon} + \eta^2 \right),
\end{equation}
where $C$ is a constant independent of $\varepsilon$ and $\eta$.
\end{theorem}

The following estimate, which is proved in~\cite[proof of Proposition
11]{rate} and useful to demonstrate~\eqref{restueta}, will also be
useful here: 

\begin{lemme}[from~\cite{rate}, proof of Proposition 11]
\label{lem:chi}
We assume~\eqref{hyp:a_holder} and $d>1$.
For any $p \in \RR^d$, any $k \in \ZZ^d$, and any bounded domain ${\cal
  D} \subset \RR^d$, the function $\chi_p$ defined
by~\eqref{eq:def_chi} satisfies 
\begin{equation}
\label{eq:bound_chi}
\left\|
\chi_p\left(\frac{\cdot}{\varepsilon}-k\right)\right\|_{L^2(\mathcal{D})}^2 
\leq
C \varepsilon^d R_{d,\eps},
\end{equation}
for a constant $C$ independent of $k$ and $\eps$, where $R_{d,\eps} = 1$ if
$d>2$, and $R_{d,\eps} = 1+\ln(1/\varepsilon)$ if $d=2$.
\end{lemme}

\subsubsection{Main result}
\label{sec:weak_stochastic-MsFEM}

Before presenting our main result, we need some useful notation.
Following the approach presented in Section~\ref{sec:weak-stoch}, we
recall that 
$$
\mathcal{W}_h := \mbox{span}(\phi_i^\eps, \ i=1,\cdots,L),
$$
where $\phi^{\varepsilon}_i$ are the highly oscillatory MsFEM basis
functions. By construction, the solution $u_S \in \mathcal{W}_h$ 
of the weak stochastic MsFEM approach~\eqref{eq:var-mac-stoch} satisfies
\begin{equation}
\forall v_h \in \mathcal{W}_h, \quad 
\mathcal{A}^h_{\varepsilon,\eta}(u_S,v_h)=b(v_h) \quad \text{a.s.}
\label{eq:var-mac}
\end{equation}
where, for any $u$ and $v$ in $\mathcal{W}_h$, 
\begin{equation}
\mathcal{A}^h_{\varepsilon,\eta}(u,v) =
\sum\limits_{\cell\in\mathcal{T}_h} \int_\cell 
\left( \nabla v(x) \right)^T A_\eta\left(\frac{x}{\varepsilon},\omega\right)
\nabla u(x) \, dx 
\quad \text{and} \quad
b(v)= \int_{\mathcal D} f(x) v(x) \, dx. 
\label{def:stoch-bili}
\end{equation}
For future use, we also define, on the standard finite element space
$$
\mathcal{V}_h := \mbox{span}(\phi_i^{0}, \ i=1,\cdots,L),
$$
the forms
\begin{equation}
\widetilde{\cal A}^h_{\varepsilon,\eta}(u,v) =
\sum\limits_{\cell\in\mathcal{T}_h} \int_\cell
\left(\nabla\left(\Rop_\cell(v)\right)(x)\right)^T 
A_\eta\left(\frac{x}{\varepsilon},\omega\right)
\nabla\left(\Rop_\cell(u)\right)(x) \, dx 
\quad \mbox {and} \quad
\widetilde{b}_h(v)=\sum\limits_{\cell\in\mathcal{T}_h}\int_\cell\! \! \!
f(x) \Rop_\cell(v)(x) \, dx,
\label{def:stoch-bili-tilde}
\end{equation}
where the local, linear operators $\Rop_\cell$ are defined on
$\mathcal{V}_h$ by
\begin{equation}
\label{eq:R_local}
\forall 1 \leq i \leq L, \quad 
\Rop_\cell(\left. \phi^{0}_i \right|_\cell)=
\left. \phi^{\varepsilon}_i \right|_\cell.
\end{equation}
These local operators give rise to the global operator 
$\Rop : \mathcal{V}_h \to \mathcal{W}_h$ defined by
\begin{equation}
\label{eq:R_global}
\forall \cell, \quad
\forall v \in \mathcal{V}_h, \quad
\left. \Rop(v) \right|_{\cell} = 
\Rop_\cell\left(\left. v \right|_{\cell} \right).
\end{equation}

As pointed out above, the space $\mathcal{W}_h$ is not a subspace of
$H^1_0(\mathcal{D})$, as the basis functions $\phi_i^\varepsilon$ may
have jumps at the finite element boundaries (due to the use of the
oversampling technique). We will hence work with the
broken $H^1$-norm introduced in~\eqref{eq:broken_H1}, that reads, we
recall, 
$$
\forall v_h \in \mathcal{W}_h, \quad 
\|v_h\|_\h = \left[ \sum_{\cell \in \mathcal{T}_h}
  \|v_h\|^2_{H^1(\cell)} \right]^{1/2}.
$$
We are now in position to present the main result of this article. We
introduce the notation $Q_i^\varepsilon=\varepsilon (i + Q)$ for any $i
\in \ZZ^d$, and denote by $N_\cell$ the number of cells
$Q_i^\varepsilon$ in the element $\cell$: 
$\displaystyle N_\cell = \hbox{Card}(i;Q_i^\varepsilon \subset
\cell)$. We make in the theorem below a regularity hypothese on the
macroscopic mesh, assuming that the volume of each element is bounded
from below by $\alpha h^d$, for some $\alpha>0$, and hence that
$N_\cell \geq \alpha \left( h/\eps \right)^d$.
\begin{theorem}
\label{theo:H1}
Assume that $A_\eta$
satisfies~\eqref{def:A-eta-1}-\eqref{struc:A1}-\eqref{hyp:a_holder}-\eqref{hyp:b_holder}.
We assume that $u_0^\star$ and $\overline{u}_1^\star$ respectively defined
by~\eqref{eq:homog-0} and~\eqref{PB:u1barstar} satisfy
$u_0^\star \in W^{2,\infty}(\mathcal{D})$ and
$\overline{u}_1^\star \in W^{2,\infty}(\mathcal{D})$.
Let $u_\eta^\varepsilon$ be
the solution to~\eqref{PB:stoch-a} and $u_S$ be the weakly stochastic
MsFEM solution to~\eqref{eq:var-mac}. Suppose that $d>1$, 
$\eps \leq h$, and that there exists $\alpha>0$, independent of $\cell$,
$h$ and $\varepsilon$, such that $\displaystyle N_\cell \geq \alpha \left(
\frac{h}{\varepsilon} \right)^d$. We then have 
\begin{equation}
\label{est-H1}
\sqrt{\esp \left[ \|u^\varepsilon_\eta - u_S\|^2_\h \right] } \leq C 
\left( \sqrt{\eps} + h + \frac{\varepsilon}{h} + \eta \left(
    \frac{\varepsilon}{h} \right)^{d/2} \ln(N(h)) 
+ \eta + \eta^2 \mathcal{C}(\eta) \right),
\end{equation}
where $C$ is a constant independent of $\varepsilon$, $h$ and $\eta$,
$N(h)$ is the number of elements $\cell$ in the domain ${\cal D}$ (which is of
order $h^{-d}$ in dimension $d$), 
and $\mathcal{C}$ is a bounded function as $\eta$ goes to $0$. 
\end{theorem}
The restriction to $d>1$ comes from the fact that the proof of this
result uses the rate of convergence on the two-scale expansion of
$u^\varepsilon_\eta$ that we stated in Theorem~\ref{exp2scaleueta}. This
rate of convergence is not optimal in dimension one, as can be seen from
the comparison of~\eqref{restueta} and~\eqref{restueta-1d-H1-Linfty}.
The one-dimensional version of the above result is stated in
Section~\ref{sec:1d} below (see Theorem~\ref{theo:H1-1d}), where we
briefly consider the one-dimensional situation.

\begin{remark}
\label{rem:optim}
In the case $\eta=0$, our approach reduces to the standard deterministic
MsFEM and we
obtain the same estimate as in the deterministic case with oversampling
(see e.g.~\cite[Theorem 3.1]{Efendiev2000}). 
\end{remark}

\subsection{Proof of Theorem~\ref{theo:H1}}
\label{sec:proof_main}

The proof of Theorem~\ref{theo:H1} is the direct consequence of three
lemmas. First we recall the second Strang's lemma~\cite[Theorem~4.2.2, p.~210]{Ciarlet}. 
\begin{lemme}
\label{lem:Strang2}
Consider a family of Hilbert spaces $\mathcal{W}_h$ with the norm $\|
\cdot \|_\h$, a family of continuous bilinear forms 
${\cal A}^h_{\varepsilon,\eta}$
on $\mathcal{W}_h$ that are uniformly 
$\mathcal{W}_h$-elliptic, and a continuous linear form $b$ on
$\mathcal{W}_h$. For any $h>0$, introduce $u_S$ solution to 
$$
\forall v_h \in \mathcal{W}_h, \quad 
{\cal A}^h_{\varepsilon,\eta}(u_S,v_h) = b(v_h)
$$
and $u_\eta^\varepsilon \in H^1_0(\mathcal{D})$ solution to 
$$
\forall v \in H^1_0(\mathcal{D}), \quad 
{\cal A}^h_{\varepsilon,\eta}(u^\varepsilon_\eta,v) = b(v).
$$
Then there exists a constant $C$ independent of $\eta$, $h$ and $\varepsilon$
such that
\begin{equation}
\|u^\varepsilon_\eta - u_S\|_\h
\leq C \left(\inf\limits_{v_h \in \mathcal{W}_h} 
\|u^\varepsilon_\eta - v_h\|_\h + \sup\limits_{w_h \in \mathcal{W}_h}
\frac{\left|{\cal A}^h_{\varepsilon,\eta}(u^\varepsilon_\eta,w_h)- b(w_h)\right|}{\|w_h\|_\h} 
\right). 
\label{Strang2}
\end{equation}
\end{lemme}

\medskip

The first term in the right hand side of~\eqref{Strang2} is the
so-called best approximation error. The main part (step 2) of the proof of
Theorem~\ref{theo:H1} is devoted to its estimation, following up on the
estimate~\eqref{restueta} provided by Theorem~\ref{exp2scaleueta}. 

The second term in the right hand side of~\eqref{Strang2} is the
so-called nonconforming error, which vanishes in the case $\mathcal{W}_h
\subset H^1_0(\mathcal{D})$ (the method is then conforming, and we are
left with the standard C\'ea lemma). In our case, we use the
oversampling technique, hence our approximation is not conforming, and
this second term does not vanish. It will be estimated in the step 3 of
the proof of Theorem~\ref{theo:H1}, using the
following two results, which are proved in
Appendix~\ref{sec:tech-proof}.

\begin{lemme}
\label{Lem1}
Consider the two bilinear forms ${\cal A}^\star_\eta$ and $\widetilde{\cal
  A}^h_{\varepsilon,\eta}$ respectively defined in 
(\ref{def:homo-bili}) and (\ref{def:stoch-bili-tilde}). 
Under assumption~\eqref{hyp:a_holder}, there 
exists a deterministic constant $C$, independent of $\eta$, $\varepsilon$ and
$h$, such that, for any $v_h \in \mathcal{V}_h$, 
\begin{equation}
\sup\limits_{w_h \in \mathcal{V}_h} 
\frac{\left|\widetilde{\cal A}^h_{\varepsilon,\eta}(v_h,w_h)- {\cal
    A}^\star_\eta(v_h,w_h)\right|}{\|w_h\|_{H^1(\mathcal{D})}} \leq C
\left(\frac{\varepsilon}{h}+\eta \lambda(\omega,h,\varepsilon) + \eta^2
  \mathcal{C}(\eta)\right) \|v_h\|_{H^1(\mathcal{D})} \quad \text{a.s.}, 
\label{lemme1}
 \end{equation}
where $\mathcal{C}$ is a deterministic function independent of $\eps$
and $h$ and bounded when $\eta
\rightarrow 0$, and $\lambda$ is defined by
\begin{equation}
\lambda(\omega,h,\varepsilon)=\max\limits_{\cell} \max\limits_{1\leq p,m
  \leq d}\left|\frac{1}{|I^\varepsilon_\cell|}
  \int_{I^\varepsilon_\cell}\left[e_p + \nabla
    \wper_{e_p}\left(\frac{x}{\varepsilon}\right)\right]^T\left(
    A_1\left(\frac{x}{\varepsilon},\omega\right)-\esp\left(A_1\left(\frac{x}{\varepsilon},\cdot\right)\right)\right)\left[e_m + \nabla \wper_{e_m}\left(\frac{x}{\varepsilon}\right)\right] \, dx\right|,
\label{def:sigN}
\end{equation}
where $I^\varepsilon_\cell$ is the largest domain composed of cells of
size $\varepsilon$ included in $\cell$: 
$$
I^\varepsilon_\cell=
\bigcup\limits_{Q_i^\varepsilon \subset \cell} Q_i^\varepsilon,
\quad
Q_i^\varepsilon=\varepsilon (i + Q), \ i \in \ZZ^d.
$$
\end{lemme}

\begin{lemme}
\label{Lem2}
Consider the two linear forms $b$ and $\widetilde{b}_h$ respectively
defined in (\ref{def:homo-bili}) and (\ref{def:stoch-bili-tilde}). 
Under assumption~\eqref{hyp:a_holder}, there 
exists a deterministic constant $C$ independent of $\eta$, $\varepsilon$
and $h$ such that 
\begin{equation}
\sup\limits_{w_h \in \mathcal{V}_h}
\frac{\left|\widetilde{b}_h(w_h)-b(w_h)\right|}{\|w_h\|_{H^1(\mathcal{D})}} \leq  
C \varepsilon \|f\|_{L^2(\mathcal{D})}.
\label{lemme2}
\end{equation}
\end{lemme}

Before turning to the proof of Theorem~\ref{theo:H1}, we first give some
properties of the random variable $\lambda(\omega,h,\eps)$ that appears
in the right hand side of~\eqref{lemme1}, and we next detail a two
scale expansion of the highly oscillatory basis functions 
$\phi^{\varepsilon}_i$, which will be useful in the sequel. 

\begin{remark}
\label{rem:coerc}
We will show in Lemma~\ref{coro:lambda} below that $\lambda$ defined
in~\eqref{def:sigN} is 
uniformly bounded with respect to $h$, $\eps$ and $\omega$.
Since ${\cal A}^\star_\eta$ is coercive, we deduce from~\eqref{lemme1} that
$\widetilde{\cal A}^h_{\varepsilon,\eta}$ is also coercive, in the sense that
there exists a 
deterministic constant $\alpha>0$, independent of $h$, $\varepsilon$ and
$\eta$, such that 
$$
\forall v_h \in \mathcal{V}_h, \quad 
\alpha \|v_h\|^2_{H^1(\mathcal{D})} \leq 
\widetilde{\cal A}^h_{\varepsilon,\eta}(v_h,v_h).
$$
\end{remark}

\subsubsection{Properties of $\lambda(\omega,h,\eps)$}

We state here some useful properties of the random variable
$\lambda(\omega,h,\eps)$ that appears in~\eqref{lemme1}. They will be
proved in Appendix~\ref{sec:tech-proof}. 
As mentioned above, we recall that the assumption $N_\cell \geq \alpha \left(
  h/\varepsilon \right)^d$ that we make below is a regularity assumption
on the macroscopic mesh (the volume of each element $\cell$ is bounded from
below by $\alpha h^d$). 

\begin{lemme}
\label{coro:lambda}
Let $\lambda(\omega,h,\eps)$ be defined by~\eqref{def:sigN}. Then,
there exists a deterministic constant $C$ such that, for any $h$ and
$\eps$, we have $0 \leq \lambda(\omega,h,\varepsilon) \leq C$ almost
surely. 

Assume now that the random matrix $A_1$ 
satisfies~\eqref{struc:A1}, where the law of $X_k(\omega)$ is absolutely
continuous with respect to the Lebesgue measure. Assume furthermore that
the number 
$\displaystyle N_\cell=\hbox{Card}(i;Q_i^\varepsilon \subset \cell)$ of
cells in $\cell$ satisfies $\displaystyle N_\cell \geq \alpha \left(
\frac{h}{\varepsilon} \right)^d$ for some $\alpha>0$ independent of the element
$\cell$, $h$ and $\varepsilon$. Then
\begin{equation}
\label{eq:resu4a}
\esp(\lambda(\cdot,h,\varepsilon)^2)\leq C
\frac{\varepsilon^d}{h^d}\left[\ln(N(h))\right]^2,
\end{equation}
where, we recall, $N(h)$ is the number of elements $\cell$ in the domain
${\cal D}$ (which is of order $h^{-d}$ in dimension $d$) and 
$C$ is a deterministic constant independent of $h$ and $\varepsilon$.
\end{lemme}
Because of the specific form~(\ref{struc:A1}) of $A_1$, we will see in
the proof of that result (see Appendix~\ref{sec:tech-proof} below) that
\begin{equation}
\label{titi}
\lambda(\omega,h,\varepsilon)
=
\max\limits_\cell \max\limits_{1 \leq m,p \leq d} \left| S^{m,p}_\cell \right|,
\end{equation}
where each random variable $S^{m,p}_\cell$ is a normalized sum of
$(h/\eps)^d$ i.i.d. variables. Applying the Central Limit Theorem, we
hence know that 
$S^{m,p}_\cell$ converges, when $\eps \to 0$, to a Gaussian random
variable (up to an appropriate renormalization). Likewise, computing
the expectation of $\left( S^{m,p}_\cell \right)^2$ is not
difficult. However, in the above 
lemma, the difficulty stems from the fact that
$\lambda(\omega,h,\varepsilon)$ is the maximum 
of {\em many} such random variables $S^{m,p}_\cell$ (in~\eqref{titi},
the number of elements $\cell$ is indeed of the order of $h^{-d}$). Our
main task is 
hence to control how  $\esp(\lambda(\cdot,h,\varepsilon)^2)$ depends on
$h$. See also Remark~\ref{rem:mittal}. 

\subsubsection{Two-scale expansion of the highly oscillatory basis
  functions}

Following~\cite{Efendiev2000}, we recall here an expansion of
$\phi_i^{\varepsilon,\cell}$ that will be useful in the sequel. 
By definition (see (\ref{eq:MsFEM-Over-nb-1}) and
(\ref{eq:MsFEM-Over-nb-2})), we have, for any $1 \leq i \leq d+1$, 
\begin{equation}
\label{rocq1}
\phi_i^{\varepsilon,\cell} = \sum\limits_{j=1}^{d+1} \alpha_{ij}
\left. \chi^{\varepsilon,\mathbf{S}}_j \right|_\cell, 
\end{equation}
where $\alpha_{ij}$ is such that 
\begin{equation}
\label{rocq2}
\phi_i^{0,\cell} = \sum\limits_{j=1}^{d+1}
\alpha_{ij} \left. \chi^{0,\mathbf{S}}_j \right|_\cell.
\end{equation}
We thus first turn to $\chi^{\varepsilon,\mathbf{S}}_i$, which, by
definition (see~\eqref{PB:MsFEM-Over-nb}), is the unique solution to 
\begin{equation}
-\mbox{div}\left[A_{per}\left(\frac{x}{\varepsilon}\right)\nabla
  \chi^{\varepsilon,\mathbf{S}}_i(x) \right] = 0 \ \ \mbox{in $\mathbf{S}$}, 
\quad
\chi^{\varepsilon,\mathbf{S}}_i(x)=\chi^{0,\mathbf{S}}_i(x) \ \ 
\mbox{on $\partial \mathbf{S}$}. 
\label{PB:MsFEM-nb-3} 
\end{equation}
We introduce the function 
\begin{equation}
\theta_i^{\varepsilon,\mathbf{S}}(x)= \varepsilon^{-1}\left( \chi^{0,\mathbf{S}}_i(x) +
\varepsilon \sum\limits_{j=1}^d \wper_{e_j}
\left(\frac{x}{\varepsilon}\right) \partial_j \chi^{0,\mathbf{S}}_i(x) -\chi^{\varepsilon,\mathbf{S}}_i(x)\right),
\label{dev-phi-eps}
\end{equation}
where $\wper_{e_i}$ is solution to the periodic corrector problem (\ref{PB:w0}). By construction,
using~\eqref{dev-phi-eps},~\eqref{PB:w0}, \eqref{PB:MsFEM-nb-3} and the
fact that $\nabla \chi^{0,\mathbf{S}}_i$ is constant on $\mathbf{S}$, we
have
\begin{eqnarray*}
-\mbox{div}\left[A_{per}\left(\frac{x}{\varepsilon}\right)\nabla
  \theta^{\varepsilon,\mathbf{S}}_i(x) \right] 
&=& 
\frac{1}{\varepsilon}\mbox{div}\left[A_{per}\left(\frac{x}{\varepsilon}\right)\nabla
  \left(\chi^{\varepsilon,\mathbf{S}}_i(x) - \chi^{0,\mathbf{S}}_i(x) -
    \varepsilon \sum\limits_{j=1}^d \wper_{e_j}
    \left(\frac{x}{\varepsilon}\right) \partial_j \chi^{0,\mathbf{S}}_i
  \right)\right] 
\\
&=& 
\frac{1}{\varepsilon}\mbox{div}\left[A_{per}\left(\frac{x}{\varepsilon}\right)\nabla
  \chi^{\varepsilon,\mathbf{S}}_i(x)\right] -
\frac{1}{\varepsilon}\sum\limits_{j=1}^d \partial_j
\chi^{0,\mathbf{S}}_i
\mbox{div}\left[A_{per}\left(\frac{x}{\varepsilon}\right)\left(e_j +
    \nabla \wper_{e_j} \left(\frac{x}{\varepsilon}\right)\right)\right] 
\\
&=&0,
\end{eqnarray*}
while, from~\eqref{dev-phi-eps}, $\dps \theta_i^{\varepsilon,\mathbf{S}}(x)=\sum\limits_{j=1}^d
\partial_j \chi^{0,\mathbf{S}}_i(x) \wper_{e_j} \left(\frac{x}{\varepsilon}\right) $ on $\partial \mathbf{S}$. So, by linearity, we obtain  
\begin{equation}
\theta_i^{\varepsilon,\mathbf{S}}(x)= \sum\limits_{j=1}^d \partial_j
\chi^{0,\mathbf{S}}_i \ \xi_\varepsilon^j(x), 
\label{eq:thetaxi}
\end{equation}
where $\xi_\varepsilon^j \in H^1(\mathbf{S})$ is the unique solution to
\begin{equation}
-\mbox{div}\left[A_{per}\left(\frac{x}{\varepsilon}\right)\nabla
  \xi_\varepsilon^j(x) \right] = 0 \ \ \mbox{in $\mathbf{S}$}, 
\quad
\xi_\varepsilon^j(x) = \wper_{e_j}
\left(\frac{x}{\varepsilon}\right) \ \ \mbox{on $\partial \mathbf{S}$}. 
\label{PB:xip} 
\end{equation}

Using~\eqref{dev-phi-eps}, we now obtain a useful relation between
$\phi_i^{\varepsilon,\cell}$ and $\phi_i^{0,\cell}$. Indeed,
collecting~\eqref{rocq1},~\eqref{rocq2},~\eqref{dev-phi-eps}
and~\eqref{eq:thetaxi}, we obtain the {\em exact} expression
\begin{equation}
\phi^{\varepsilon,\cell}_i(x) 
= 
\phi^{0,\cell}_i(x) + \varepsilon \sum\limits_{j=1}^d
\left(\wper_{e_j}\left(\frac{x}{\varepsilon}\right) -
 \left. \xi_\varepsilon^j(x) \right|_\cell \right) \partial_j \phi^{0,\cell}_i.
\label{phiexp}
\end{equation}
Recall now that $\phi^{\varepsilon,\cell}_i(x) = 
\Rop_\cell(\phi^{0,\cell}_i)$, by definition of the local operator
$\Rop_\cell$ (see~\eqref{eq:R_local}). Correspondingly,
the global operator $\Rop$, defined on $\mathcal{V}_h$
by~\eqref{eq:R_global}, equivalently writes
\begin{equation}
\label{eq:R_global2}
\forall u \in \mathcal{V}_h,
\quad
\Rop(u) = u + \varepsilon
\sum_{j=1}^d \left( \wper_{e_j} \left(\frac{\cdot}{\varepsilon}\right) -
  \overline{\xi_\varepsilon^j} \right) \partial_j u,
\end{equation}
where $\overline{\xi_\varepsilon^j}$ is locally defined on each element
$\cell$ by 
$\left. \overline{\xi_\varepsilon^j} \right|_{\cell} = 
\left. \xi_\varepsilon^j \right|_\cell$. By construction, for each $\cell$,
$\overline{\xi_\varepsilon^j} \in H^1(\cell)$, but it a priori does not
belong to $H^1(\mathcal{D})$. The
relation~\eqref{eq:R_global2} allows to extend the operator
$\Rop$ on $H^1({\cal D})$.

We now recall the following bound on the function
$\xi_\varepsilon^j$, that appears in~\eqref{phiexp}.
In~\cite{Efendiev2000}, this lemma is stated in dimension $d=2$, but
its proof, which essentially makes use
of~\cite[Lemma~16]{avellaneda-lin-87}, carries over to any dimension. 
\begin{lemme}[see~\cite{Efendiev2000}, Lemma~2.1]
\label{LemChen}
Let $\xi_\varepsilon^j$ be the solution to~\eqref{PB:xip}, with 
$A_{per}$ satisfying~\eqref{hyp:a_holder}.
Consider
$\cell \subset \mathbf{S}$, with $\text{diam}(\cell)=h$ 
and $\mbox{dist}(\cell,\partial \mathbf{S}) \geq h$. Then there exists a
constant $C$ independent of $h$ and $\varepsilon$ such that 
\begin{equation}
\label{eq:estim_bl}
\|\nabla \xi_\varepsilon^j\|_{L^\infty(\cell)} \leq \frac{C}{h}.
\end{equation}
\end{lemme}

\subsubsection{Proof of Theorem~\ref{theo:H1}}
 
The proof is based on the bound~\eqref{Strang2} in
Lemma~\ref{lem:Strang2}, where the bilinear form 
${\cal A}^h_{\varepsilon,\eta}$ and the linear form $b$ 
are defined by~\eqref{def:stoch-bili}. 
In Step 1, we show that the bilinear form ${\cal
  A}^h_{\varepsilon,\eta}$ is coercive for the norm $\| \cdot \|_\h$
defined by~\eqref{eq:broken_H1}. 
Step 2 is
devoted to appropriately selecting an element $v_h \in \mathcal{W}_h$ such that
$\|u^\varepsilon_\eta - v_h\|_\h$ can be analytically estimated. This
will provide a bound on the first term in the right hand side
of~\eqref{Strang2}. 
In Step 3, we bound from above the second term in the right hand
side of~\eqref{Strang2}, using Lemmas~\ref{Lem1} and~\ref{Lem2}. 
Step 4 collects our estimates and concludes. 

\bigskip

\noindent
{\bf Step 1:} We first show that the bilinear form 
${\cal A}^h_{\varepsilon,\eta}$ defined by~\eqref{def:stoch-bili} is
coercive for the norm $\| \cdot \|_\h$ defined by~\eqref{eq:broken_H1}. 
Consider the bilinear form
$\widetilde{\mathcal{A}}^h_{\varepsilon,\eta}$ defined
by~\eqref{def:stoch-bili-tilde}. We pointed out above (see
Remark~\ref{rem:coerc}) that it is coercive on $\mathcal{V}_h$. Hence,
there exists $\alpha>0$ such that, for all $v_h \in
\mathcal{W}_h$,   
\begin{equation}
\label{eq:merc1}
\alpha \|\widetilde{v}_h\|^2_{H^1(\mathcal{D})} \leq 
\widetilde{\mathcal{A}}^h_{\varepsilon,\eta}(\widetilde{v}_h,\widetilde{v}_h)
=
\mathcal{A}^h_{\varepsilon,\eta}(v_h,v_h),
\end{equation}
where $\widetilde{v}_h \in \mathcal{V}_h$ is such that
$v_h = \Rop(\widetilde{v}_h)$. Since, in 
the bilinear form $\mathcal{A}^h_{\varepsilon,\eta}$, the matrix
$A_\eta$ is bounded, we deduce the estimate
\begin{equation}
\label{equiv-norm}
\|\widetilde{v}_h\|^2_{H^1(\mathcal{D})} \leq C \|v_h\|^2_\h,
\end{equation}
that we will use in the sequel. The sequel of this step is devoted to
proving that there exists $\widetilde{C}$ independent of $h$ and $\eps$
such that, for all $\widetilde{v}_h \in \mathcal{V}_h$, 
\begin{equation}
\label{eq:merc2}
\|v_h\|^2_\h \leq \widetilde{C} \|\widetilde{v}_h\|^2_{H^1(\mathcal{D})}
\quad \text{with $v_h = \Rop(\widetilde{v}_h)$.}
\end{equation}
Combined with~\eqref{eq:merc1}, this shows that 
${\cal A}^h_{\varepsilon,\eta}$ is coercive for the norm $\| \cdot \|_\h$.

\medskip

To prove~\eqref{eq:merc2}, we first write that,
since $v_h = \Rop(\widetilde{v}_h)$ and $\widetilde{v}_h \in
\mathcal{V}_h$, there exist some coefficients $\left\{ \beta_i
\right\}_{i=1}^L$ such that, for any $x \in {\cal D}$, 
$\widetilde{v}_h = \sum_{i=1}^L \beta_i \phi_i^0$ and
$v_h = \Rop(\widetilde{v}_h) = \sum_{i=1}^L \beta_i \phi_i^\eps$.
Consider now an element $\cell$, and its corresponding oversampling
domain $\mathbf{S}$. We know
from~\eqref{eq:MsFEM-Over-nb-1} and~\eqref{eq:MsFEM-Over-nb-3}
that
$$
\forall x \in \cell, \quad 
\widetilde{v}_h(x) = \sum_{i=1}^L \sum_{j=1}^{d+1} \beta_i \alpha_{ij}
\chi_j^{0,\mathbf{S}}(x), 
\quad
v_h(x) = \sum_{i=1}^L \sum_{j=1}^{d+1} \beta_i \alpha_{ij}
\chi_j^{\eps,\mathbf{S}}(x).
$$
Consider now the functions
$$
\widetilde{w}^{\mathbf{S}}_h(x) := 
\sum_{i=1}^L \sum_{j=1}^{d+1} \beta_i \alpha_{ij}
\chi_j^{0,\mathbf{S}}(x),
\quad
w^{\mathbf{S}}_h(x) := \sum_{i=1}^L \sum_{j=1}^{d+1} \beta_i \alpha_{ij}
\chi_j^{\eps,\mathbf{S}}(x),
$$
defined on $\mathbf{S}$, and that satisfy, by construction,
\begin{equation}
\label{eq:merc3}
\forall x \in \cell, \quad 
\widetilde{v}_h(x) = \widetilde{w}^{\mathbf{S}}_h(x),
\quad
v_h(x) = w^{\mathbf{S}}_h(x).
\end{equation}
In view of~\eqref{PB:MsFEM-Over-nb}, we have
$$
-\mbox{div}\left[A^{\varepsilon}(x)\nabla
   w^{\mathbf{S}}_h(x) \right] = 0 
\ \ \mbox{in $\mathbf{S}$}, 
\quad
w^{\mathbf{S}}_h = \widetilde{w}^{\mathbf{S}}_h
\ \ \mbox{on $\partial \mathbf{S}$},
$$
which implies that
$$
\| w^{\mathbf{S}}_h \|_{H^1({\mathbf{S}})} \leq C 
\| \widetilde{w}^{\mathbf{S}}_h \|_{H^1({\mathbf{S}})}.
$$
We deduce from~\eqref{eq:merc3} and the above bound that
\begin{equation}
\label{eq:merc4}
\| v_h \|_{H^1(\cell)} 
=
\| w^{\mathbf{S}}_h \|_{H^1(\cell)}
\leq 
\| w^{\mathbf{S}}_h \|_{H^1({\mathbf{S}})}
\leq C 
\| \widetilde{w}^{\mathbf{S}}_h \|_{H^1({\mathbf{S}})}.
\end{equation}
We next see that there exists $C$ independent of $h$ such that, for any
piecewise-affine function $\tau$ on $\mathbf{S}$, we have $\| \tau
\|_{H^1({\mathbf{S}})} \leq C \| \tau \|_{H^1(\cell)}$, provided there
exists $0<c_- \leq c_+$ independent of the element such that
$\dps c_- \leq \frac{|\mathbf{S}|}{|\cell|} \leq c_+$. Using this bound
for $\tau = \widetilde{w}^{\mathbf{S}}_h$, we infer
from~\eqref{eq:merc4} and~\eqref{eq:merc3} that
$$
\| v_h \|_{H^1(\cell)} 
\leq C 
\| \widetilde{w}^{\mathbf{S}}_h \|_{H^1({\mathbf{S}})}
\leq
\bar C \| \widetilde{w}^{\mathbf{S}}_h \|_{H^1(\cell)}
=
\bar C \| \widetilde{v}_h \|_{H^1(\cell)}.
$$
Summing over all elements $\cell$, we obtain~\eqref{eq:merc2}, and this
concludes this first step. 

\bigskip

\noindent
{\bf Step 2:} 
Let $\Pi^h u_\eta^\star$ be
the $H^1$ projection of $u_\eta^\star$, solution
to~\eqref{PB:homo}, on the standard FEM space $\mathcal{V}_h$. We have 
$\Rop \left( \Pi^h u_\eta^\star \right) \in \mathcal{W}_h$
(recall $\Rop$ is defined by~\eqref{eq:R_global}, and equivalently
writes as in~\eqref{eq:R_global2}).
Our argument is based on the following triangle inequality: 
\begin{eqnarray}
\label{eq:triang2_pre}
\esp \left( \inf\limits_{v_h \in \mathcal{W}_h} 
\| u_\eta^\varepsilon - v_h\|^2_\h \right) 
& \leq &
2 \esp \left( \|u_\eta^\varepsilon - v_\eta^\varepsilon\|^2_{H^1(\mathcal{D})}
\right)
+
2 \esp \left( \inf\limits_{v_h \in \mathcal{W}_h} 
\| v_\eta^\varepsilon - v_h\|^2_\h \right) 
\\
&\leq& 
2 \esp \left( \|u_\eta^\varepsilon - v_\eta^\varepsilon\|^2_{H^1(\mathcal{D})}
\right)
+
2 \esp \left( \| v_\eta^\varepsilon - \Rop(\Pi^h u_\eta^\star) \|^2_\h
\right) 
\label{eq:triang2_pre2}
\\
\label{eq:triang2}
&\leq& 
2 \esp \left( \|u_\eta^\varepsilon - v_\eta^\varepsilon\|^2_{H^1(\mathcal{D})}
\right) 
+ 
4 \esp \left( \|v_\eta^\varepsilon - \Rop(u_\eta^\star)\|^2_\h \right) 
+ 
4 \| \Rop(u_\eta^\star) - \Rop(\Pi^h u_\eta^\star) \|^2_\h,
\end{eqnarray}
where $v^\varepsilon_\eta(\cdot,\omega) \in H^1(\mathcal{D})$ is defined
by~\eqref{expueta}. The estimate~\eqref{restueta} in
Theorem~\ref{exp2scaleueta} bounds the first term from above. In the
following two sub-steps, we bound the other two terms
of~\eqref{eq:triang2}. 

\medskip

\noindent
{\bf Step 2a: bound on 
$\esp \left( \|v_\eta^\varepsilon - \Rop(u_\eta^\star)\|^2_\h \right)$} 

Using the expansion~\eqref{exp:ustareta} of
$u_\eta^\star$ in a series in powers of $\eta$,
and~\eqref{eq:R_global2}, we write 
$$
\Rop(u_\eta^\star) = u_0^\star + \eta \esp(X_0) \overline{u}_1^\star +
\varepsilon \sum_{p=1}^d \left( \wper_{e_p}\left(\frac{\cdot}{\varepsilon}\right) -
  \overline{\xi_\varepsilon^p} \right) \left( \partial_p u_0^\star + \eta
\esp(X_0) \partial_p \overline{u}_1^\star \right) + \eta^2 g_\eta, 
$$
where
\begin{equation}
\label{eq:def_geta}
g_\eta = r_\eta + \varepsilon \sum_{p=1}^d \left(
  \wper_{e_p}\left(\frac{\cdot}{\varepsilon}\right) -
  \overline{\xi_\varepsilon^p} \right) \partial_p r_\eta.
\end{equation}
Using~\eqref{expueta}, we thus have
\begin{multline}
v^\varepsilon_\eta(\cdot,\omega) - \Rop(u_\eta^\star) 
= 
\eta \varepsilon \sum\limits_{p=1}^d \left(\esp(X_0)
  \psi_{e_p}\left(\frac{\cdot}{\varepsilon}\right) \partial_p u_0^\star +
  \sum\limits_{k \in I_\varepsilon} (X_k(\omega)-\esp(X_0)) \
  \chi_{e_p}\left(\frac{\cdot}{\varepsilon}-k\right) \partial_p u_0^\star
\right) 
\\
+ 
\varepsilon \sum\limits_{p=1}^d \overline{\xi_\varepsilon^p} 
(\partial_p u_\eta^\star - \eta^2 \partial_p r_\eta)
- \eta^2 g_\eta.
\label{expression0}
\end{multline}
To bound 
$\dps \esp \left[ \sum_\cell \|v^\varepsilon_\eta -
  \Rop(u_\eta^\star)\|^2_{H^1(\cell)} \right]$,  
we first establish a few simple results.
First, there exists $\delta > 0$ such that, for any $1 \leq p \leq d$,
we have 
\begin{equation}
\label{aide1a_pre}
\wper_{e_p} \in C^{1,\delta}(Q),
\end{equation}
where $w_p^0$ is the periodic corrector, solution to~\eqref{PB:w0}. This
is a consequence of the fact that $A_{per}$ is 
H\"older-continuous (see assumption~\eqref{hyp:a_holder}), 
in view 
of~\cite[Theorem 8.22 and Corollary~8.36]{gilbarg-trudinger}.
We infer from~\eqref{aide1a_pre} and the periodicity of $\wper_{e_p}$
that, for any $1 \leq p \leq d$, we have
\begin{equation}
\label{aide1a}
\wper_{e_p} \in W^{1,\infty}(\RR^d).
\end{equation}
Second, for any $1 \leq p \leq d$, we have
\begin{equation}
\label{aide1b}
\left\|
  \chi_{e_p}\left(\frac{\cdot}{\varepsilon}-k\right)\right\|_{L^2(\mathcal{D})}^2 
\leq 
C \varepsilon^d R_{d,\eps}
\quad \text{and} \quad
\nabla \chi_{e_p} \in \left( L^2(\RR^d) \right)^d,
\end{equation}
where $C$ is independent from $\varepsilon$ and $k$,
$R_{d,\eps} = 1$ if $d>2$ and $R_{d,\eps} = 1+\ln(1/\varepsilon)$ if $d=2$
(see~\eqref{eq:def_chi} and~\eqref{eq:bound_chi}).
Third, we see that, for any $1 \leq p \leq d$,  
\begin{equation}
\label{aide2}
\left\| \overline{\xi_\varepsilon^p} \right\|_{L^\infty(\mathcal{D})} \leq C
\quad \text{and} \quad
\left\| \nabla \xi_\varepsilon^p \right\|_{L^\infty(\cell)} \leq \frac{C}{h},
\end{equation}
where $C$ is independent from $\varepsilon$ and $h$. 
The second assertion is given by Lemma~\ref{LemChen} above,
whereas the first assertion comes~\eqref{PB:xip}: using
again~\cite[Theorem 8.22 and Corollary~8.36]{gilbarg-trudinger}
and~\eqref{aide1a_pre}, we first see that, for any $\mathbf{S}$,
$\xi_\varepsilon^p \in C^{1,\delta}(\mathbf{S})$ for some
$\delta>0$. Using next the maximum principle on~\eqref{PB:xip}, we have
$\dps
\| \xi_\varepsilon^p \|_{L^\infty(\mathbf{S})} \leq
\left\|w_{e_p} \left(\frac{\cdot}{\varepsilon}\right) \right\|_{L^\infty(\RR^d)} 
\leq C.
$
Lastly, using~\eqref{eq:def_geta},~\eqref{eq:bound_reta_H2},~\eqref{aide1a}
and~\eqref{aide2}, we obtain that, for any element 
$\cell$, 
$$
\|g_\eta\|_{H^1(\cell)} \leq C \|r_\eta\|_{H^2(\cell)} \left(1 +
  \varepsilon + \frac{\varepsilon}{h}\right) \leq C \|r_\eta\|_{H^2(\cell)},
$$
hence
\begin{equation}
\label{aide3}
\|g_\eta\|^2_{L^2(\mathcal{D})} \leq
\sum_\cell \|g_\eta\|^2_{H^1(\cell)} \leq 
C \|r_\eta\|^2_{H^2(\mathcal{D})}
\leq C.
\end{equation}

\medskip

We are now in position to estimate~\eqref{expression0}. Using that 
$u_0^\star \in W^{1,\infty}(\mathcal{D})$,
we deduce from~\eqref{aide1b},~\eqref{aide2}
and~\eqref{aide3} that  
\begin{eqnarray}
\nonumber
\esp \left[ \|v^\varepsilon_\eta -
\Rop(u_\eta^\star)\|^2_{L^2(\mathcal{D})} \right] 
& \leq &
C \eta^2 \varepsilon^2  \sum\limits_{p=1}^d \|\partial_p
u_0^\star\|^2_{L^\infty} \left(\esp(X_0)^2
  \left\| \psi_{e_p} \left(\frac{\cdot}{\varepsilon}\right)\right\|^2_{L^2(\mathcal{D})}
  + \var(X_0) \sum\limits_{k \in I_\varepsilon}
  \left\| \chi_{e_p} \left(\frac{\cdot}{\varepsilon}-k\right)\right\|_{L^2(\mathcal{D})}^2\right)
\\
\nonumber
&& \qquad + C \varepsilon^2 \sum\limits_{p=1}^d 
\left\| \overline{\xi_\varepsilon^p} \right\|^2_{L^\infty(\mathcal{D})}
\left( \|\nabla u_\eta^\star\|^2_{L^2(\mathcal{D})} +
  \eta^4 \|\nabla r_\eta \|^2_{L^2(\mathcal{D})} \right) + 
C \eta^4 \|g_\eta\|^2_{L^2(\mathcal{D})}
\\
\nonumber
& \leq &
C \eta^2 \varepsilon^2  \left( 1 +
  (\text{Card $I_\eps$}) \, \eps^d R_{d,\eps} \right)
+ C \varepsilon^2 + C \eta^4
\\
& \leq &
C \left[ \eta^2 \varepsilon^2 R_{d,\eps} + \varepsilon^2 + \eta^4 \right]
\label{borneL2-1}
\end{eqnarray}
for some constant $C$ independent of $\varepsilon$, $\eta$ and
$h$, and where, we recall, $R_{d,\eps} = 1 + \ln(1/\varepsilon)$ if
$d=2$ and $R_{d,\eps} = 1$ if $d>2$.

We thus have a bound on 
$v^\varepsilon_\eta(\cdot,\omega) - \Rop(u_\eta^\star)$ in
the $L^2$ norm. To prove a bound in the broken $H^1$ norm,
we consider $\nabla v^\varepsilon_\eta(\cdot,\omega) - \nabla
\Rop(u_\eta^\star)$: we see from~\eqref{expression0} that, in each
element $\cell$,
\begin{equation}
\nabla v^\varepsilon_\eta(\cdot,\omega) - \nabla \Rop(u_\eta^\star) 
= 
\eta \varepsilon D_0 + \eta D_1 + D_2 - \eta^2 \nabla g_\eta,
\label{expv-Reta}
\end{equation}
where
\begin{eqnarray*}
D_0 &=& \sum\limits_{p=1}^d  \left(\esp(X_0)
  \psi_{e_p} \left(\frac{\cdot}{\varepsilon}\right) \nabla \partial_p
  u_0^\star + \sum\limits_{k \in I_\varepsilon} (X_k(\omega)-\esp(X_0))
  \chi_{e_p} \left(\frac{\cdot}{\varepsilon}-k\right) \nabla \partial_p
  u_0^\star \right), 
\\
D_1 &=& \sum\limits_{p=1}^d  \left(\esp(X_0) \nabla
  \psi_{e_p} \left(\frac{\cdot}{\varepsilon}\right) \partial_p u_0^\star +
  \sum\limits_{k \in I_\varepsilon} (X_k(\omega)-\esp(X_0)) \nabla
  \chi_{e_p} \left(\frac{\cdot}{\varepsilon}-k\right) \partial_p u_0^\star
\right), 
\\
D_2 &=& \varepsilon \sum_{p=1}^d \left( \overline{\xi_\varepsilon^p} \ \nabla
  \partial_p (u^\star_\eta - \eta^2 r_\eta) + 
\nabla \overline{\xi_\varepsilon^p} \ \partial_p (u^\star_\eta - \eta^2 r_\eta) \right).
\end{eqnarray*}
Note that $D_0$ and $D_1$ are globally defined on $\mathcal{D}$, but
$D_2$ is not (as $\overline{\xi_\varepsilon^p}$ may have jumps from one element
$\cell$ to the other). We now bound these three
quantities. Using~\eqref{aide1b} and 
the fact that $u_0^\star \in W^{2,\infty}(\mathcal{D})$, we
have
\begin{eqnarray}
\esp(\|D_0\|^2_{L^2(\mathcal{D})}) 
&\leq& 
C \sum\limits_{p=1}^d  \|\nabla \partial_p u_0^\star\|^2_{L^\infty}
\left(\esp(X_0)^2
  \left\|\psi_{e_p}\left(\frac{\cdot}{\varepsilon}\right)\right\|^2_{L^2(\mathcal{D})}
  + \var(X_0) \sum\limits_{k \in I_\varepsilon}
  \left\| \chi_{e_p} \left(\frac{\cdot}{\varepsilon}-k\right) \right\|_{L^2(\mathcal{D})}^2 \right) 
\nonumber \\
&\leq& 
C R_{d,\eps},
\label{majL2E0}
\end{eqnarray}
where $C$ is a constant independent of $\varepsilon$ and $h$. We now
turn to $D_1$: using again~\eqref{aide1b} and that $u_0^\star \in
W^{1,\infty}(\mathcal{D})$, we obtain 
\begin{eqnarray}
\esp(\|D_1\|^2_{L^2(\mathcal{D})}) 
&\leq& 
C \sum\limits_{p=1}^d
 \|\partial_p u_0^\star\|^2_{L^\infty} \left( \esp(X_0)^2 \left\|\nabla
   \psi_{e_p}\left(\frac{\cdot}{\varepsilon}\right)\right\|^2_{L^2(\mathcal{D})} +
   \var(X_0) \sum\limits_{k \in I_\varepsilon} \left\|\nabla
   \chi_{e_p} \left(\frac{\cdot}{\varepsilon}-k\right)
 \right\|_{L^2(\mathcal{D})}^2  \right) 
\nonumber \\
&\leq& 
C \sum\limits_{p=1}^d 
\left( 1 + \var(X_0) \sum\limits_{k \in I_\varepsilon} \varepsilon^d 
\left\| \nabla \chi_{e_p} \right\|_{L^2(\RR^d)}^2 \right) 
\nonumber \\
&\leq& C,
\label{majL2E1}
\end{eqnarray}
where $C$ is a constant independent of $\varepsilon$ and $h$. 
Turning to $D_2$, using~\eqref{aide2}, 
we have in each element $\cell$ that
$$
\|D_2\|^2_{L^2(\cell)} \leq 
C \varepsilon^2 \| u_\eta^\star - \eta^2 r_\eta\|^2_{H^2(\cell)} \left(1 +
\frac{1}{h^2}\right) 
\leq 
C \frac{\varepsilon^2}{h^2} \| u_\eta^\star - \eta^2 r_\eta\|^2_{H^2(\cell)}, 
$$
hence, using~\eqref{eq:bound_reta_H2},
\begin{equation}
\sum_\cell \|D_2\|^2_{L^2(\cell)} \leq 
C \frac{\varepsilon^2}{h^2} \| u_\eta^\star - \eta^2 r_\eta\|^2_{H^2(\mathcal{D})}
\leq C \frac{\varepsilon^2}{h^2}.
\label{majL2E2}
\end{equation}

Collecting~\eqref{expv-Reta},~\eqref{majL2E0},~\eqref{majL2E1},~\eqref{majL2E2}
and~\eqref{aide3}, we obtain that
$$
\esp\left[ \sum_\cell \|\nabla v^\varepsilon_\eta - \nabla
  \Rop(u_\eta^\star)\|^2_{L^2(\cell)} \right] 
\leq 
C \left( \eta^2 \eps^2 R_{d,\eps}
+
\eta^2 + \frac{\varepsilon^2}{h^2} + \eta^4 \right). 
$$
Collecting this bound with~\eqref{borneL2-1}, and assuming that $|\eta|
< 1$ and $\eps^2 R_{d,\eps} \leq 1$, we
deduce that
\begin{equation}
\esp \left[ \| v^\varepsilon_\eta - \Rop(u_\eta^\star) \|^2_{\h} \right]
=
\esp\left[ \sum_\cell \| v^\varepsilon_\eta -
  \Rop(u_\eta^\star)\|^2_{H^1(\cell)} \right]
\leq 
C \left( \eta^2 + \frac{\varepsilon^2}{h^2} \right),
\label{majH1v-Reta}
\end{equation}
where $C$ is independent from $\eps$, $h$ and $\eta$.

\bigskip

\noindent
{\bf Step 2b: bound on 
$\| \Rop(u_\eta^\star) - \Rop(\Pi^h u_\eta^\star) \|^2_\h$} 

Recall that $\Pi^h u_\eta^\star$
is the $H^1$ projection of $u_\eta^\star$ (on the standard
$\mathbb{P}_1$ FEM space $\mathcal{V}_h$), hence $\|\Pi^h
u_\eta^\star\|_{H^1(\mathcal{D})}\leq
\|u_\eta^\star\|_{H^1(\mathcal{D})}$.  
In addition, using~\eqref{eq:bound_reta_H2}, we have, using a standard
result from the theory of $\mathbb{P}_1$ finite elements
(see~\cite[Theorem~3.1.6 p.~124]{Ciarlet})
\begin{equation}
\label{ef_ustar}
\|u_\eta^\star-\Pi^h u_\eta^\star\|_{L^2(\mathcal{D})} + 
h \|u_\eta^\star-\Pi^h u_\eta^\star\|_{H^1(\mathcal{D})} 
\leq C h^2
\| \nabla^2 u_\eta^\star \|_{L^2(\mathcal{D})} 
\leq C h^2,
\end{equation}
where $C$ is a constant independent of $h$ and $\eta$.
In view of~\eqref{eq:R_global2}, we have
$$
\Rop(u_\eta^\star) - \Rop(\Pi^h u_\eta^\star) 
= 
u_\eta^\star - \Pi^h u_\eta^\star + \varepsilon \sum\limits_{p=1}^d
\left(\wper_{e_p} \left(\frac{\cdot}{\varepsilon}\right) -
  \overline{\xi_\varepsilon^p} \right) \partial_p (u_\eta^\star - \Pi^h
u_\eta^\star).
$$
We deduce from~\eqref{aide1a},~\eqref{aide2} and~\eqref{ef_ustar} that
\begin{equation}
\| \Rop(u_\eta^\star) - \Rop(\Pi^h u_\eta^\star) \|_{L^2(\mathcal{D})} 
\leq C (h^2 + \varepsilon h). 
\label{eq:yo}
\end{equation}
We now turn to bounding the gradients. Recall that $\nabla (\Pi^h
u_\eta^\star)$ is constant in each element $\cell$. 
We thus have, using~\eqref{aide1a} and~\eqref{aide2}, that
\begin{eqnarray*}
\|\nabla \Rop(u_\eta^\star) - 
\nabla \Rop(\Pi^h u_\eta^\star)\|_{L^2(\cell)} 
& \leq & 
\|u_\eta^\star-\Pi^h u_\eta^\star \|_{H^1(\cell)} \left( 1+ 
\sum\limits_{p=1}^d \left(
    \left\| \nabla \wper_{e_p} \left(\frac{\cdot}{\varepsilon}\right)
    \right\|_{L^\infty(\cell)} + 
\varepsilon \| \nabla \xi_\varepsilon^p\|_{L^\infty(\cell)} \right) \right) 
\\
&& \qquad
+ \varepsilon \sum\limits_{p=1}^d 
\left\| \wper_{e_p} \left(\frac{\cdot}{\varepsilon}\right) - 
\xi_\varepsilon^p \right\|_{L^\infty(\cell)}
\|u_\eta^\star \|_{H^2(\cell)} 
\\
& \leq & 
C \|u_\eta^\star-\Pi^h u_\eta^\star \|_{H^1(\cell)} 
\left( 1+ \frac{\varepsilon}{h} \right) 
+ \varepsilon \|u_\eta^\star \|_{H^2(\cell)}.
\end{eqnarray*}
We then deduce, using~\eqref{ef_ustar} and~\eqref{eq:bound_reta_H2}, that
\begin{eqnarray*}
\sum_\cell \|\nabla \Rop(u_\eta^\star) - 
\nabla \Rop(\Pi^h u_\eta^\star)\|^2_{L^2(\cell)} 
& \leq &
C \|u_\eta^\star-\Pi^h u_\eta^\star \|^2_{H^1(\mathcal{D})} 
\left( 1+ \frac{\varepsilon}{h} \right)^2 
+ \varepsilon^2 \|u_\eta^\star \|^2_{H^2(\mathcal{D})}
\\
& \leq &
C h^2 \left( 1+ \frac{\varepsilon}{h} \right)^2 
+ C \varepsilon^2,
\end{eqnarray*}
where $C$ is a constant independent of $\varepsilon$, $\eta$ and
$h$. Collecting this bound and~\eqref{eq:yo}, we obtain
\begin{equation}
\label{eq:yo2}
\| \Rop(\Pi^h u_\eta^\star) - \Rop(u_\eta^\star)\|^2_\h
=
\sum_\cell \| \Rop(u_\eta^\star) - 
\Rop(\Pi^h u_\eta^\star)\|^2_{H^1(\cell)} 
\leq
C \left( h^2 + \varepsilon^2 \right),
\end{equation}
where $C$ is a constant independent of $\varepsilon$, $\eta$ and
$h$.

\bigskip

\noindent
{\bf Step 2c:} 
We are now in position to bound the first term in~\eqref{Strang2}. 
We infer from~\eqref{eq:triang2},~\eqref{restueta},~\eqref{majH1v-Reta}
and~\eqref{eq:yo2} that 
\begin{eqnarray}
\sqrt{ \esp \left( \inf\limits_{v_h \in \mathcal{W}_h} 
\| u^\varepsilon_\eta - v_h\|^2_\h \right) } 
&\leq&
C \left(\sqrt{\varepsilon} + \eta \sqrt{\varepsilon
    \ln(1/\varepsilon)} + \eta
+ \frac{\varepsilon}{h} + h \right)
\nonumber
\\
&\leq&
C \left(\sqrt{\varepsilon} + \eta
+ \frac{\varepsilon}{h} + h \right),
\label{eq3:theo-uS}
\end{eqnarray}
where we have assumed that $\varepsilon \ln(1/\varepsilon) \leq 1$.

\bigskip

\noindent
{\bf Step 3:} We next turn to estimating the non-conforming error,
namely the second term of the right-hand side of~\eqref{Strang2}.
For any $w_h \in \mathcal{W}_h$, introduce $\widetilde{w}_h \in \mathcal{V}_h$
such that $\Rop(\widetilde{w}_h) = w_h$ (recall that
$\Rop$ is defined by~\eqref{eq:R_global}). We note that
$b(w_h) = \widetilde{b}_h(\widetilde{w}_h)$, where the linear forms $b$
and $\widetilde{b}_h$ are defined by~\eqref{def:stoch-bili}
and~\eqref{def:stoch-bili-tilde}. Using the weak form of the 
homogenized equation (see~\eqref{def:homo-bili}), we see that
$b(\widetilde{w}_h) =
\mathcal{A}_\eta^\star(u_\eta^\star,\widetilde{w}_h)$. In addition, by
definition of $\widetilde{\mathcal{A}}^h_{\varepsilon,\eta}$
(see~\eqref{def:stoch-bili-tilde}), we have  
$\dps
\widetilde{\mathcal{A}}^h_{\varepsilon,\eta}(\Pi_h u_\eta^\star,\widetilde{w}_h)
=
\mathcal{A}^h_{\varepsilon,\eta}(\Rop(\Pi_h u_\eta^\star),w_h)$.
For any $w_h \in \mathcal{W}_h$, we have
\begin{eqnarray*}
\left|\mathcal{A}^h_{\varepsilon,\eta}(u_\eta^\varepsilon,w_h) - b(w_h)\right| 
&\leq& 
\left|\mathcal{A}^h_{\varepsilon,\eta}(u^\varepsilon_\eta,w_h) - 
\mathcal{A}^h_{\varepsilon,\eta}(\Rop(\Pi_h u_\eta^\star),w_h)\right| 
+ 
\left| \mathcal{A}^h_{\varepsilon,\eta}(\Rop(\Pi_h u_\eta^\star),w_h) -
b(\widetilde{w}_h)\right| 
+  
\left| b(\widetilde{w}_h) - b(w_h) \right| 
\\
&\leq& 
\| A_\eta \|_{L^\infty} 
\|u^\varepsilon_\eta - \Rop(\Pi_h u_\eta^\star)\|_\h \ \|w_h\|_\h 
+ 
\left|\widetilde{\mathcal{A}}^h_{\varepsilon,\eta}(\Pi_h u_\eta^\star,\widetilde{w}_h) 
- 
\mathcal{A}_\eta^\star(u_\eta^\star,\widetilde{w}_h)\right| 
+ 
\left|b(\widetilde{w}_h) - \widetilde{b}_h(\widetilde{w}_h)\right| 
\\
&\leq& 
\| A_\eta \|_{L^\infty} 
\|u_\eta^\varepsilon - \Rop(\Pi_h u_\eta^\star)\|_\h \|w_h\|_\h 
+ 
\left|\widetilde{\mathcal{A}}^h_{\varepsilon,\eta}(\Pi^h u_\eta^\star,\widetilde{w}_h) 
- 
\mathcal{A}_\eta^\star(\Pi^h u_\eta^\star,\widetilde{w}_h)\right| 
\\
&& \quad \quad 
+ \| A^\star_\eta \| \ \|u_\eta^\star-\Pi^h u_\eta^\star\|_{H^1(\mathcal{D})} 
\|\widetilde{w}_h\|_{H^1(\mathcal{D})} 
+ \left|b(\widetilde{w}_h) - \widetilde{b}_h(\widetilde{w}_h)\right|,
\end{eqnarray*}
where we have successively used the continuity of the bilinear forms
$\mathcal{A}^h_{\varepsilon,\eta}$ and $\mathcal{A}_\eta^\star$. Using
Lemmas~\ref{Lem1} and~\ref{Lem2} for the second and the fourth terms
respectively, we deduce that
\begin{eqnarray*}
\left|\mathcal{A}^h_{\varepsilon,\eta}(u_\eta^\varepsilon,w_h) - b(w_h)\right| 
&\leq& 
C 
\|u_\eta^\varepsilon - \Rop(\Pi_h u_\eta^\star)\|_\h \|w_h\|_\h 
+ 
C \left( \frac{\varepsilon}{h} + \eta \lambda(\omega,h,\varepsilon) +
  \eta^2 \mathcal{C}(\eta) \right)
\| \Pi^h u_\eta^\star \|_{H^1(\mathcal{D})} \| \widetilde{w}_h \|_{H^1(\mathcal{D})} 
\\
&& \quad \quad 
+ C \|u_\eta^\star-\Pi^h u_\eta^\star\|_{H^1(\mathcal{D})} 
\|\widetilde{w}_h\|_{H^1(\mathcal{D})} 
+ C \varepsilon \| \widetilde{w}_h \|_{H^1(\mathcal{D})},
\end{eqnarray*}
hence, using~\eqref{equiv-norm} and~\eqref{ef_ustar},
\begin{equation}
\label{eq:pluie1}
\frac{\left|\mathcal{A}^h_{\varepsilon,\eta}(u_\eta^\varepsilon,w_h) -
  b(w_h)\right|}{\|w_h\|_\h} 
\leq
C \|u_\eta^\varepsilon - \Rop(\Pi_h u_\eta^\star)\|_\h 
+ 
C \left( h + \frac{\varepsilon}{h} + \eta \lambda(\omega,h,\varepsilon)
  + \varepsilon + \eta^2 \mathcal{C}(\eta) 
\right).
\end{equation}
The first term is bounded as in Step 2:
$$
\| u_\eta^\varepsilon - \Rop(\Pi_h u_\eta^\star) \|_\h 
\leq
\| u_\eta^\varepsilon - v_\eta^\varepsilon \|_{H^1({\cal D})}
+
\| v_\eta^\varepsilon - \Rop(u_\eta^\star) \|_\h
+
\| \Rop(u_\eta^\star) - \Rop(\Pi_h u_\eta^\star) \|_\h, 
$$
hence, using~\eqref{restueta},~\eqref{majH1v-Reta} and~\eqref{eq:yo2}, and
assuming that $\eps \ln(1/\eps) \leq 1$, we have
\begin{equation}
\label{eq:pluie2}
\sqrt{
\esp \left[ \| u_\eta^\varepsilon - \Rop(\Pi_h u_\eta^\star) \|^2_\h \right]
}
\leq
C \left( \sqrt{\eps} + \eta + \frac{\eps}{h} + h \right).
\end{equation}
Collecting~\eqref{eq:pluie1} and~\eqref{eq:pluie2}, we thus obtain
\begin{equation}
\sqrt{ \esp \left[ \left( \sup_{w_h \in {\cal W}_h}
\frac{\left|\mathcal{A}^h_{\varepsilon,\eta}(u_\eta^\varepsilon,w_h) -
  b(w_h)\right|}{\|w_h\|_\h} 
\right)^2 \right] }
\leq
C \left( h + \frac{\varepsilon}{h} + \eta 
\sqrt{ \esp \left[ \lambda^2(\cdot,h,\varepsilon) \right] }
  + \sqrt{\varepsilon} + \eta + \eta^2 \mathcal{C}(\eta) 
\right).
\label{eq2:theo-uS}
\end{equation}

\bigskip

\noindent
{\bf Step 4:} Collecting~\eqref{Strang2},~\eqref{eq3:theo-uS} and~\eqref{eq2:theo-uS}, we get
$$
\sqrt{ \esp \left[ \|u_\eta^\varepsilon - u_S\|^2_\h \right] } 
\leq C \left( \sqrt{\varepsilon} + h + \frac{\varepsilon}{h} + \eta
\sqrt{\esp \left[ \lambda^2(\cdot,h,\varepsilon) \right] } + \eta + \eta^2 \mathcal{C}(\eta) 
\right)
$$
where $C$ is a constant independent of $\varepsilon$, $h$ and $\eta$, and
$\mathcal{C}$ is a bounded function as $\eta$ goes to $0$.
Using~\eqref{eq:resu4a}, we deduce that 
$$
\sqrt{ \esp \left[ \|u_\eta^\varepsilon - u_S\|^2_\h \right] } 
\leq C \left( \sqrt{\varepsilon} + h + \frac{\varepsilon}{h} + \eta
  \left( \frac{\varepsilon}{h} \right)^{d/2} \ln(N(h)) + \eta + \eta^2 \mathcal{C}(\eta) 
\right)
$$
where $N(h)$ is the number of elements $\cell$ in the domain (which is of
order $h^{-d}$ in dimension $d$). This concludes the proof of Theorem~\ref{theo:H1}.

\subsection{The one dimensional case}
\label{sec:1d}

In this section, we briefly consider the one dimensional situation. 
As in the multi-dimensional case, we assume here that $\dps
a^\varepsilon_\eta(x,\omega) = a_\eta
\left(\frac{x}{\varepsilon},\omega\right)$, where $a_\eta$ is a 
stationary random function satisfying, for any $|\eta| \leq 1$,
the condition $0< a_- \leq a_\eta(x,\omega) \leq a^+$ almost
everywhere in $\RR$, almost surely. In line
with~\eqref{def:A-eta-1}, we assume that
\begin{equation}
a_\eta(x,\omega)=a_{per}(x)+ \eta \: a_1(x,\omega),
\label{def:a-eta-1d}
\end{equation}
where $\eta$ is a small parameter ($|\eta| \leq 1$),
$a_{per}$ is a $1$-periodic function satisfying the 
condition $0< a_- \leq a_{per}(x) \leq a^+$ almost everywhere on $\RR$,
and $a_1$ is a bounded stationary random function:
$|a_1(x,\omega)|\leq C$ almost everywhere in $\RR$, almost surely. 
In the vein of~\eqref{struc:A1}, we suppose that
\begin{equation}
\label{struc:a1}
a_1(x,\omega) = \sum\limits_{k \in \ZZ} \mathbf{1}_{(k,k+1]}(x)
X_k(\omega) \, b_{per}(x)
\quad \text{such that} \quad
\exists C, \, \forall k \in \ZZ, \quad 
|X_k(\omega)| \leq C \quad \text{ almost surely},
\end{equation}
where $\left(X_k(\omega)\right)_{k\in \ZZ}$ is a sequence of
i.i.d. scalar random variables 
and $b_{per} \in L^\infty(\RR)$ is a $1$-periodic
function. 

Note that, in this one-dimensional setting, we
do not make any regularity assumption on $a_{per}$ (in the vein
of~\eqref{hyp:a_holder}). In the multi-dimensional case, this
assumption is useful to e.g. state that the periodic corrector satisfies
$\wper_p \in W^{1,\infty}(\RR^d)$ for any $p \in \RR^d$. In the
one-dimensional case, the corrector problem can be solved analytically,
and one can see that the above assumption $a_{per}(x) \geq a_- >
0$ almost everywhere on $\RR$ is sufficient to obtain such regularity on
the corrector. Similarly, we do not need to assume here, in contrast to
Theorem~\ref{theo:H1}, that $u_0^\star$ and $\overline{u}_1^\star$ defined
by~\eqref{eq:homog-0} and~\eqref{PB:u1barstar} both belong to 
$W^{2,\infty}({\cal D})$ (an assumption equivalent to $f \in
L^\infty({\cal D})$, in the present one-dimensional setting). The
assumption $f \in L^2({\cal D})$ is sufficient. 

The problem (\ref{PB:stoch-a}) now reads 
\begin{equation}
\dps-\frac{d}{dx} \left( a_\eta\left(\frac{x}{\varepsilon},\omega\right)
  \frac{d}{dx}u_\eta^\varepsilon(x,\omega) \right) = f(x) \ \ \mbox{ in
  $(0,1)$}, 
\quad
u_\eta^\varepsilon(0,\omega) = u_\eta^\varepsilon(1,\omega) = 0.
\label{PB:stoch-1d}
\end{equation}
We consider a uniform discretization of the interval $(0,1)$ in the
elements $\cell_i=(x_i,x_{i+1})$, with $x_{i+1}-x_i=h=1/L$
for some $L\in \NN^\star$. 

The one-dimensional version of Theorem~\ref{theo:H1} reads as follows:
\begin{theorem}
\label{theo:H1-1d}
In the one-dimensional setting, 
assume that $a_\eta^\varepsilon$
satisfies~\eqref{def:a-eta-1d}-\eqref{struc:a1}. 
Let $u_\eta^\varepsilon$ be the solution to~\eqref{PB:stoch-1d} with $f
\in L^2(0,1)$, and $u_S$ be the weakly stochastic
MsFEM solution to~\eqref{eq:var-mac}. Suppose that 
$h/\eps \in \NN^\star$. We then have 
\begin{equation}
\label{est-H1-1d}
\sqrt{ \esp \left[ \|u^\varepsilon_\eta - u_S\|^2_\h \right] } \leq C 
\left( \eps + h + \eta \left(
    \frac{\varepsilon}{h} \right)^{1/2} \ln(1/h) 
+ \eta + \eta^2 \mathcal{C}(\eta) \right),
\end{equation}
where $C$ is a constant independent of $\varepsilon$, $h$ and $\eta$ 
and $\mathcal{C}$ is a bounded function as $\eta$ goes to $0$. 
\end{theorem}

\begin{proof}
The proof of this result follows the same lines as that for the
multi-dimensional case. It is based upon the homogenization result
contained in Theorem~\ref{exp2scaleueta-1D} above. 
\end{proof}

As pointed out above, the rate of convergence stated in
Theorem~\ref{exp2scaleueta} (and hence the estimate provided by
Theorem~\ref{theo:H1}) is not optimal in dimension one. This hence
motivates Theorems~\ref{exp2scaleueta-1D} and~\ref{theo:H1-1d},
which are their respective one-dimensional variants.
On another note, the assumption $h/\eps \in \NN^\star$ implies that some
terms in the error bound vanish. A result similar to~\eqref{est-H1-1d}
holds in the absence of such assumption, with the additional term
$\eps/h$ in the right-hand side. 

\appendix

\section{Proofs of Lemmas~\ref{Lem1},~\ref{Lem2}
  and~\ref{coro:lambda}}
\label{sec:tech-proof}

\begin{proof}[Proof of Lemma~\ref{Lem1}]
This result relies on the expansion 
$$
\phi^{\varepsilon,\cell}_j(x) = \phi^{0,\cell}_j(x) + \varepsilon 
\sum\limits_{m=1}^d \left(\wper_{e_m}\left(\frac{x}{\varepsilon}\right)
  - \left. \xi_\varepsilon^m(x) \right|_\cell \right) \partial_m \phi^{0,\cell}_j 
$$
from~\eqref{phiexp} and the fact that $\nabla \phi^{0,\cell}_j$ is
constant on $\cell$. 

For any $v_h$ and $w_h$ in $\mathcal{V}_h$, we write
$$
\left|{\cal A}_\eta^\star(v_h,w_h)-
\widetilde{\cal A}^h_{\varepsilon,\eta}(v_h,w_h)\right|
=
\left|\sum\limits_{\cell \in \mathcal{T}_h}
\left(
\int_\cell \left(\nabla w_h(x) \right)^T
  A_\eta^\star \nabla v_h(x) \, dx  
- 
\sum\limits_{i,j =1}^L v_h^j w_h^i
\int_\cell  \left(\nabla \phi_i^{\varepsilon,\cell}\right)^T
A_\eta\left(\frac{x}{\varepsilon},\omega\right)\nabla\phi_j^{\varepsilon,\cell} 
\, dx \right)\right|, 
$$
where $v_h=\sum\limits_{j=1}^L v_h^j \phi_j^0$ and likewise for
$w_h$. Using the above expansion of $\phi^{\varepsilon,\cell}_j$ and
the fact that $\nabla \phi^{0,\cell}_j$ is constant on $\cell$, we have
\begin{eqnarray*}
&&
\sum\limits_{i,j=1}^L
v_h^j w_h^i \int_\cell \left(\nabla
  \phi_i^{\varepsilon,\cell}\right)^T A_\eta\left(\frac{x}{\varepsilon},\omega\right)\nabla\phi_j^{\varepsilon,\cell}
\, dx 
\\ 
&=& 
\sum\limits_{m,p=1}^d \frac{1}{|\cell|}\int_\cell  \left[e_p + \nabla
  \wper_{e_p}\left(\frac{x}{\varepsilon}\right) - \varepsilon \nabla
  \xi_\varepsilon^p(x) \right]^T
A_\eta\left(\frac{x}{\varepsilon},\omega\right)\left[e_m + \nabla
  \wper_{e_m}\left(\frac{x}{\varepsilon}\right) - \varepsilon \nabla
  \xi_\varepsilon^m(x) \right] \, dx 
\\
&& \hspace{6cm}
\times \sum\limits_{i,j=1}^L v_h^j w_h^i \int_{\cell} \partial_p
\phi^{0,\cell}_i \partial_m \phi^{0,\cell}_j
\\ 
&=& 
\sum\limits_{m,p=1}^d \frac{1}{|\cell|}\int_\cell  \left[e_p + \nabla
  \wper_{e_p}\left(\frac{x}{\varepsilon}\right) - \varepsilon \nabla
  \xi_\varepsilon^p(x) \right]^T
A_\eta\left(\frac{x}{\varepsilon},\omega\right)\left[e_m + \nabla
  \wper_{e_m}\left(\frac{x}{\varepsilon}\right) - \varepsilon \nabla
  \xi_\varepsilon^m(x) \right] \, dx 
\int_\cell \partial_m v_h \partial_p w_h.
\end{eqnarray*}
We thus obtain
\begin{equation}
\label{maj-Ah-Astar}
\left|
{\cal A}_\eta^\star(v_h,w_h)-\widetilde{\cal
  A}^h_{\varepsilon,\eta}(v_h,w_h)
\right| 
=
\left|\sum\limits_{\cell}\sum\limits_{m,p=1}^d
  \Lambda^\cell_{mp}\int_\cell \partial_m v_h \partial_p w_h \right|
\leq 
\sum\limits_\cell \|v_h\|_{H^1(\cell)} \|w_h\|_{H^1(\cell)}  
\sum\limits_{m,p=1}^d |\Lambda^\cell_{mp}|,
\end{equation}
where
$$
\Lambda^\cell_{mp} = \left[A^\star_\eta\right]_{mp} - \frac{1}{ |\cell|
} \int_\cell \left[e_p + \nabla
  \wper_{e_p}\left(\frac{x}{\varepsilon}\right) - \varepsilon \nabla
  \xi_\varepsilon^p(x) \right]^T
A_\eta\left(\frac{x}{\varepsilon},\omega\right) \left[e_m + \nabla
  \wper_{e_m}\left(\frac{x}{\varepsilon}\right) - \varepsilon \nabla
  \xi_\varepsilon^m(x) \right] \, dx, 
  $$
  which we write
\begin{equation}
\Lambda^\cell_{mp} = D_0 + D_1 - D_2,
\label{eq:decompo_Lambda}
\end{equation}
with
\begin{eqnarray}
D_0
&=&\left[A^\star_\eta\right]_{mp} - \frac{1}{ |\cell| } \int_\cell
\left[e_p + \nabla \wper_{e_p}\left(\frac{x}{\varepsilon}\right)
\right]^T A_\eta\left(\frac{x}{\varepsilon},\omega\right)\left[e_m +
  \nabla \wper_{e_m}\left(\frac{x}{\varepsilon}\right) \right] \, dx, 
\label{1-order-lambda}
\\
D_1 &=&
\frac{\varepsilon}{|\cell|}\left(\int_\cell  \left[e_p + \nabla
    \wper_{e_p}\left(\frac{x}{\varepsilon}\right)\right]^T
  A_{\eta}\left(\frac{x}{\varepsilon},\omega\right)  \nabla
  \xi_\varepsilon^m(x) \, dx + \int_\cell  \left(\nabla
    \xi_\varepsilon^p(x)\right)^T
  A_{\eta}\left(\frac{x}{\varepsilon},\omega\right) \left[e_m + \nabla
    \wper_{e_m}\left(\frac{x}{\varepsilon}\right)\right] \, dx \right), 
\nonumber \\
D_2 &=& 
\frac{\varepsilon^2}{|\cell|} \int_\cell  \left(\nabla
  \xi_\varepsilon^p(x)\right)^T
A_{\eta}\left(\frac{x}{\varepsilon},\omega\right)  \nabla
\xi_\varepsilon^m(x) \, dx. 
\nonumber
\end{eqnarray}
We are thus left with bounding $|\Lambda^\cell_{mp}|$ from above. We
first bound $D_1$ and $D_2$. Using Lemma~\ref{LemChen} (recall that
$A_{per}$ satisfies~\eqref{hyp:a_holder}, i.e. is H\"older continuous)
and the fact that $\wper_{e_p}\in H^{1}(Q)$ and is $Q$-periodic, we
obtain 
\begin{equation}
\left| D_2 \right| \leq C \frac{\varepsilon^2}{h^2}
\label{maj-th1-3}
\end{equation}
and
\begin{eqnarray*}
 \left| \frac{\varepsilon}{|\cell|} \int_\cell \left[e_p + \nabla
     \wper_{e_p}\left(\frac{x}{\varepsilon}\right)\right]^T
   A_{\eta}\left(\frac{x}{\varepsilon},\omega\right) \nabla
   \xi_\varepsilon^m(x) \, dx \right| &\leq&
 \frac{\varepsilon}{|\cell|}\|A_\eta\|_{L^\infty}\frac{C}{h} \int_\cell
 \left|e_p+\nabla \wper_{e_p}\left(\frac{x}{\varepsilon}\right)\right|
 \, dx 
 \\
 &\leq& C \frac{\varepsilon}{|\cell|h}\left(|\cell| + \varepsilon^d
   \int_{\cell/\varepsilon}|\nabla \wper_{e_p}(y)| \, dy\right) 
\\
&\leq& C \frac{\varepsilon}{h}
\end{eqnarray*}
hence
\begin{equation}
| D_1 | \leq C \frac{\varepsilon}{h},
\label{maj-th1-1}
\end{equation}
where $C$ is a deterministic constant independent of $h$, $\varepsilon$
and $\eta$. We next turn to $D_0$. We introduce the cells
$Q_i^\varepsilon=\varepsilon(Q+i), \, i\in \ZZ^d$, let
$I^\varepsilon_\cell=\bigcup\limits_{Q_i^\varepsilon \subset
  \cell}Q_i^\varepsilon$, and recast~\eqref{1-order-lambda} as 
\begin{equation}
\label{eq:decompo_D0}
D_0 = D_0^{\rm bulk} - D_0^{\rm boundary}
\end{equation}
with
\begin{eqnarray}
D_0^{\rm bulk}
&=& \left[A^\star_\eta\right]_{mp} - \frac{1}{ |\cell| }
\int_{I^\varepsilon_\cell}  \left[e_p + \nabla
  \wper_{e_p}\left(\frac{x}{\varepsilon}\right) \right]^T
A_\eta\left(\frac{x}{\varepsilon},\omega\right)\left[e_m + \nabla
  \wper_{e_m}\left(\frac{x}{\varepsilon}\right) \right] \, dx,
\label{maj-th1-4}\\
D_0^{\rm boundary}&=& 
\frac{1}{|\cell|} \int_{\cell \setminus I^\varepsilon_\cell}  \left[e_p
  + \nabla \wper_{e_p}\left(\frac{x}{\varepsilon}\right) \right]^T
A_\eta\left(\frac{x}{\varepsilon},\omega\right) \left[e_m + \nabla
  \wper_{e_m}\left(\frac{x}{\varepsilon}\right)\right] \, dx. 
\nonumber
\end{eqnarray}
We denote by $J^\varepsilon_\cell$ the set of cells $Q_i^\varepsilon$
that intersect the element $\cell$, i.e. 
$$
J^\varepsilon_\cell = 
\bigcup\limits_{Q_i^\varepsilon \bigcap \cell \neq \emptyset} Q_i^\varepsilon.
$$
By construction, $I^\varepsilon_\cell \subset \cell \subset
J^\varepsilon_\cell$. Using that $\wper_{e_p}\in H^{1}(Q)$ and is
$Q$-periodic, we write
\begin{eqnarray}
\left| D_0^{\rm boundary} \right|
&\leq&  \|A_\eta\|_{L^\infty}
\frac{1}{|\cell|}\sqrt{\int_{J^\varepsilon_\cell \setminus
    I^\varepsilon_\cell}\left[e_p + \nabla
    \wper_{e_p}\left(\frac{x}{\varepsilon}\right) \right]^2 \, dx}
\ 
\sqrt{\int_{J^\varepsilon_\cell \setminus I^\varepsilon_\cell}\left[e_m
    + \nabla \wper_{e_m}\left(\frac{x}{\varepsilon}\right) \right]^2 \,
  dx} 
\nonumber \\
&\leq& C \frac{\varepsilon^d}{|\cell|} \frac{|\partial
  \cell|}{\varepsilon^{d-1}} \sqrt{\int_{Q}\left[e_p + \nabla
    \wper_{e_p}\left(y\right) \right]^2 \, dy}
\ 
\sqrt{\int_{Q}\left[e_m + \nabla \wper_{e_m}\left(y\right) \right]^2 \,
  dy} 
\nonumber \\
&\leq& C \frac{\varepsilon}{h}. 
\label{maj-rest}
\end{eqnarray}
We next consider~\eqref{maj-th1-4}:
\begin{equation}
\label{eq:decompo_D0bulk}
D_0^{\rm bulk}
=
\frac{|\cell \setminus I^\varepsilon_\cell|}{|\cell|}
\left[A^\star_\eta\right]_{mp} + \frac{|I^\varepsilon_\cell|}{|\cell|}
\overline{D}_0^{\rm bulk},
\end{equation}
with
$$
\overline{D}_0^{\rm bulk}
=
\left[A^\star_\eta\right]_{mp} -
  \frac{1}{|I^\varepsilon_\cell|}\int_{I^\varepsilon_\cell}  \left[e_p +
    \nabla \wper_{e_p}\left(\frac{x}{\varepsilon}\right)\right]^T
  A_\eta\left(\frac{x}{\varepsilon},\omega\right)\left[e_m + \nabla
    \wper_{e_m}\left(\frac{x}{\varepsilon}\right)\right] \, dx.
$$
Using the expansion~\eqref{exp:Astareta} of $A^\star_\eta$, we write
\begin{multline}
\overline{D}_0^{\rm bulk}
=
\int_Q \left[e_p + \nabla \wper_{e_p}(y)\right]^T
A_{per}(y)\left[e_m + \nabla \wper_{e_m}(y)\right] \, dy  
\\
-
\frac{1}{|I^\varepsilon_\cell|} \int_{I^\varepsilon_\cell} \left[e_p +
  \nabla \wper_{e_p}\left(\frac{x}{\varepsilon}\right)\right]^T
A_{per}\left(\frac{x}{\varepsilon}\right)\left[e_m + \nabla
  \wper_{e_m}\left(\frac{x}{\varepsilon}\right)\right] \, dx  
\\
+ \eta \left( \int_Q \left[e_p + \nabla \wper_{e_p}(y)\right]^T
  \esp(A_1(y,\cdot))\left[e_m + \nabla \wper_{e_m}(y)\right] \, dy
\right.  
\\
\left.- \frac{1}{|I^\varepsilon_\cell|} \int_{I^\varepsilon_\cell} \left[e_p + \nabla \wper_{e_p}\left(\frac{x}{\varepsilon}\right)\right]^T A_1\left(\frac{x}{\varepsilon},\omega\right)\left[e_m + \nabla \wper_{e_m}\left(\frac{x}{\varepsilon}\right)\right] \, dx\right) 
+ \eta^2 \mathcal{C}(\eta).
\label{maj-th1-5}
\end{multline}
The leading order term in~\eqref{maj-th1-5} vanishes. We are hence left with
\begin{equation}
\overline{D}_0^{\rm bulk}
= \frac{\eta}{|I^\varepsilon_\cell|}
\int_{I^\varepsilon_\cell}\left[e_p+\nabla
  \wper_{e_p}\left(\frac{x}{\varepsilon}\right)\right]^T
\left(\esp\left(A_1\left(\frac{x}{\varepsilon},\cdot\right)\right)-A_1\left(\frac{x}{\varepsilon},\omega\right)\right)\left[e_m
  + \nabla \wper_{e_m}\left(\frac{x}{\varepsilon}\right)\right] \, dx +
\eta^2 \mathcal{C}(\eta). 
\label{maj-th1-6}
\end{equation}
Collecting~\eqref{eq:decompo_Lambda}, (\ref{maj-th1-3}),
(\ref{maj-th1-1}), \eqref{eq:decompo_D0}, (\ref{maj-rest}),
\eqref{eq:decompo_D0bulk} and (\ref{maj-th1-6}), together with the fact
that 
$\displaystyle \frac{|\cell \setminus
  I^\varepsilon_\cell|}{|\cell|}
\left|\left[A^\star_\eta\right]_{mp}\right|\leq C
\frac{\varepsilon}{h}$, we obtain 
$$
|\Lambda^\cell_{mp}| 
\leq 
C \frac{\varepsilon}{h} + 
\frac{\eta}{|I^\varepsilon_\cell|} \left|
\int_{I^\varepsilon_\cell}\left[e_p+\nabla
  \wper_{e_p}\left(\frac{x}{\varepsilon}\right)\right]^T
\left(\esp\left(A_1\left(\frac{x}{\varepsilon},\cdot\right)\right)-A_1\left(\frac{x}{\varepsilon},\omega\right)\right)\left[e_m
  + \nabla \wper_{e_m}\left(\frac{x}{\varepsilon}\right)\right] \, dx \right|
+ \eta^2 \mathcal{C}(\eta).
$$
We set
$$
\lambda(\omega,h,\varepsilon)=
\max\limits_{\cell \in \mathcal{T}_h} \max\limits_{1\leq m,p \leq d}
\left| \frac{1}{|I^\varepsilon_\cell|} \int_{I^\varepsilon_\cell} \left[e_p + \nabla \wper_{e_p}\left(\frac{x}{\varepsilon}\right)\right]^T \left(\esp\left(A_1\left(\frac{x}{\varepsilon},\cdot\right)\right)-A_1\left(\frac{x}{\varepsilon},\omega\right)\right)\left[e_m + \nabla \wper_{e_m}\left(\frac{x}{\varepsilon}\right)\right] \, dx \right|.
$$
We thus have, for any $\cell$, $\dps \sum\limits_{m,p=1}^d
|\Lambda^\cell_{mp}| \leq C \left( \frac{\varepsilon}{h} + \eta
  \lambda(\omega,h,\varepsilon) + \eta^2 \mathcal{C}(\eta)
\right)$. Using (\ref{maj-Ah-Astar}), we thus obtain that 
$$
\left|{\cal A}_\eta^\star(v_h,w_h)-\widetilde{\cal
    A}^h_{\varepsilon,\eta}(v_h,w_h)\right| 
\leq 
C \left( \frac{\varepsilon}{h} + \eta \lambda(\omega,h,\varepsilon) +
  \eta^2
  \mathcal{C}(\eta)\right)\|v_h\|_{H^1(\mathcal{D})}\|w_h\|_{H^1(\mathcal{D})}. 
$$
This concludes the proof of Lemma~\ref{Lem1}. 
\end{proof}

\noindent\begin{proof}[Proof of Lemma~\ref{Lem2}]
Again, as for Lemma~\ref{Lem1}, this result relies on the
expansion~\eqref{phiexp} of $\phi_i^{\varepsilon,\cell}$ and the fact
that $\nabla \phi_i^{0,\cell}$ is constant on $\cell$. 

Setting $w_h(x)=\sum\limits_i w_h^i \phi_i^0(x)$, we observe that
$$
\left|\widetilde{b}_h(w_h)-b(w_h)\right|
= 
\left| \sum\limits_{\cell} \sum\limits_{i=1}^L 
w_h^i \int_\cell f(x) 
\left(\phi_i^{\varepsilon,\cell}(x) - \phi_i^{0,\cell}(x) \right) \, dx \right|.
$$
Using~\eqref{phiexp} and the fact that $\nabla \phi_i^{0,\cell}$ is
constant on $\cell$, we obtain 
\begin{equation}
\sum\limits_{i=1}^L w_h^i \int_\cell f(x)
\left(\phi_i^{\varepsilon,\cell}(x) - \phi_i^{0,\cell}(x) \right)  \, dx 
= 
\sum\limits_{p=1}^d \frac{\varepsilon}{|\cell|} 
\int_\cell f(x) \left(\wper_{e_p}\left(\frac{x}{\varepsilon}\right) -
  \xi_\varepsilon^p(x) \right) \, dx  
\sum\limits_{i=1}^L w_h^i \int_\cell \partial_p \phi_i^{0,\cell}. 
\label{prooflem2-0}
\end{equation}
We have
\begin{equation}
\left| \sum\limits_{i=1}^L w_h^i \int_\cell \partial_p \phi_i^{0,\cell} \right|=\left| \int_\cell \partial_p w_h \right| \leq \sqrt{|\cell|} \ \|w_h\|_{H^1(\cell)} \label{prooflem2-1}
\end{equation}
and
\begin{eqnarray}
\nonumber
\left|\int_\cell f(x)
  \left(\wper_{e_p}\left(\frac{x}{\varepsilon}\right) -
    \xi_\varepsilon^p(x) \right) \,  dx\right| 
&\leq &
\|f\|_{L^2(\cell)}
\left(\left\| \wper_{e_p}\left(\frac{\cdot}{\varepsilon}\right)
  \right\|_{L^2(\cell)} + \|\xi_\varepsilon^p\|_{L^2(\cell)}\right)
\\
&\leq &
\|f\|_{L^2(\cell)} \sqrt{|\cell|}
\left(
\|\wper_{e_p}\|_{L^\infty(\RR^d)}
+ 
\|\xi_\varepsilon^p\|_{L^\infty(\cell)}
\right).
\label{prooflem2-2pre}
\end{eqnarray}
Recall now that, since $A_{per}$ satisfies~\eqref{hyp:a_holder} (i.e. is
H\"older continuous), we know that 
$\xi_\varepsilon^p$ and $\wper_{e_p}$ are both continuous, and that 
$\wper_{e_p} \in L^\infty(\RR^d)$. Using the maximum
principle on~\eqref{PB:xip}, we write 
$$
\|\xi_\varepsilon^p\|_{L^\infty(\cell)} 
\leq 
\|\wper_{e_p}\|_{L^\infty(\partial \mathbf{S})}
\leq 
\|\wper_{e_p}\|_{L^\infty(\RR^d)},
$$
and we thus deduce from~\eqref{prooflem2-2pre} that
\begin{equation}
\left|\int_\cell f(x)
  \left(\wper_{e_p}\left(\frac{x}{\varepsilon}\right) -
    \xi_\varepsilon^p(x) \right) \,  dx\right| 
\leq
C \|f\|_{L^2(\cell)} \sqrt{|\cell|}
\label{prooflem2-2}
\end{equation}
for a constant $C$ independent of $h$ and $\eps$.
Collecting~\eqref{prooflem2-0},~\eqref{prooflem2-1} and~\eqref{prooflem2-2},
we obtain 
$$
\left|\widetilde{b}_h(w_h)-b(w_h) \right| \leq C 
\varepsilon \|f\|_{L^2(\mathcal{D})}\|w_h\|_{H^1(\mathcal{D})}.
$$
This concludes the proof of Lemma~\ref{Lem2}.
\end{proof}

\noindent\begin{proof}[Proof of Lemma~\ref{coro:lambda}]
We first prove the uniform bound on $\lambda$. Recall that the field
$A_1$ is bounded almost surely and almost everywhere. This implies that 
\begin{multline*}
\left|\frac{1}{|I^\varepsilon_\cell|}
\int_{I^\varepsilon_\cell}\left[e_p + \nabla
    \wper_{e_p}\left(\frac{x}{\varepsilon}\right)\right]^T\left(
    A_1\left(\frac{x}{\varepsilon},\omega\right)-\esp\left(A_1\left(\frac{x}{\varepsilon},\cdot\right)\right)\right)\left[e_m + \nabla \wper_{e_m}\left(\frac{x}{\varepsilon}\right)\right] \, dx \right|  
\\
\leq 2 \|A_1\|_{L^\infty} \frac{1}{|I^\varepsilon_\cell|} \left\|e_p +
  \nabla
  \wper_{e_p}\left(\frac{\cdot}{\varepsilon}\right)\right\|_{L^2(I^\varepsilon_\cell)} \ \left\|e_m + \nabla \wper_{e_m}\left(\frac{\cdot}{\varepsilon}\right)\right\|_{L^2(I^\varepsilon_\cell)}.
\end{multline*}
Then, using the $Q$-periodicity of $\wper_{e_p}$, we obtain
$$
\left\|e_p + \nabla 
\wper_{e_p}\left(\frac{\cdot}{\varepsilon}\right)\right\|^2_{L^2(I^\varepsilon_\cell)}
=
\sum\limits_{Q_i^\varepsilon \subset I^\varepsilon_\cell} \int_{Q_i^\varepsilon} \left[ e_p + \nabla \wper_{e_p}\left(\frac{x}{\varepsilon}\right) \right]^2 \, dx
= |I^\varepsilon_\cell| \ \left\|e_p + \nabla
  \wper_{e_p}\right\|^2_{L^2(Q)}. 
$$
We thus have
$$
\lambda(\omega,h,\varepsilon) \leq 2 \|A_1\|_{L^\infty}
\max\limits_{1 \leq p,m \leq d} \left[
\left\| e_p + \nabla \wper_{e_p}\right\|_{L^2(Q)} \ \left\|e_m + \nabla
  \wper_{e_m}\right\|_{L^2(Q)} \right],
$$
hence $\lambda(\omega,h,\varepsilon)$ is bounded almost surely by a
deterministic constant independent of $h$ and $\varepsilon$.

\medskip

We next turn to~\eqref{eq:resu4a}. Rewrite~\eqref{def:sigN} as
$$
\lambda(\omega,h,\varepsilon)
=
\max\limits_\cell \max\limits_{1 \leq m,p \leq d} \left| S^{m,p}_\cell \right|,
$$ 
with
$$
S^{m,p}_\cell := 
\frac{1}{|I^\varepsilon_\cell|} \int_{I^\varepsilon_\cell}\left[e_p +
  \nabla \wper_{e_p}\left(\frac{x}{\varepsilon}\right)\right]^T\left(
  A_1\left(\frac{x}{\varepsilon},\omega\right)-\esp(A_1\left(\frac{x}{\varepsilon},\cdot\right)\right)\left[e_m + \nabla \wper_{e_m}\left(\frac{x}{\varepsilon}\right)\right] \, dx.
$$
Using the periodicity of the correctors $\wper_p$ and the specific
form~(\ref{struc:A1}) of $A_1$, we have
\begin{equation}
\label{eq:lundi2}
S^{m,p}_\cell 
= 
\tau^{m,p} \frac{1}{N_\cell} \sum\limits_{i;Q_i^\varepsilon \subset
  I^\varepsilon_\cell}\frac{X_i - \esp(X_0)}{\sqrt{\var(X_0)}}
\end{equation}
with 
$$
\tau^{m,p} = \sqrt{\var(X_0)} \int_Q  \left[e_p + \nabla \wper_{e_p}\left(y\right)\right]^T B_{per}(y) \left[e_m + \nabla \wper_{e_m}\left(y\right)\right] \, dy
\quad \quad \text{and} \quad \quad
N_\cell = \text{Card}\{i;Q_i^\varepsilon \subset I^\varepsilon_\cell \}.
$$
Thus, $\lambda(\omega,h,\varepsilon)$ reads
$$
\lambda(\omega,h,\varepsilon)
=
\gamma \max\limits_{\cell} \left| S^\varepsilon_\cell(\omega) \right|,
$$
where 
$\displaystyle \gamma = \max\limits_{1 \leq m,p \leq d} \tau^{m,p}$ and
$$
S^\varepsilon_\cell(\omega) = 
\frac{1}{N_\cell} \sum\limits_{i;Q_i^\varepsilon \subset
  I^\varepsilon_\cell}\frac{X_i - \esp(X_0)}{\sqrt{\var(X_0)}}. 
$$
Introduce the probability density function $\varphi_{N_\cell}$ of the
random variable $\sqrt{N_\cell}|S^\varepsilon_\cell(\omega)|$, and
$F_{N_\cell}(x)=\PP\left(\sqrt{N_\cell}|S^\varepsilon_\cell|\leq
  x\right)$. 
Using the assumption that each element $\cell$ contains a number $\dps N_\cell$ of
cells of size $\varepsilon$ that satisfies
$\dps N_\cell \geq \alpha \left( \frac{h}{\varepsilon} \right)^d$ for
some $\alpha>0$, independent of $\cell$, $h$ and $\varepsilon$, we write
$$
\esp\left(\alpha \frac{h^d}{\varepsilon^d}
  \frac{\lambda^2(\cdot,h,\varepsilon)}{\gamma^2} \right)
\leq
\esp\left(N_\cell \max\limits_\cell |S^\varepsilon_\cell|^2 \right)=\int_0^\infty x^2 \frac{d}{dx}\PP\left(\sqrt{N_\cell}\max\limits_\cell|S^\varepsilon_\cell|\leq x\right)dx.
$$
Since
$$
\PP\left(\sqrt{N_\cell}\max\limits_\cell|S^\varepsilon_\cell|\leq
  x\right)=\left[\PP\left(\sqrt{N_\cell}|S^\varepsilon_\cell|\leq
    x\right)\right]^{N(h)}=\left[F_{N_\cell}(x)\right]^{N(h)}, 
$$
we deduce that
\begin{equation}
\esp\left(\alpha \frac{h^d}{\varepsilon^d} 
\frac{\lambda^2(\cdot,h,\varepsilon)}{\gamma^2}\right)
\leq
\int_0^\infty x^2 \, N(h) \, F_{N_\cell}^{N(h)-1}(x) \ \varphi_{N_\cell}(x) \, dx 
=
e_1 + e_2 
\label{eq:cor1}
\end{equation}
where
$$
e_1 = \int_0^{c_h} x^2 N(h) \, F_{N_\cell}^{N(h)-1}(x) \ \varphi_{N_\cell}(x)
\, dx
\quad \text{and} \quad
e_2 = \int_{c_h}^\infty x^2 N(h) \, F_{N_\cell}^{N(h)-1}(x) \
\varphi_{N_\cell}(x) \, dx, 
$$
with $c_h=2\ln(N(h))$. We now successively bound from above $e_1$ and $e_2$. First, integrating by part, and using that $0 \leq
F_{N_\cell} \leq 1$, we obtain
\begin{equation}
\label{eq:cor2}
0 \leq e_1 = \left[ x^2 F_{N_\cell}^{N(h)}(x) \right]_{x=0}^{x=c_h}-\int_0^{c_h}2x F_{N_\cell}^{N(h)}(x) \, dx \leq c_h^2.
\end{equation}
Second, again using $0 \leq F_{N_\cell} \leq 1$, we get
$$
0 \leq e_2 \leq \int_{c_h}^\infty x^2 N(h) \varphi_{N_\cell}(x) \, dx =
N(h) \esp\left(\mathbf{1}_{\left\{\sqrt{N_\cell}|S^\varepsilon_\cell| >
      c_h\right\}}N_\cell|S^\varepsilon_\cell|^2\right). 
$$
Using the Cauchy-Schwartz inequality, we obtain
$$
e_2^2 \leq N(h)^2 \ \esp \left[ N_\cell^2 |S^\varepsilon_\cell|^4
\right] \ 
\PP\left(\sqrt{N_\cell}|S^\varepsilon_\cell| > c_h\right).
$$
Introduce $\displaystyle Y_i=\frac{X_i - \esp(X_0)}{\sqrt{\var(X_0)}}$,
so that $\displaystyle S^\varepsilon_\cell(\omega) = \frac{1}{N_\cell}
\sum\limits_i Y_i(\omega)$. Recall now that $\left(Y_i\right)_{i \in
  \ZZ^d}$ is a sequence of independent identically distributed
variables, with mean zero. Any such variables
satisfy the bounds
$$
\forall p\in \mathbb{N}^\star, \quad \exists C_p>0, \quad \forall N \in
\mathbb{N}^\star, \quad 
\left| \esp\left[ \left( \frac1N \sum_{i=1}^N Y_i\right)^{2p} \right]
\right| 
\leq
\frac{C_p}{N^p}, 
$$
for a constant $C_p$ that depends on $p$ and the moments of $Y_i$, up to
order $2p$. Recall that all moments of $Y_i$ are well defined, as $Y_i$
is bounded almost surely. Thus 
\begin{equation}
e_2^2 
\leq
C_4 N(h)^2 \, \PP\left(\sqrt{N_\cell}|S^\varepsilon_\cell| > c_h\right) 
\leq C_4 N(h)^2 \,\left[\PP\left(\sqrt{N_\cell}S^\varepsilon_\cell >
    c_h\right)+\PP\left(-\sqrt{N_\cell}S^\varepsilon_\cell > c_h\right)
\right]. 
\label{eq:cor7}
\end{equation}
We now recall the Markov inequality: for any positive non-decreasing
function $\psi$ on $\RR$, and any real-valued random variable $Z$, we
have 
$$
\forall b\in\RR,\quad \PP(Z \geq b) \leq \frac{\esp(\psi(Z))}{\psi(b)}.
$$
We apply this inequality to the random variable
$Z(\omega)=\sqrt{N_\cell}S^\varepsilon_\cell(\omega)$, with $\psi=\exp(t
\cdot)$ for some $t \geq 0$, and $b=c_h$. This yields 
\begin{equation}
 \PP(\sqrt{N_\cell}S^\varepsilon_\cell \geq c_h) 
\leq
e^{-t c_h}
\esp\left[\exp\left(t\sqrt{N_\cell}S^\varepsilon_\cell\right)\right] 
\leq
e^{-t c_h}
\left[\esp\left(\exp\left(\frac{t}{\sqrt{N_\cell}}Y_0\right)\right)\right]^{N_\cell}, 
\label{eq:cor4}
\end{equation}
where we have used the fact that $S^\varepsilon_\cell$ is a sum of i.i.d. variables. Using a Taylor expansion with respect to $t$, we see that
$$
\esp\left[\exp\left(\frac{t}{\sqrt{N_\cell}}Y_0\right)\right]=1+\frac{t^2}{2N_\cell}\esp(Y_0^2)+\frac{1}{6 N_\cell^{3/2}}\esp\left[Y_0^3\exp(\xi Y_0 / \sqrt{N_\cell})\right] \quad \text{for some }\xi\in(0,t).
$$
Thus
$$
 \left[\esp\left(\exp\left(\frac{t}{\sqrt{N_\cell}}Y_0\right)\right)\right]^{N_\cell}\leq \exp\left[\frac{t^2}{2}\esp(Y_0^2)+\frac{1}{6 \sqrt{N_\cell}}\esp\left(Y_0^3\exp(\xi Y_0 / \sqrt{N_\cell})\right)\right].
$$
Using~\eqref{eq:cor4}, taking $t=1$, and using that $\displaystyle e^{-c_h}=\frac{1}{N(h)^2}$, we obtain
\begin{eqnarray}
 \PP(\sqrt{N_\cell}S^\varepsilon_\cell \geq c_h) &\leq& \frac{1}{N(h)^2}
 \exp\left[\frac{1}{2}\esp(Y_0^2)+\frac{1}{6
     \sqrt{N_\cell}}\esp\left(Y_0^3\exp(\xi Y_0 /
     \sqrt{N_\cell})\right)\right] \quad \text{for some }\xi \in (0,1), 
\nonumber \\
&\leq& \frac{1}{N(h)^2} \exp\left[\frac{1}{2}\esp(Y_0^2)+\frac{1}{6}\esp\left(|Y_0|^3\exp(|Y_0|)\right)\right]. \label{eq:cor5}
\end{eqnarray}
Likewise, considering
$Z(\omega)=-\sqrt{N_\cell}S^\varepsilon_\cell(\omega)$, we obtain a
similar bound. 
Collecting~\eqref{eq:cor7},~\eqref{eq:cor5} and the fact that $Y_0$ is bounded almost surely, we have
\begin{equation}
\label{eq:cor3}
e_2^2 \leq C,
\end{equation}
with $C$ independent of $h$ and $\varepsilon$. Collecting~\eqref{eq:cor1},~\eqref{eq:cor2} and~\eqref{eq:cor3}, we get, for a constant $C$ independent of $h$ and $\varepsilon$,
$$
\esp(\lambda(\cdot,h,\varepsilon)^2) \leq C \frac{\varepsilon^d}{h^d} \left[\ln(N(h))\right]^2.
$$
This concludes the proof of Lemma~\ref{coro:lambda}.
\end{proof}

\begin{remark}
\label{rem:mittal}
The above proof shows that, when $\eps \to 0$, the random variable
$\dps \left( \frac{h}{\eps} \right)^{d/2} \lambda(\omega,h,\eps)$ converges
in law to $\dps {\cal G}_h(\omega) = \max_\cell \left| G_\cell(\omega) \right|$,
where $G_\cell(\omega)$ are i.i.d. Gaussian random variables. Precise
results on the behavior of ${\cal G}_h(\omega)$ when $h \to 0$ (i.e.,
when the number of Gaussian random variables involved diverges) are
given in e.g.~\cite{Mittal1974}.
\end{remark}

\bigskip

\noindent
\textbf{Acknowledgements:} Support from EOARD under Grant
FA8655-10-C-4002 is gratefully acknowledged.

\end{document}